\documentclass[reqno, 11pt, a4paper]{amsart}
\usepackage{amsmath, amssymb, enumerate, esint, tikz}
\usepackage{bbm}
\usepackage{bm}
	
\usepackage[mathscr]{eucal}
\usepackage[top=3cm, bottom=3cm, left=2.5cm, right=2.5cm]{geometry}

\usepackage{multirow}
\usepackage{longtable}
\usepackage{booktabs}
\usepackage{tablefootnote}

\newcommand{\overbar}[1]{\mkern 1.7mu\overline{\mkern-1.7mu#1\mkern-1.7mu}\mkern 1.7mu}
\makeatletter
\newcommand*{\textlabel}[2]{%
	\edef\@currentlabel{#1}
	\phantomsection
	#1\label{#2}
}
\makeatother

\allowdisplaybreaks

\usepackage{comment}

\usepackage[colorlinks=true,linkcolor=blue, citecolor=purple]{hyperref}

\usepackage[nameinlink,capitalize]{cleveref}
\crefname{equation}{}{}
\crefname{subsection}{subsection}{Subsections}
\crefname{theo}{Theorem}{Theorems}
\crefname{coro}{Corollary}{Corollaries}
\crefname{prop}{Proposition}{Propositions}
\crefname{exam}{Example}{Examples}
\crefname{assum}{Assumption}{Assumptions}

\makeatletter
\newcommand{\refcheckize}[1]{%
	\expandafter\let\csname @@\string#1\endcsname#1%
	\expandafter\DeclareRobustCommand\csname relax\string#1\endcsname[1]{%
		\csname @@\string#1\endcsname{##1}\wrtusdrf{##1}}%
	\expandafter\let\expandafter#1\csname relax\string#1\endcsname
}
\makeatother

\newtheorem{theo}{Theorem}[section]
\newtheorem{coro}[theo]{Corollary}
\newtheorem{lemm}[theo]{Lemma}
\newtheorem{prop}[theo]{Proposition}

\theoremstyle{definition}
\newtheorem{defi}[theo]{Definition}
\newtheorem{assum}[theo]{Assumption}
\newtheorem{exam}[theo]{Example}
\newtheorem{rema}[theo]{Remark}

\numberwithin{equation}{section}

\newcommand{\R}{\mathbb R}
\newcommand{\bN}{\mathbb N}
\newcommand{\bQ}{\mathbb Q}

\newcommand{\bD}{\mathbb D}
\newcommand{\E}{\mathbb E}
\newcommand{\p}{\mathbb P}
\newcommand{\bF}{\mathbb F}
\newcommand{\1}{\mathbbm{1}}
\newcommand{\m}{\mathbbm{m}}
\newcommand{\bI}{\mathbb I}

\newcommand{\sS}{\mathscr S}
\newcommand{\sUS}{\mathscr{US}}

\newcommand{\bfA}{\mathbf A}

\newcommand{\cB}{\mathcal B}
\newcommand{\cD}{\mathcal D}

\newcommand{\cP}{\mathcal P}
\newcommand{\cE}{\mathcal E}

\newcommand{\cM}{\mathcal M}
\newcommand{\cN}{\mathcal N}
\newcommand{\G}{\mathcal G}
\newcommand{\F}{\mathcal F}

\newcommand{\cS}{\mathcal S}
\newcommand{\cT}{\mathcal T}

\newcommand{\rmI}{\mathrm{I}}
\newcommand{\rmD}{\mathrm{D}}
\newcommand{\rmC}{\mathrm{C}}
\newcommand{\rmS}{\mathrm{S}}

\newcommand{\whP}{\ol{\mathbb P}}

\newcommand{\wt}{\widetilde}
\newcommand{\wh}{\widehat}
\newcommand{\ol}{\overbar}

\newcommand{\lee}{\leqslant}
\newcommand{\gee}{\geqslant}

\newcommand{\ep}{\varepsilon}
\newcommand{\e}{\mathrm{e}}
\newcommand{\mmm}{\mathbb P^*}
\newcommand{\mme}{\mathbb E^*}
\newcommand{\supp}{\mathrm{supp}\,}
\newcommand{\Leb}{\lambda}
\newcommand{\BMO}{\mathrm{BMO}}
\newcommand{\bmo}{\mathrm{bmo}}

\newcommand{\mar}{\mathrm{m}}
\newcommand{\fv}{\mathrm{fv}}
\newcommand{\corr}{\mathrm{corr}}
\newcommand{\riem}{\mathrm{Rm}}
\newcommand{\pd}{\partial}
\newcommand{\CL}{\mathrm{CL}}
\newcommand{\cSM}{\mathcal{SM}}
\newcommand{\RH}{\mathcal{RH}}

\newcommand{\im}{\mathrm{i}}
\newcommand{\od}{\mathrm{d}}

\newcommand{\ttt}{\tilde\vartheta}
\newcommand{\ts}{\textstyle}

\newcommand{\sn}{\kappa(\sigma, \nu)}

\def\({\left(}
\def\){\right)}
\def\[{\hspace{-.05cm}\left[}
\def\]{\right]}

\def\<{\langle}
\def\>{\rangle}

\newcommand{\comb}[2]{#1 \sqcup #2 }

\newcommand{\ce}[2]{\E_{#1}\hspace{-.05cm}\left [ #2 \right ]}
\newcommand{\bce}[2]{\E_{#1} \hspace{-.05cm}\big[ #2 \big]}

\newcommand{\hce}[2]{\E^{\whP}_{#1}\hspace{-.05cm}\left [ #2 \right ]}
\newcommand{\hbce}[2]{\E^{\whP}_{#1} \hspace{-.05cm}\big[ #2 \big]}

\newcommand{\hbbce}[2]{\E^{\whP}_{#1} \hspace{-.05cm}\bigg[ #2 \bigg]}

\newcommand{\cem}[2]{ \E^{*}_{#1}\hspace{-.1cm}\left [ #2 \right ]}
\newcommand{\bcem}[2]{ \E^{*}_{#1} \hspace{-.05cm}\big[ #2 \big]}

\begin{document}

\title[Explicit F\"ollmer--Schweizer decomposition]{Explicit F\"ollmer--Schweizer decomposition and discretization with jump correction in exponential L\'evy models}

\author{Nguyen Tran Thuan}
\address{$^{1}$Department of Mathematics, Saarland University,  Postfach 15~11~50, 66041 Saarbr\"ucken, Germany}
\email{nguyen@math.uni-sb.de, thuan.tr.nguyen@gmail.com}

\address{$^{2}$Department of Mathematics, Vinh University, 182 Le Duan, Vinh, Nghe An, Vietnam}

\thanks{The author was supported by the Project 298641 ``\textit{Stochastic Analysis and Nonlinear Partial Differential Equations, Interactions and Applications}'' of the Academy of Finland.}
\date{\today}


\begin{abstract}
	We investigate two hedging problems in exponential L\'evy models. First, we provide
	an explicit representation for the F\"ollmer--Schweizer
	decomposition of European type options under mild conditions, which implies a closed-form
	expression of the corresponding local risk-minimizing strategies. Secondly,
	we discretize stochastic integrals driven by an exponential L\'evy process using a jump correction method. The convergence rate of the resulting discretization error as the expected number of discretization times increases is
	measured in weighted BMO spaces, implying also $L_p$-estimates, $p \in (2, \infty)$.
	Moreover,  the effect of a change of measure satisfying a reverse
	H\"older inequality is addressed. As an application, the error caused by
	discretizing the local risk-minimizing strategies is investigated
	in dependence of properties of the L\'evy measure, the regularity of the
	payoff function and the chosen random discretization times.
\end{abstract}


\maketitle


\section{Introduction}

 This article is concerned with hedging problems in semimartingale financial markets driven by exponential L\'evy processes. We investigate two problems corresponding to two typical types of risks for hedging an option.
 The first one comes from the incompleteness of the market. We consider the semimartingale setting and aim to determine an explicit form for the F\"ollmer--Schweizer decomposition of European type options which provides directly a \textit{closed form} for the local risk-minimizing strategies (a similar closed form expression in the martingale setting has been established in \cite{CTV05, JMP00, Ta10, Th20}).
 The second type of risk is due to the impossibility of continuously rebalancing a hedging portfolio which leads to the discrete-time hedging. We use an approximation scheme based on tracking jumps of the driving process, the so-called \textit{discretization with jump correction}, and measure the discretization error in weighted \textit{bounded mean oscillation} (BMO) spaces. This approach enables to achieve good distributional tail estimates for the error such as a $p$th-order polynomial decay, $p \in (2, \infty)$. 
 
 Let us introduce some notations to state the main results. Let $T \in (0, \infty)$ be a fixed time horizon and  $X = (X_t)_{t\in [0, T]}$ a L\'evy process defined on a complete filtered probability space $(\Omega, \F, \p, \bF)$, where  $\bF = (\F_t)_{t\in [0, T]}$ is the augmented natural filtration of $X$ which satisfies the usual conditions (right continuity and completeness). Assume that $\F = \F_T$. Let $\sigma\gee 0$ be the coefficient of the standard Brownian component and $\nu$ the L\'evy measure of $X$, see \eqref{charac-exponent}. We assume that the underlying discounted price process is modelled by the exponential $S = \e^X$.
 
   \subsection{Explicit F\"ollmer--Schweizer (FS) decomposition}
 Because models with jumps typically correspond to incomplete markets, in general there is no hedging strategy which is self-financing and replicates an option at maturity. Hence, one has to look for certain strategies that minimize some types of risk. In the current work, we choose the quadratic hedging approach which is a popular method to deal with the problem in models with jumps. We refer the reader to the survey article \cite{Sc01} for this approach. Two typical types of quadratic hedging strategies are the \textit{local risk-minimizing} (LRM) strategies and the \textit{mean-variance hedging} (MVH) strategies. Roughly speaking, the LRM strategy is mean-self-financing, replicates an option at maturity and minimizes the riskiness of the cost process locally in time, while the MVH strategy is self-financing and minimizes the \textit{global} hedging error in the mean square sense. Both types of those strategies are intimately related to the so-called \textit{FS decomposition}. Namely, in our (exponential L\'evy) setting, the FS decomposition gives directly the LRM strategy, and the MVH strategy then can be determined  based on this decomposition. This article discusses the FS decomposition and focuses on the LRM strategies only.

	 Assume that $S$ is square integrable so that it is a semimartingale satisfying the \textit{structure condition}, and that the \textit{mean-variance trade-off process} of $S$ is deterministic and bounded, see \cref{rema:SC-condition}. Then, the FS decomposition of an $H \in L_2(\p)$ is of the form
 \begin{align}\label{eq:intro:FS-decom}
 H = H_0 + \int_0^T \vartheta^H_t \od S_t + L^H_T,
 \end{align}
 where $H_0 \in \R$, $\vartheta^H$ is an admissible integrand specified in \eqref{eq:adm-strategies}, and $L^H$ is an $L_2(\p)$-martingale starting at zero which is orthogonal to the martingale part of $S$. The integrand $\vartheta^H$ is called the LRM strategy of $H$, and it is unique up to a $\p \otimes \Leb$-null set where $\Leb$ is the Lebesgue measure. A key tool to study the FS decomposition is the \textit{minimal (signed) local martingale measure} for $S$ (see \cite{Sc95}), and we denote this signed measure by $\mmm$ from now on. Recently, \cite[Theorem 4.3]{CVV10} indicated that under a regularity condition for $\mmm$, we can determine the LRM strategy $\vartheta^H$ based on the martingale representation of $H$ with respect to $\mmm$.

 There are many works interested in finding an explicit representation for the FS decomposition and the LRM strategy in the semimartingale framework (see, e.g., \cite{AS15, GOR14, HKK06, KP10, Ta10}). In the  exponential L\'evy setting and in the case of a European type option $H = g(S_T)$, Hubalek et al. \cite{HKK06} assumed that the function $g$ can be represented as an integral transform of a finite complex  measure from which one can determine a closed form for the LRM strategy. The key idea of this approach is the separation of the function $g$ and the underlying price process $S$ by using a kind of inverse Fourier transform. An advantage of this method is that one gains much flexibility for choosing the underlying L\'evy process where there is no extra regularity required for the driving process $S$ except some mild integrability.

 As our first main result,  \cref{theo:LRM-strategy} below provides a closed form for the LRM strategy $\vartheta^H$ of a European type option $H = g(S_T)$. To obtain this result, except of some mild integrability conditions, we neither assume any regularity for the payoff function $g$ nor require any extra condition for the small jump behavior of $X$ and nor need the presence of the diffusion component. Instead of those regularities, we require the condition that $\mmm$ exists as a true probability measure (see \cref{assum:condition-MMM}) which leads to a constraint for the characteristics of $X$. This result might be regarded as a counterpart of \cite[Proposition 3.1]{HKK06} in which only the square integrability is required for $S$ while the function $g$ are supposed to be the integral transform of finite complex  measures. The notation $\mme$ below means the expectation with respect to $\mmm$.

  \begin{theo}\label{theo:LRM-strategy} Assume that $X$ is not a.s. deterministic and $S =\e^X$ is square $\p$-integrable. If $\mmm$ is a probability measure (i.e., \cref{assum:condition-MMM} holds), then for any Borel function $g\colon (0, \infty) \to \R$ with $\mme |g(y S_t)| <\infty$ $\forall(t, y) \in [0, T] \times (0, \infty)$ and $g(S_T) \in L_2(\p) \cap L_2(\mmm)$ the following assertions hold:
  	\begin{enumerate}[\rm(1)]
  		\itemsep0.3em
  		
  		\item \label{item:LRM-strategy-form} The LRM strategy $\vartheta^H$ corresponding to $H = g(S_T)$ is of the form 
  		\begin{align}\label{eq:LRM-strategy}
  		\vartheta^H_t = \frac{1}{\sn}\(\sigma^2\pd_y G^*(t, S_{t-}) + \int_{\R}\frac{G^*(t, \e^x S_{t-}) - G^*(t, S_{t-})}{S_{t-}} (\e^x -1)\nu(\od x)\)
  		\end{align}
  		for $\p \otimes \Leb$-a.e. $(\omega, t) \in \Omega \times [0, T]$, where $\sn := \sigma^2 + \int_{\R}(\e^x -1)^2\nu(\od x) \in (0, \infty)$,  $G^*(t, y) : = \mme g(y S_{T-t})$, and we set $\pd_y G^* : = 0$ when $\sigma =0$ by convention.
  		
  		\item \label{item:LRM-strategy-cadlag}  There exists a process $\tilde \vartheta^g$ which is adapted and c\`adl\`ag  on $[0, T)$ such that
  		\begin{enumerate}[\rm (a)]
  			\itemsep0.2em
  			\item $\tilde \vartheta^g_{t-}(\omega) = \vartheta^H_t(\omega)$ for $\p \otimes \Leb$-a.e. $(\omega, t) \in \Omega \times [0, T)$;
  			
  			\item $\tilde \vartheta^g S$ is a  $\mmm$-martingale.
  		\end{enumerate}
  	\end{enumerate}
  \end{theo}
 According to \cref{theo:LRM-strategy}\eqref{item:LRM-strategy-cadlag}, $\tilde \vartheta^g_{-}$ is also a LRM strategy of $H=g(S_T)$, and one can determine its values at any time point $t \in [0, T)$ with the aid of \eqref{eq:LRM-strategy} and \cref{rema:unique-cadlag-verion-LRM}. Furthermore, the c\`adl\`ag property of $\tilde\vartheta^g$ is useful to design some Riemann-type approximations  for  $\int_0^{T} \tilde\vartheta^g_{t-} \od S_t$. For example, an approximation scheme based on  tracking jumps of $\tilde \vartheta^g$  has been constructed in \cite{RT14}. We also employ this c\`adl\`ag version of the LRM strategy for the discrete-time hedging problem in \cref{sec:Discrete-time hedging}. Such a path regularity  for integrands in the martingale setting was studied in \cite{MPZ01}.

Several formulas resembling  \eqref{eq:LRM-strategy} have been established in \cite[Formula (2.12)]{JMP00}, \cite[Formula (4.1)]{CTV05}, \cite[Formula (45)]{Ta10}, or in \cite[Formula (4.2)]{Th20}. However, in fact, \eqref{eq:LRM-strategy} is different from those. The formulas in \cite{CTV05, JMP00,  Ta10, Th20} were obtained by projecting $H$ orthogonally down to the space of stochastic integrals driven by a (local) martingale, while \eqref{eq:LRM-strategy} is derived from the FS decomposition which is a different orthogonal decomposition in the semimartingale framework. 

\cref{theo:LRM-strategy} is proved in \cref{sec:LRM-strategy}. The main tool for the proof is a martingale representation for functionals of $X_T$ in which the integrands with respect to the Brownian part and the jump part are determined explicitly. Such a martingale representation is established in \cref{thm:martingale-representation} by using Malliavin calculus.

\subsection{Discretization with jump correction method via the weighted BMO-approach}
We continue to investigate the discrete-time approximation problem for stochastic integrals driving by the exponential L\'evy process $S$. Let $E= (E_t)_{t \in [0, T]}$ be the error process given  by
\begin{align*}
E_t: = \int_0^t \tilde\vartheta_{u-} \od S_u -A_t, \quad t\in [0, T],
\end{align*}
where $\tilde\vartheta$ is an admissible integrand and $A = (A_t)_{t\in [0, T]}$ is a discretization approximation for the stochastic integral. In mathematical finance, the stochastic integral can be interpreted as the theoretical hedging portfolio which is continuously readjusted. However, in practice one can only rebalance the portfolio finitely many times, and this fact leads to a discretization of the stochastic integral, represented by $A$.

In the case that $A = A^{\riem}$ is the Riemann approximation, the corresponding error $E = E^{\riem}$  and its convergence rate have been investigated in the $L_2$-sense in several works. When $S$ is assumed to be a martingale, the discretization error along \textit{deterministic} time-nets was examined by, among others,  Brod\'en and Tankov \cite{BT11} and Geiss, Geiss and Laukkarinen \cite{GGL13}. Although the approaches in \cite{BT11} and in \cite{GGL13} are different, both arrived at a result stating that, if the stochastic integral is sufficiently regular, then the convergence rate of the error measured in $L_2$ is of order $n^{-\frac{1}{2}}$ when $n$, which represents the cardinality of used time-nets,  tends to infinity. This obtained rate is shown to be \textit{asymptotically optimal} in their settings. Later, Rosenbaum and Tankov in \cite{RT14} considered a model where the driving process is a purely discontinuous local martingale. The authors shown that if the small jump activity of the semimartingale integrand behaves like an $\alpha$-stable process with $\alpha \in (1, 2)$ and if the employed discretization times-nets are (random) hitting times of a suitably chosen space grid, then one can achieve the convergence rate $n^{-\frac{1}{\alpha}}$, which is significantly better than $n^{-\frac{1}{2}}$ above. Moreover, this rate was obtained under the $L_2$-norm and was proved to be \textit{asymptotically optimal} in that framework.

In fact, the martingale property of the driving process is technically convenient for handling the discretization errors in $L_2$ and it also simplifies the selection of quadratic hedging strategies.  However, the problem of discrete-time hedging is not necessarily considered under the martingale setting. Therefore, the second part of this article is twofold: We examine the discretization errors under the setting of semimartingale with jumps and investigate the convergence rates for the error in $L_p$ for $p \in (2, \infty)$. Those are natural extensions for the martingale- and for the $L_2$-setting and they seem to be still missing in the literature. 

To address those goals, we use the weighted BMO-approach which has been recently exploited in \cite{Th20}. In addition, we employ the discretization method introduced in \cite{Th20}, the so-called \textit{jump correction method} which was constructed by tracking jumps of the driving process $S$. Moreover, we show  how the weighted $\BMO$-approach can   be used to obtain $L_p$-estimates, $p \in (2, \infty)$, for the corresponding error. This approach also allows a change of the underlying measure which leaves the error estimates unchanged provided the change of measure satisfies a reverse H\"older inequality, see \cref{prop:feature-BMO}.    The latter is frequently encountered in mathematical finance, and it is particularly useful here to switch the discretization problem between the martingale setting  and the semimartingale setting.

The main results of this part are \cref{thm:approximation-BMO,thm:approximation-convergence-rate}. We provide in \cref{thm:approximation-BMO} several estimates for the discretization error measured under weighted $\BMO$-norms and describe a situation so that $L_p$-estimates can be achieved for $p \in (2, \infty)$. \cref{thm:approximation-convergence-rate} serves as an application of \cref{thm:approximation-BMO} where we consider the  approximation problem for the stochastic integral term in \eqref{eq:intro:FS-decom} and the chosen integrand is the LRM strategy of a European type option. The results show how the interplay between the regularity of payoff functions and the small jumps intensity of the underlying L\'evy process affects the convergence rate.

Let us list some illustrative examples, which are direct consequences of \cref{thm:approximation-convergence-rate}, showing the convergence rates for $E^{\corr}_g$. Here, $E^{\corr}_g$ denotes the global discretization error resulted from the approximation with jump correction method using the LRM strategy of a European payoff $H = g(S_T)$. To reveal the effect of the small jump activity of $X$, we assume that $X$ does not have a Brownian component and the small jump intensity of $X$ behaves like an $\alpha$-stable process with $\alpha \in (0, 2)$. Then, for the European call/put option (or any Lipschitz $g = g^{\mathrm{Lip}}$) and for the binary option (or any bounded and Borel function $g = g^{\mathrm{bdd}}$), under $\E S_T^p <\infty$ with $p >3$ and some mild  conditions we have
\begin{align*}
\|E^{\corr}_{g^{\mathrm{Lip}}}\|_{L_p} \lee c \begin{cases}
1/n & \mbox{if } \alpha \in (0, 1)\\
(1+ \log n)/n & \mbox{if } \alpha =1\\
1/n^{\frac{1}{\alpha}} & \mbox{if } \alpha \in (1, 2)
\end{cases},\quad \|E^{\corr}_{g^{\mathrm{bdd}}}\|_{L_p} \lee c \begin{cases}
1/n & \mbox{if } \alpha \in (0, 1)\\
(1+ \log n)/n & \mbox{if } \alpha =1\\
1/n^{\frac{1}{\alpha}[1- \frac{1}{\alpha}(\alpha -1)^2]-\delta} & \mbox{if } \alpha \in (1, 2),
\end{cases}
\end{align*}
where the latter case holds for any $0 < \delta < \frac{1}{2}(1- \frac{1}{\alpha})(\frac{2}{\alpha}-1)$. The parameter $n$ in those estimates represents the expected cardinality of the used time-nets. Our results show the convergence rate $n^{-\frac{1}{\alpha}}$ not only for Lipschitz $g^{\mathrm{Lip}}$ but also for certain H\"older functionals $g$. We remark here that the obtained rate $n^{-\frac{1}{\alpha}}$ is consistent with that in \cite[Remark 5 with $\beta = 0$]{RT14}, and moreover, our approach enables the case $\alpha \in (0, 1]$ and the $L_p$-setting which are not treated there.

\subsection{Structure of the article} In \cref{sec:preliminaries}, we introduce the notation, recall Malliavin--Sobolev spaces and exponential L\'evy processes. \cref{sec:martingale-representation} aims to establish a martingale representation with explicit integrands for functionals of a L\'evy process. \cref{sec:LRM-strategy} is devoted to prove \cref{theo:LRM-strategy} above. \cref{sec:Discrete-time hedging} presents the discrete-time hedging problem with the weighted $\BMO$-approach for exponential L\'evy models. \cref{appen:proof:thm:approximation-BMO} provides the proof for main results of \cref{sec:Discrete-time hedging}, and some technical results are given in \cref{appen:technical-results}.

\section{Preliminaries} \label{sec:preliminaries}

\subsection{General notations}
Denote $\R_0 := \R \backslash \{0\}$. For $a, b \in \R$, we set $a\vee b := \max\{a, b\}$ and $a\wedge b : = \min\{a, b\}$ as usual. For $A, B \gee 0$ and $c\gee 1$, by $A\sim_c B$ we mean $A/c \lee B \lee cA$.  Subindexing a symbol by a label indicates the place where that symbol appears (e.g., $c_{\eqref{eq:RH-inequality}}$ refers to formula \eqref{eq:RH-inequality}). 

Let $\cB(\R)$ be the Borel $\sigma$-algebra on $\R$. The Lebesgue measure on $\cB(\R)$ is denoted by $\Leb$, and we also write $\od x$ instead of $\Leb(\od x)$ for simplicity. For $p \in [1, \infty]$ and $A \in \cB(\R)$, the space $L_p(A)$ consists of all $p$-order integrable Borel functions on $A$ with respect to $\Leb$, where the essential supremum is taken when $p=\infty$. For a measure $\mu$ on $\cB(\R)$, its support is given by 
\begin{align*}
\supp\mu : = \{x \in \R : \mu((x -\ep, x+\ep)) >0 ,\forall \ep>0\}.
\end{align*}

Let $(\Omega, \F, \p)$ be a probability space and $\xi \colon \Omega \to \R$ a random variable. Denote by $\p_\xi$ the push-forward measure of $\p$ with respect to $\xi$. If $\xi$ is $\p$-integrable (non-negative), then the (generalized) conditional expectation of $\xi$ given a sub-$\sigma$-algebra $\G\subseteq \F$ is denoted by  $\E^{\p}_{\G}[\xi]$ for which we usually omit the reference measure $\p$ if there is no risk of confusion. Set $L_p(\p) : = L_p(\Omega, \F, \p)$.

For a non-empty and open interval $U \subseteq \R$, let $C^\infty(U)$ denote the family of all functions $f$ which have derivatives of all orders on $U$.

\subsection{Notation for stochastic processes} 
Let $T >0$ be a fixed finite time horizon,  and let $(\Omega, \F, \p)$ be a complete probability space equipped with a right continuous filtration $\bF = (\F_t)_{t \in [0, T]}$. Assume that $\F_0$ is generated by $\p$-null sets only. The conditions imposed on $\bF$ allow us to assume that every martingale adapted to this filtration is \textit{c\`adl\`ag} (right-continuous with left limits).  We use the following notations and conventions where
 \begin{align*}
 \bI = [0, T] \quad \mbox{or} \quad \bI = [0, T).
 \end{align*}
 \begin{itemize}
 	\itemsep0.3em
 	\item For processes $X = (X_t)_{t \in \bI}$ and $Y = (Y_t)_{t \in \bI}$, we write $X = Y$ to indicate that $X_t = Y_t$ for all $t \in \bI$ a.s., and analogously when the relation ``='' is replaced by some other standard relations such as ``$\lee$'', ``$\gee$'', etc.
 	
 	\item For a c\`adl\`ag process $X = (X_t)_{t\in \bI}$, the process $X_- = (X_{t-})_{t\in \bI}$ is defined by setting $X_{0-} : = X_0$ and $X_{t-} : = \lim_{0< s \uparrow t} X_{s}$ for $t \in \bI \backslash\{0\}$. We set $\Delta X: = X - X_{-}$.
 	
 	\item $\CL(\bI)$ denotes the family of all c\`adl\`ag and $\bF$-adapted processes.
 	
 	\item $\CL_0(\bI)$ (resp. $\CL^+(\bI)$) consists of all $X \in \CL(\bI)$ with $X_0 = 0$ a.s. (resp. $X\gee 0$). 
 	
 	\item For $p \in [1, \infty]$ and $X \in \CL([0, T])$, we set $\|X\|_{S_p(\p)}: = \|\sup_{t \in [0, T]}|X_t|\|_{L_p(\p)}$.
 	
 	\item $\cP$ is the predictable $\sigma$-algebra\footnote{$\cP$ is the $\sigma$-algebra generated by $\{A \times \{0\} : A \in \F_0\} \cup \{A \times (s, t] : 0\lee s < t \lee T, A \in \F_s\}$.} on $\Omega \times [0, T]$ and $\wt \cP : = \cP \otimes \cB(\R)$.
 \end{itemize}
 	
 	We recall some notions regarding semimartingales on the finite time interval $[0, T]$.
 	
 	\begin{itemize}
 		\itemsep0.3em
 	\item An $M\in \CL([0, T])$ is called a local (resp. locally square integrable) $\p$-martingale if there is a sequence of non-decreasing stopping times $(\rho_n)_{n\gee 1}$ taking values in $[0, T]$ such that $\p(\rho_n <T) \to 0$ as $n\to \infty$ and the stopped process $M^{\rho_n} = (M_{t\wedge \rho_n})_{t\in [0, T]}$ is a  $\p$-martingale (resp. square integrable $\p$-martingale) for all $n\gee 1$. Let $\cM_2^0(\p)$ be the space of all square integrable $\p$-martingales $M = (M_t)_{t\in [0, T]}$ with $M_0 =0$ a.s. 
 	
 	\item An $S \in \CL([0, T])$ is called a $\p$-semimartingale if $S$ can be written as a sum of a local $\p$-martingale and a process of finite variation a.s. The \textit{quadratic covariation} of two semimartingales $S$ and $R$ is denoted by $[S, R]$.  The predictable $\p$-compensator of $[S, R]$, if it exists, is denoted by $\<S, R\>^\p$.
 	
 	\item Let $M$, $N$ be  locally square integrable $\p$-martingales. Then, $M$ and $N$ are said to be \textit{$\p$-orthogonal} if $[M, N]$ is a local $\p$-martingale, or equivalently, $\<M, N\>^{\p} =0$.
 \end{itemize}

\subsection{L\'evy process and It\^o's chaos expansion} \label{subsection-Levy-process} Let $X=(X_t)_{t \in [0, T]}$ be a real-valued L\'evy process on a complete probability space $(\Omega, \F, \p)$, i.e. $X_0 =0$, $X$ has independent and stationary increments and $X$ has c\`adl\`ag paths. Let $\bF = (\F_t)_{t\in [0, T]}$ denote the augmented natural filtration generated by $X$. Throughout this article, we assume $\F = \F_T$. According to the L\'evy--Khintchine formula (see, e.g., \cite[Theorem 8.1]{Sa13}),  the \textit{characteristic exponent $\psi$} of $X$, which is defined by
$$\E \e^{\im u X_t} = \e^{- t \psi(u)}, \quad u \in \R, t\in [0, T],$$
 is of the form 
\begin{align}\label{charac-exponent}
\psi(u)=- \im \gamma u + \frac{\sigma^2 u^2}{2} - \int_{\R}\(\e^{\im u x}-1-\im ux \1_{\{|x| \lee 1\}}\)\nu(\od x), \quad u \in \R.
\end{align} 
Here, $\gamma \in \R$, while $\sigma \gee 0$ is the coefficient of the Brownian component, and $\nu\colon \cB(\R) \to [0, \infty]$ is a L\'evy measure, i.e., $\nu(\{0\}):=0$ and $\int_{ \R}(x^2 \wedge 1) \nu(\od x) <\infty$. The triplet $(\gamma, \sigma, \nu)$ is called the characteristics of $X$. To indicate explicitly the characteristics of $X$ under $\p$, we write 
\begin{align*}
(X|\p) \sim (\gamma, \sigma, \nu) \quad \mbox{or} \quad (X|\p)\sim \psi.
\end{align*}

We present briefly the Malliavin calculus for L\'evy processes by means of It\^o's chaos expansion which is the main tool to prove \cref{thm:martingale-representation}. For further details, we refer to \cite{Ap09, NN18, NOP09, SUV07b} and the references therein.  Define the $\sigma$-finite measures $\mu$ on $\cB(\R)$ and $\m$ on $\cB([0, T] \times \R)$ by setting
\begin{align*}
\mu(\od x) := \sigma^2 \delta_0(\od x) + x^2 \nu(\od x)\quad \mbox{and} \quad \m := \Leb \otimes \mu,
\end{align*}
where $\delta_0$ is the Dirac measure at zero. For $B\in \cB([0, T] \times \R)$ with $\m(B)<\infty$, the random measure $M$ is defined by
\begin{align*}
M(B): = \sigma \int_{\{t\in [0, T] : (t, 0) \in B\}} \od W_t + L_2(\p)\mbox{-}\lim_{ n \to \infty}\int_{B \cap ([0, T] \times \{\frac{1}{n} <|x| < n\})} x \wt N(\od t, \od x),
\end{align*} 
where $W$ is the standard Brownian motion and $\wt N(\od t, \od x): = N(\od t, \od x) - \od t \nu(\od x)$ is the compensated Poisson random measure appearing in the L\'evy--It\^o decomposition of $X$ (see, e.g., \cite[Theorem 2.4.16]{Ap09}). 
Set $L_2(\mu^{0}) = L_2(\m^{0}) : = \R$. For $n \gee 1$, we denote 
\begin{align*}
L_2(\mu^{\otimes n}) &: = L_2(\R^n, \cB(\R^n), \mu^{\otimes n}),\\ L_2(\m^{\otimes n}) &: = L_2(([0, T] \times \R)^n, \cB(([0, T] \times \R)^n), \m^{\otimes n}).
\end{align*} 
 The multiple integral $I_n \colon L_2(\m^{\otimes n}) \to L_2(\p)$ is defined in the sense of It\^o \cite{It56} by a standard way using approximation, where it is given for simple functions as follows: for $$\xi_n^m: = \sum_{k=1}^m a_k \1_{B_1^k \times \cdots \times B_n^k},$$
  where $a_k \in \R$, $B_i^k \in \cB([0, T] \times \R)$ with $\m(B_i^k)<\infty$ and $B_i^k \cap B_j^k = \emptyset$ for $k = 1,\ldots, m$, $i, j =1, \ldots, n$, $i\neq j$ and $m\gee 1$, we define
  $$I_n(\xi_n^m) : = \sum_{k=1}^m a_k M(B_1^k)\cdots M(B_n^k).$$
   Then, \cite[Theorem 2]{It56} asserts the following It\^o chaos expansion
\begin{align*}
L_2(\p) =  \bigoplus_{n=0}^\infty  \{ I_n(\xi_n) : \xi_n \in L_2(\m^{\otimes n})\},
\end{align*}
where $I_0(\xi_0):= \xi_0 \in \R$. For $n\gee 1$, the symmetrization $\tilde \xi_n$ of a $\xi_n \in L_2(\m^{\otimes n})$ is 
\begin{align*}
\tilde \xi_n((t_1, x_1), \ldots, (t_n, x_n)) := \frac{1}{n!}\sum_{\pi}\xi_n ((t_{\pi(1)}, x_{\pi(1)}), \ldots, (t_{\pi(n)}, x_{\pi(n)})),
\end{align*}
where the sum is taken over all permutations $\pi$ of $\{1, \ldots, n\}$, so that  $I_n(\xi_n) = I_n(\tilde \xi_n)$ a.s.  The It\^o chaos decomposition verifies that $\xi \in L_2(\p)$ if and only if there are $\xi_n \in L_2(\m^{\otimes n})$ such that $\xi =  \sum_{n=0}^\infty I_n(\xi_n)$ a.s., and this  expansion is unique if every $\xi_n$ is symmetric, i.e. $\xi_n = \tilde \xi_n$. 
 Furthermore, $\|\xi\|_{L_2(\p)}^2  = \sum_{n=0}^\infty n! \|\tilde \xi_n\|_{L_2(\m^{\otimes n})}^2$.
 \begin{defi}\label{Malliavin-derivative}
 	Let $\bD_{1,2}$ be the Malliavin--Sobolev space of all $\xi = \sum_{n=0}^\infty I_n(\xi_n) \in L_2(\p)$ such that
 	\begin{align*}
 	\|\xi\|_{\bD_{1,2}}^2 := \sum_{n=0}^\infty (n+1)! \|\tilde \xi_n\|_{L_2(\m^{\otimes n})}^2 <\infty.
 	\end{align*}
 	The \textit{Malliavin derivative operator}  $D \colon \bD_{1,2} \to L_2(\p \otimes \m)$, where $L_2(\p \otimes \m) : = L_2(\Omega \times [0, T] \times \R, \F \otimes \cB([0, T] \times \R), \p \otimes  \m)$, is defined for $\xi = \sum_{n=0}^\infty I_n(\xi_n) \in \bD_{1,2}$ by
 	\begin{align*}
 	D_{t, x} \xi := \sum_{n=1}^\infty n I_{n-1}(\tilde \xi_n((t, x), \cdot)), \quad (\omega, t, x) \in \Omega \times [0, T] \times \R.
 	\end{align*}
 \end{defi}

\subsection{Exponential L\'evy processes} \label{subsec:exponential-levy} Let $X$ be a L\'evy process with $(X|\p) \sim (\gamma, \sigma, \nu)$. The stochastic exponential of $X$, denoted by $\cE(X)$, is the c\`adl\`ag process that satisfies the stochastic differential equation (SDE)
\begin{align*}
\od \cE(X) = \cE(X)_{-}\od X, \quad \cE(X)_0 =1.
\end{align*}
 We apply \cite[Theorem 5.1.6]{Ap09} with the truncation function $x\1_{\{|x| \lee 1\}}$ instead of $x\1_{\{|x| < 1\}}$ to obtain that if $\cE(X)>0$, then there is a L\'evy process $Y$ with $(Y|\p) \sim (\gamma_Y, \sigma_Y, \nu_Y)$ such that $\cE(X) = \e^Y$, where $\sigma_Y = \sigma$ and
 \begin{align*}
 	\nu_Y & = \nu \circ h^{-1} \quad \mbox{for } h(x) := \ln (1 + x),\\
 \gamma_Y & = \gamma - \frac{\sigma^2}{2} + \int_{\R}\(\1_{\{|h(x)|\lee 1\}}h(x) - x \1_{\{|x|\lee 1\}}\)\nu(\od x).
 \end{align*}
Conversely, there is a L\'evy process $Z$ with $(Z|\p) \sim (\gamma_Z, \sigma_Z, \nu_Z)$ such that $\e^X = \cE(Z)$.  Moreover, one has $\sigma_Z = \sigma$ and
\begin{align*}
\nu_Z  &= \nu \circ \tilde h^{-1} \quad \mbox{for } \tilde h(x): = \e^x -1,\\
\gamma_Z & = \gamma + \frac{\sigma^2}{2} + \int_{\R}\(\1_{\{|\tilde h(x)| \lee 1\}} \tilde h(x) -  x\1_{\{|x|\lee 1\}}\)\nu(\od x). 
\end{align*}


\section{Martingale representation with explicit integrands} \label{sec:martingale-representation}

This section is to prove a martingale representation for $f(X_T)$ by using Malliavin calculus. There are two key observations: first, the kernels in the chaos expansion of $f(X_T) \in L_2(\p)$ do not depend on the time variables which implies the Malliavin differentiability of  $\ce{\F_t}{f(X_T)}$ for any $t\in [0, T)$, see \cref{lemm:smooth-conditional-expentation}; secondly, the Malliavin derivative of a functional of $X_t$, provided it is Malliavin differentiable, can be expressed in an explicit form, see \cref{lemm:Malliavin-derivative-functional}.



Assume  $(X|\p) \sim (\gamma, \sigma, \nu)$ in this section. We first need the following:

\begin{lemm}[\cite{GN20}, Theorem 9.13(1)]\label{lemm:sigma>0}
	Assume $\sigma > 0$. Let $f\colon \R \to \R$ be a Borel function with $\E|f(X_T)|^q <\infty$ for some $q>1$. Then, $\E|f(x+X_{T-t})| <\infty$ for all $(t, x) \in [0, T] \times \R$, and the function $x \mapsto F(t, x)  : = \E f(x + X_{T-t})$ belongs to $C^\infty(\R)$ for any $t \in [0, T)$. Furthermore,  
			\begin{align*}
			\ce{\F_s}{\pd_x F(t, X_t)} = \pd_x F(s, X_s) \quad \mbox{a.s.}
			\end{align*}
			for any $0\lee s<t<T$.
\end{lemm}

\cref{lemm:Malliavin-derivative-functional} below was established in \cite[Corollary 3.1 in the second article of this thesis]{La13} and it provides an equivalent condition such that a functional of $X_t$ belongs to $\bD_{1,2}$. We refer in particular to \cite[Proposition V 2.3.1]{MAKL95} when $X$ is a Brownian motion and refer to \cite[Lemma 3.2]{GS16} when $X$ has no Brownian component.

\begin{lemm}[\cite{La13}] \label{lemm:Malliavin-derivative-functional}
	Let $t \in (0, T]$ and a Borel function $f\colon \R \to \R$ with $f(X_t) \in L_2(\p)$. Then, $f(X_t) \in \bD_{1, 2}$ if and only if the following two assertions hold:
	\begin{enumerate}[\quad \rm (a)]

		\item when $\sigma > 0$, $f$ has a weak derivative\footnote{A locally integrable $h$ is called a \textit{weak derivative} of a locally integrable $f$ on $\R$ if $\int_{\R} f(x) \phi'(x) \od x  = - \int_{\R} h(x) \phi(x) \od x$ for all smooth functions $\phi$ with compact support in $\R$. If such an $h$ exists (unique up to a $\Leb$-null set), then we denote $f'_w: = h$.} $f'_w$ on $\R$ with $f'_w(X_t) \in L_2(\p)$;
		
		\item the map $(s, x) \mapsto \frac{f(X_t + x) - f(X_t)}{x}\1_{[0, t] \times \R_0}(s, x)$ belongs to $L_2(\p\otimes \m)$.
	\end{enumerate}
	Furthermore, if $f(X_t) \in \bD_{1,2}$, then for $\p \otimes \m$-a.e. $(\omega, s, x) \in \Omega \times [0, T]\times \R$ one has
	\begin{align*}
	D_{s, x}f(X_t) = f'_w(X_t)\1_{[0, t] \times \{0\}}(s, x) + \frac{f(X_t + x) - f(X_t)}{x}\1_{[0, t] \times \R_0}(s, x),
	\end{align*}
	where we set, by convention, $f'_w:=0$ when $\sigma =0$.
\end{lemm}

\begin{lemm}\label{lemm:smooth-conditional-expentation}
	Let $f\colon \R \to \R$ be a Borel function with  $f(X_T) \in L_2(\p)$.
	\begin{enumerate}[\rm (1)]

		\item \label{item:chaos-expansion-f(XT)} There are symmetric $\tilde f_n \in L_2(\mu^{\otimes n})$ such that $f(X_T) = \sum_{n=0}^\infty I_n(\tilde f_n \1_{[0, T]}^{\otimes n})$ a.s.
		\item \label{item:chaos-expansion-conditional-f(XT)} For $t\in [0, T)$, one has $\ce{\F_t}{f(X_T)} = \sum_{n=0}^\infty I_n(\tilde f_n \1_{[0, t]}^{\otimes n})$ a.s. and $\ce{\F_t}{f(X_T)} \in \bD_{1,2}$.
		\item \label{item:integrability-conditional-expectation} For $t\in (0, T)$, it holds
		\begin{align}\label{eq:integrable-derivative}
		\E \[|\sigma \pd_x F(t, X_t)|^2 + \int_{\R}|F(t, X_t+x) - F(t, X_t)|^2 \nu(\od x)\] < \infty,
		\end{align}
		where  $F(t, x) : = \E f(x + X_{T-t})$ if $\sigma >0$, and in the case $\sigma =0$ we let $F(t, \cdot)$ be a Borel function such that $F(t, X_t) = \ce{\F_t}{f(X_T)}$ a.s. and set $\pd_x F : = 0$.
	\end{enumerate} 
\end{lemm}

\begin{proof} Items \eqref{item:chaos-expansion-f(XT)} and \eqref{item:chaos-expansion-conditional-f(XT)} are due to  \cite[Lemma D.1]{GN20}. For Item \eqref{item:integrability-conditional-expectation}, it is clear for the case $\sigma =0$ that \eqref{eq:integrable-derivative} is implied by \cref{lemm:Malliavin-derivative-functional}. Let us turn to the case $\sigma >0$. According to \cref{lemm:sigma>0}, one has $F(t, \cdot) \in C^\infty(\R)$, and hence $(F(t, \cdot))'_w = \pd_x F(t, \cdot)$ a.e. with respect to the Lebesgue measure $\Leb$. Since the law of $X_t$ is absolutely continuous with respect to $\Leb$, it holds that $(F(t, \cdot))'_w(X_t) = \pd_x F(t, X_t)$ a.s. Then, \eqref{eq:integrable-derivative} follows from \cref{lemm:Malliavin-derivative-functional}.
\end{proof}

The following result provides a martingale representation for $f(X_T)$ with explicit integrands.

\begin{prop}\label{thm:martingale-representation} Let $f\colon \R \to \R$ be a Borel function such that $\E|f(x+ X_t)| <\infty$ for all $(t, x) \in [0, T] \times \R$. If $f(X_T) \in L_2(\p)$, then 
	\begin{align*}
		\E \int_0^T  |\sigma \pd_x F(t, X_{t-})|^2 \od t + \E \int_0^T\!\!\int_{\R}|F(t, X_{t-}+x) - F(t, X_{t-})|^2 \nu(\od x) \od t <\infty
	\end{align*}
	and, a.s.,
	\begin{align}\label{eq:martingale-representation}
		f(X_T) = \E f(X_T)  + \int_0^T \sigma \pd_x F(t, X_{t-}) \od W_t   + \int_0^T\!\!\int_{\R\backslash\{0\}} (F(t, X_{t-} +x) - F(t, X_{t-})) \wt N(\od t, \od x),
	\end{align}
	where  $F(t, x) : = \E f(x+ X_{T-t})$ for $(t, x) \in [0, T] \times \R$, and we set $\pd_x F :=0$ if $\sigma =0$.
\end{prop}

\begin{proof} For $(t, x) \in [0, T] \times \R$, denote
	\begin{align}\label{eq:defi-Delta-F}
	\cD F(t, x) : = \pd_x F(t, X_{t-}) \1_{\{x =0\}} + \frac{F(t, X_{t-} +x ) - F(t, X_{t-})}{x} \1_{\{x \neq 0\}},
	\end{align}
	where we recall that $\pd_x F:=0$ if $\sigma =0$ by  convention. The assumption $\E|f(x+X_t)|<\infty$ for all $(t, x) \in [0, T] \times \R$ implies that $(F(t, X_t + x) - F(t, X_t))_{t\in [0, T]}$ is a martingale for each $x\in \R$. Moreover, in the case $\sigma >0$, the assumption $f(X_T) \in L_2(\p)$ and \cref{lemm:sigma>0} imply that  $F(t, \cdot) \in C^\infty(\R)$ for all $t\in [0, T)$ and $(\pd_x F(t, X_t))_{t\in [0, T)}$ is a martingale.
	
	\smallskip
	
	\textit{Step 1.} We show that for any $t\in (0, T)$,
	\begin{align*}
	C(t): = \E \int_0^t \!\!\int_{\R} |\cD F(s, x)|^2  \m(\od s, \od x) <\infty.
	\end{align*}
	Note that $(t, x) \mapsto F(t, x)$ is Borel measurable by Fubini's theorem. Since $X_-$ is predictable, it implies that $(\omega, t, x) \mapsto F(t, X_{t-}(\omega) +x)$ is $\wt \cP$-measurable. Hence, $\cD F$ given in \eqref{eq:defi-Delta-F} is $\wt \cP$-measurable. Since $X_s = X_{s-}$ a.s. for each $s\in [0, T]$, using Fubini's theorem and the martingale property, together with \eqref{eq:integrable-derivative}, we obtain for any $t\in (0, T)$ that
	\begin{align*}
	C(t)
	& = \E \int_0^t \!\!  |\sigma \pd_x F(s, X_s)|^2 \od s   + \E \int_0^t\!\!\int_{\R} |F(s, X_s +x) - F(s, X_s)|^2 \nu(\od x) \od s\\
	& \lee t\(\E |\sigma \pd_x F(t, X_t)|^2 + \E \int_{\R}|F(t, X_t +x) - F(t, X_t)|^2 \nu(\od x)\)  <\infty.
	\end{align*}
	Hence, the  stochastic integral $\int_0^t\int_{\R} \cD F(s, x) M(\od s, \od x)$ exists as an element in $L_2(\p)$.
	
	\smallskip
	
	\textit{Step 2.} Fix $t\in (0, T)$. We prove that, a.s., 
	\begin{align}\label{eq:Clark-Ocone-F}
	F(t, X_t)  = \E f(X_T) + \int_0^t\!\!\int_{\R} \cD F(s, x) M(\od s, \od x).
	\end{align}
	The representation \eqref{eq:Clark-Ocone-F} can be regarded as a consequence of the Clark--Ocone formula. However, this formula seems to be considered   either when the L\'evy process $X$ is square integrable or when $X$ has no Brownian component (i.e., $\sigma =0$), see, e.g., \cite{BNLOP03, Lo04, NN18, NOP09, SUV07b}. So, for the reader's convenience, we present here a complete proof for \eqref{eq:Clark-Ocone-F} where neither square integrability nor $\sigma =0$ is assumed. Due to the denseness of the simple multiple stochastic integrals in $L_2(\p)$ (see \cite[Lemma 2.1]{GL11}), in order to obtain \eqref{eq:Clark-Ocone-F} it is sufficient to show that
	\begin{align}\label{eq:test-product}
	\E \[I_m(k_m) F(t, X_t)\]  = \E \[ I_m(k_m) \int_0^t\!\!\int_{\R} \cD F(s, x) M(\od s, \od x)\]
	\end{align}
	for all $m\gee 1$ and all functions $k_m$ of the form
	\begin{align}\label{eq:simple-kernel}
	k_m = \1_{B_1 \times \cdots \times B_m},
	\end{align}
	where $B_i = (s_i, t_i] \times (a_i, b_i]$ in which $(a_i, b_i]$ are finite intervals and the time intervals $(s_i, t_i] \subset [0, t]$ satisfy $t_{i-1}\lee s_i$, $i=2,\ldots, m$.

	Since $F(t, X_t) \in \bD_{1, 2}$ by \cref{lemm:smooth-conditional-expentation}\eqref{item:chaos-expansion-conditional-f(XT)}, applying \cref{lemm:Malliavin-derivative-functional} we have for $\p\otimes \m$-a.e. $(\omega, s, x) \in \Omega \times [0, T]\times \R$,
	\begin{align}\label{eq:malliavin-derivative-formula}
	D_{s,x} F(t, X_t) & = \pd_x F(t, X_t)\1_{[0, t]\times\{0\}}(s, x) + \frac{F(t, X_t + x) - F(t, X_t)}{x}\1_{[0, t]\times \R_0}(s, x).
	\end{align}
	Moreover, for each $(s, x)\in [0, t]\times \R$, the martingale property implies that, a.s., 
	\begin{align*}
	& \ce{\F_s}{\pd_x F(t, X_t)\1_{\{x= 0\}}+ \frac{F(t, X_t + x) - F(t, X_t)}{x}\1_{\{x\neq 0\}}} \\
	& = \pd_x F(s, X_s)\1_{\{x=0\}} + \frac{F(s, X_s + x) - F(s, X_s)}{x}\1_{\{x\neq 0\}}\\
	& = \cD F(s, x),
	\end{align*}
	where the second equality comes from the fact that $X_s = X_{s-}$ a.s.
	
	We let $f(X_T) = \sum_{n=0}^\infty I_n(\tilde f_n \1_{[0, T]}^{\otimes n})$ and $F(t, X_t) = \sum_{n=0}^\infty I_n(\tilde f_n \1_{[0, t]}^{\otimes n})$ as in \cref{lemm:smooth-conditional-expentation}\eqref{item:chaos-expansion-f(XT)} and \eqref{item:chaos-expansion-conditional-f(XT)} respectively, where $\tilde f_n \in L_2(\mu^{\otimes n})$ are symmetric. Let $k_m$ be of the form as in \eqref{eq:simple-kernel}. Since functions $\tilde f_n$ are symmetric, the left-hand side of \eqref{eq:test-product} is  computed as follows
	\begin{align}\label{eq:compute-LHS}
	\mathrm{LHS}_{\eqref{eq:test-product}} = m! \int_{B_1\times \cdots \times B_m} \tilde f_m(x_1, \ldots, x_m) \m(\od s_1, \od x_1)\cdots \m(\od s_m, \od x_m). 
	\end{align}
	For the right-hand side of \eqref{eq:test-product}, writing $I_m(k_m) = \int_{B_m} I_{m-1} (k_{m-1}) M(\od s, \od x)$, where $k_{m-1}: = \1_{B_1\times \cdots \times B_{m-1}}$, and using Fubini's theorem we obtain
	\begin{align}
	& \mathrm{RHS}_{\eqref{eq:test-product}}  = \E \int_{B_m} I_{m-1}(k_{m-1}) \cD F(s, x)  \m(\od s, \od x) \notag \\
	& = \int_{B_m} \E \[I_{m-1}(k_{m-1}) \ce{\F_s}{\pd_x F(t, X_t)\1_{\{x= 0\}}+ \frac{F(t, X_t + x) - F(t, X_t)}{x}\1_{\{x\neq 0\}}}\] \m(\od s, \od x)  \notag \\
	& = \E \int_{B_m} I_{m-1}(k_{m-1}) D_{s,x} F(t, X_t) \m(\od s, \od x) \label{eq:F-s-measurability}\\
	& =  \E \int_{B_m} I_{m-1}(k_{m-1}) \bigg(L_2(\p\otimes \m)\mbox{-}\lim_{j\to \infty} \sum_{i=1}^j i I_{i-1}\bigg(\tilde f_i(\cdot, x) \1_{[0, t]}^{(i-1)} \1_{[0, t]}(s)\bigg)\bigg) \m(\od s, \od x) \notag \\
	& = m\int_{B_m} \E \[I_{m-1}(k_{m-1}) I_{m-1}\(\tilde f_m(\cdot, x) \1_{[0, t]}^{(m-1)}\1_{[0, t]}(s)\)\] \m(\od s, \od x) 
	 \notag \\
	& = m! \int_{B_m} \int_{B_1\times \cdots \times B_{m-1}} \tilde f_m(x_1, \ldots, x_{m-1}, x) \m(\od s_1, \od x_1)\cdots \m(\od s_{m-1}, \od x_{m-1}) \m(\od s, \od x). \label{eq:compute-RHS}
	\end{align}
	Here, one uses \eqref{eq:malliavin-derivative-formula} and the fact that  $I_{m-1}(k_{m-1})$ is $\F_s$-measurable for all $s\in (s_m, t_m]$ to obtain \eqref{eq:F-s-measurability}.  Combining \eqref{eq:compute-LHS} with \eqref{eq:compute-RHS} yields \eqref{eq:test-product}.
	
	\smallskip
	
	\textit{Step 3}. For any $t\in (0, T)$, Jensen's inequality implies that $\E|f(X_T)|^2 \gee \E|F(t, X_t)|^2$. Then, we apply \textit{Step 2} and It\^o's isometry to obtain 
	\begin{align*}
	\E|f(X_T)|^2  \gee  |\E f(X_T)|^2 + \E\int_0^t\!\!\int_{\R} |\cD F(s, x)|^2 \m(\od s, \od x).
	\end{align*}
	Letting $t\uparrow T$, we infer that the stochastic integral
	$\int_0^T\int_{\R}\cD F(s, x) M(\od s, \od x)$
	exists as an element in $L_2(\p)$ and equals to $L_2(\p)$-$\lim_{t\uparrow T} \int_0^t\int_{\R} \cD F(s, x) M(\od s, \od x)$. On the other hand, due to the martingale convergence theorem, $F(t, X_t) =\ce{\F_t}{f(X_T)} \to \ce{\F_{T-}}{f(X_T)}$ a.s. and in $L_2(\p)$ as $t\uparrow T$, where $\F_{T-} := \sigma(\cup_{t<T}\F_t)$. Since $(\F_t)_{t\in [0, T]}$ is the augmented natural filtration of the L\'evy process $X$, it holds that $\F_{T-}  = \F_T$, and hence the desired conclusion follows. 
\end{proof}

\begin{rema}
	\cref{thm:martingale-representation}  extends \cite[Proposition 7]{CTV05} in which the function $f$ has a polynomial growth and $X$ satisfies certain conditions. A similar representation to \eqref{eq:martingale-representation} in a general framework (with different assumptions from ours) can be found in the proof of \cite[Theorem 2.4]{JMP00}. On the other hand, when $f(X_T)$ is Malliavin differentiable then one can use the Clark--Ocone formula (see, e.g.,  \cite{AS15, BNLOP03, Lo04}) to obtain its explicit martingale representation. However, the Malliavin differentiability of $f(X_T)$ fails to hold in many contexts. For example, if $f(x) =\1_{[K, \infty)}(x)$ for some $K \in \R$, and if $X$ is of infinite variation and $X_T$ has a density satisfying a mild condition, then $f(X_T)$ is not Malliavin differentiable, see \cite[Theorem 6(b)]{La20}.
	
	The representation \eqref{eq:martingale-representation} is a Clark--Ocone type formula but $f(X_T)$ is not necessarily differentiable in the Malliavin sense. The idea exploiting Malliavin calculus to obtain explicit integrands in the martingale representation of non-differentiable $f(X_T)$ was used in \cite{Pe08} for more general underlying process $X$. However, their results require the presence of the diffusion component of $X$, which in our context means $\sigma >0$.
\end{rema}

\section{Closed form for the local risk-minimizing strategy}
\label{sec:LRM-strategy}

This section gives the proof of \cref{theo:LRM-strategy}. First, let us fix the setting of this section.

\subsection{Setting} \label{sett:section4} Let $S = \e^X$ be the exponential of a L\'evy process $X$ with $(X|\p) \sim (\gamma, \sigma, \nu)$.  Assume that
 $\sigma^2 + \nu(\R) \in (0, \infty]$ and $\int_{x>1} \e^{2x} \nu(\od x) <\infty$.

Condition $\sigma^2 + \nu(\R) >0$ is simply to exclude the degenerate case that $X$ is a.s. deterministic. Since $\int_{x>1} \e^{2x} \nu(\od x) <\infty \Leftrightarrow \int_{|x|>1} \e^{2x} \nu(\od x) <\infty$, it follows from \cite[Theorem 25.3]{Sa13} that $\int_{x>1} \e^{2x} \nu(\od x) <\infty$ is equivalent to the square integrability of $S$. Moreover,  It\^o's formula yields
\begin{align*}
S & = 1+ \(\int_0^{\boldsymbol{\cdot}} \sigma S_{t-} \od W_t + \int_0^{\boldsymbol{\cdot}}\!\! \int_{ \R_0} S_{t-}(\e^x -1) \wt N(\od t, \od x)\) + \gamma_{\eqref{gammaS}} \int_0^{\boldsymbol{\cdot}} S_{t-} \od t  =:  1 + S^{\mar} + S^{\fv},
\end{align*}
where $S^{\mar}$ and $S^{\fv}$ respectively denote the martingale part and the predictable finite variation part in the representation of $S$, and where
\begin{align}\label{gammaS}
\gamma_{\eqref{gammaS}} := \gamma + \frac{\sigma^2}{2} + \int_{\R}(\e^x - 1- x\1_{\{|x|\lee 1\}})\nu(\od x).
\end{align}
Recall from \cref{theo:LRM-strategy} the  notation 
\begin{align*}
\sn = \sigma^2 + \int_{\R} (\e^x -1)^2 \nu(\od x) \in (0, \infty).
\end{align*}

\subsection{F\"ollmer--Schweizer (FS) decomposition}

We briefly present the FS decomposition of a random variable and the notion of the minimal local martingale measure which is the key tool to determine the FS decomposition. We refer the reader to \cite{Sc01} for  a survey about these objects.

In this article, we follow \cite[p.863]{HKK06} and use the family of \textit{admissible strategies} as
\begin{align}\label{eq:adm-strategies}
\mathcal A(S|\p) := \left\{\vartheta \mbox{ predictable} : \E \int_0^T \vartheta_t^2 S_{t-}^2 \od t <\infty \right\}.
\end{align}
It turns out that if $\vartheta \in \mathcal A(S|\p)$, then 
\begin{align}\label{eq:QV-stochastic-integral}
\E \int_0^T \vartheta^2_t \od [S, S]_t & = \E \int_0^T \vartheta^2_t \od [S^{\mar}, S^{\mar}]_t = \E \int_0^T \vartheta^2_t \od\<S^{\mar}, S^{\mar}\>_t = \sn \E \int_0^T \vartheta_t^2 S_{t-}^2 \od t <\infty.
\end{align}

\begin{defi}[\cite{Sc01}]\label{defi:FS-decomposition} 	\begin{enumerate}[\rm(1)]
		\item An $H \in L_2(\p)$ admits an \textit{FS decomposition} if one can express
		\begin{align*}
		H = H_0 + \int_0^T \vartheta^{H}_t \od S_t + L^{H}_T,
		\end{align*}
		where $H_0 \in \R$, $\vartheta^{H} \in \mathcal A(S|\p)$ and $L^{H} \in \cM_2^0(\p)$ is $\p$-orthogonal to $S^{\mar}$.
		
		\item  The integrand $\vartheta^H$ is called the \textit{local risk-minimizing} strategy of $H$.
	\end{enumerate}
\end{defi}

\begin{rema}\label{rema:SC-condition}
	In our context, $S$ satisfies the \textit{structure condition}, and the \textit{mean-variance trade-off process} $\wh K$ of $S$ in the sense of \cite[p.553]{Sc01} is 
	\begin{align*}
	\wh K_t = \frac{\gamma_{\eqref{gammaS}}^2}{\sn}t,
	\end{align*} 
	which is uniformly bounded in $(\omega, t) \in \Omega \times [0, T]$. Hence, \cite[Theorem 3.4]{MS95} implies that any $H \in L_2(\p)$ admits a unique FS decomposition. 
\end{rema}


		We continue with the notion of minimal martingale measure.
		
		\begin{defi}[\cite{Sc95}, Section 2]\label{defi:minimal-martingale-measure}
		Let $\mathcal E(U) \in \CL([0, T])$ be the stochastic exponential of $U$, i.e. $\od \cE(U) = \cE(U)_{-}\od U$ with $\cE(U)_0=1$, where
			\begin{align}\label{eq:MMM-density-U}
			U  = - \frac{\gamma_{\eqref{gammaS}}}{\sn}\(\sigma W + \int_0^{\boldsymbol{\cdot}} \!\!\int_{\R_0}(\e^x -1) \wt N(\od s, \od x)\).
			\end{align}
			If $\cE(U)>0$, then the probability measure $\mmm$ defined by 
			\begin{align*}
			\od \mmm := \mathcal E(U)_T \od \p
			\end{align*}
			is called the \textit{minimal martingale measure} for $S$. 
		\end{defi}

Since $U$ given in \eqref{eq:MMM-density-U} is a L\'evy process and belongs to $\cM_2^0(\p)$, it is known that $\cE(U)$ is also an $L_2(\p)$-martingale (see, e.g., \cite[Ch.V, Theorem 67]{Pr05} or \cite[Lemma 1]{GGL13}). 

We now provide a condition imposed on the characteristics of $X$ such that $\mmm$ exists as a probability measure. Let $(U|\p) \sim (\gamma_U, \sigma_U, \nu_U)$ and denote 
$$\alpha_U(x): = -\frac{\gamma_{\eqref{gammaS}}(\e^x -1)}{\sn}, \quad x \in \R.$$
 Then, it follows from \eqref{eq:MMM-density-U} that 
\begin{align}\label{eq:characteristics-U}
\gamma_U = - \int_{|\alpha_U(x)|>1} \alpha_U(x) \nu(\od x), \quad \sigma_U = \frac{|\gamma_{\eqref{gammaS}}|\sigma}{\sn}, \quad \nu_U = \nu \circ \alpha_U^{-1}.
\end{align}
Because
$$\cE(U) >0 \Leftrightarrow \Delta U >-1 \Leftrightarrow \nu_U((-\infty, -1]) =0 \Leftrightarrow \nu\Big(\Big\{x \in \R : 1 - \tfrac{\gamma_{\eqref{gammaS}}(\e^x -1)}{\sn} \lee 0\Big\}\Big) =0,$$ 
the following assumption ensures the existence of $\mmm$ as a probability measure: 
\begin{assum}\label{assum:condition-MMM} $\nu(\R\backslash A) =0$ where $A: = \Big\{x \in \R : 1 - \frac{\gamma_{\eqref{gammaS}}(\e^x -1)}{\sn} > 0\Big\}$.
\end{assum}
A sufficient condition for \cref{assum:condition-MMM} is 
\begin{align}\label{eq:sufficient-assum:condition-MMM}
	\gamma_{\eqref{gammaS}} (\e^x-1)< \sn, \quad \forall x\in \supp\nu,
\end{align}
and \eqref{eq:sufficient-assum:condition-MMM} is particularly satisfied when
$$0\gee \gamma_{\eqref{gammaS}} \gee  - \sn.$$

Assume that \cref{assum:condition-MMM} holds. Then, an application of Girsanov's theorem (see, e.g.,  \cite[Propositions 2 and 3]{ES05}) asserts that $X$ is a L\'evy process under $\mmm$ with $(X|\mmm) \sim (\gamma^*, \sigma^*, \nu^*)$, where
\begin{align}
\gamma^* &= \gamma -\frac{\gamma_{\eqref{gammaS}}}{\sn}\bigg(\sigma^2 + \int_{|x|\lee 1}x(\e^x -1) \nu(\od x)\bigg), \notag \\
\sigma^* & = \sigma \quad \mbox{and} \quad \nu^* (\od x) = \(1- \frac{\gamma_{\eqref{gammaS}} (\e^x-1)}{\sn}\) \nu(\od x). \label{eq:Levy-measure-MMM}
\end{align}
Moreover, a calculation shows $\int_{|x| >1} \e^x \nu^*(\od x) <\infty$ and $\gamma^* + \frac{\sigma^2}{2} + \int_{\R}(\e^x - 1- x\1_{\{|x|\lee 1\}})\nu^*(\od x) = 0$ which implies that $S = \e^X$ is a $\mmm$-martingale via using It\^o's formula. Let $W^*$ and $\wt N^*$ respectively denote the Brownian motion and the compensated Poisson random measure of $X$ under $\mmm$, then
\begin{align}
W_t^*   = W_t + \frac{\gamma_{\eqref{gammaS}} \sigma}{\sn}t, \quad \wt N^*(\od t, \od x)   = \wt N(\od t, \od x) + \frac{\gamma_{\eqref{gammaS}}}{\sn}(\e^x -1)\nu(\od x) \od t. \label{eq:WN-under-MMM}
\end{align}

In the sequel, let $\mme$ (resp. $\E^*_{\G}$) denote the expectation (resp. conditional expectation given a $\sigma$-algebra $\G \subseteq \F$) with respect to $\mmm$.

\subsection{Proof of \cref{theo:LRM-strategy}} 
Let $f(x): = g(\e^x)$ and $F^*(t, x): = \mme f(x+X_{T-t})$  so that $G^*(t, \e^x) = F^*(t, x)$ for $(t, x) \in [0, T] \times \R$. We define
	\begin{align*}
	\cD_J G^*(t, x) : =  G^*(t, \e^x S_{t-}) - G^*(t, S_{t-}), \quad (t, x) \in [0, T] \times \R.
	\end{align*}

	\eqref{item:LRM-strategy-form} We present here a direct proof for this assertion, an alternative argument in an abstract setting can be found in \cite[Proof of Theorem 4.3]{CVV10}. By assumption, $f(X_T) = g(S_T) \in L_2(\p^*)$ and  $\mme |f(x+X_{t})| = \mme |g(\e^x S_{t})| <\infty$ for any $(t, x) \in [0, T] \times \R$, we apply \cref{thm:martingale-representation} to obtain
	\begin{align}\label{eq:mart-representation-H}
	K^* = \mme g(S_T) + \int_0^{\boldsymbol{\cdot}} \sigma S_{t-}\pd_y G^*(t, S_{t-}) \od W_t^* + \int_0^{\boldsymbol{\cdot}} \!\! \int_{\R_0} \cD_J G^*(t, x) \wt N^*(\od t, \od x),
	\end{align}
	where $K^* = (K^*_t)_{t\in [0, T]}$ denotes a c\`adl\`ag version of the $L_2(\mmm)$-martingale $(\cem{\F_t}{g(S_T)})_{t\in [0, T]}$, and
	where $W^*$ and $\wt N^*$ are introduced in \eqref{eq:WN-under-MMM}. Then, $ \cE(U) K^*$ is a martingale under $\p$. Since the $\p$-martingale $U$ given in \eqref{eq:MMM-density-U} satisfies 
	\begin{align*}
	\|\<U, U\>_T\|_{L_\infty(\p)} = \frac{\gamma_{\eqref{gammaS}}^2 T}{\sn^2}\(\sigma^2 + \int_{\R} (\e^x -1)^2 \nu(\od x)\) <\infty,
	\end{align*}
	it implies that $\cE(U)$ is regular and satisfies $(R_2)$ in the sense of \cite[Proposition 3.7]{CKS98}. Since $K^*_T = g(S_T) \in L_2(\p)$ by assumption, 
we apply \cite[Theorem 4.9((i)$\Leftrightarrow$(ii))]{CKS98} to obtain
\begin{align*}
\E [K^*, K^*]_T <\infty.
\end{align*}
Combining this with \eqref{eq:mart-representation-H} yields
\begin{align*}
\E \int_0^T \sigma^2 |S_{t-} \pd_y G^*(t, S_{t-})|^2 \od t + \E \int_0^T\!\!\int_{\R_0} |\cD_J G^*(t, x)|^2 N(\od t, \od x) = \E [K^*, K^*]_T <\infty.
\end{align*}
Since $\od t \nu(\od x)$ is the predictable $\p$-compensator of $N(\od t, \od x)$, it implies that
\begin{align}\label{eq:integrable-under-P}
\E \int_0^T \sigma^2 |S_{t-} \pd_y G^*(t, S_{t-})|^2 \od t + \E \int_0^T\!\!\int_{\R} |\cD_J G^*(t, x)|^2 \nu(\od x) \od t <\infty.
\end{align}
Using Cauchy--Schwarz's inequality yields
\begin{align}\label{eq:integrable-LRM-strategy-P}
&\E \int_0^T \sigma^2 S_{t-}^2 |\pd_y G^*(t, S_{t-})| \od t  + \E \int_0^T\!\!\int_{\R}|\cD_J G^*(t, x) S_{t-} (\e^x -1)| \nu(\od x) \od t \notag \\
& \lee  \sqrt{\E\int_0^T S_{t-}^2 \od t} \sqrt{\E \int_0^T  |\sigma^2 S_{t-} \pd_y G^*(t, S_{t-})|^2 \od t} \notag \\
& \qquad + \sqrt{\int_{\R} (\e^x -1)^2  \nu(\od x)} \sqrt{\E \int_0^T S_{t-}^2 \od t } \sqrt{\E \int_0^T\!\!\int_{\R}|\cD_J G^*(t, x)|^2 \nu(\od x) \od t} \notag \\
& <\infty.
\end{align}
On the other hand, the FS decomposition of $H = g(S_T)$ is
	\begin{align}\label{eq:FS-decomp-g(S_T)}
	g(S_T) = H_0 + \int_0^T \vartheta^H_t \od S_t + L_T^H
	\end{align}
	where $H_0 \in \R$, $\vartheta^H \in \mathcal A(S|\p)$ and $L^H \in \cM_2^0(\p)$ is $\p$-orthogonal to the martingale component $S^{\mar}$ of $S$. According to \cite[Eq. (3.10)]{Sc01}, it holds that $L^H$ is a local $\mmm$-martingale.  Note that $\int_0^{\boldsymbol{\cdot}} \vartheta^H_t \od S_t$ is also a local $\mmm$-martingale. Using Cauchy--Schwarz's inequality and \eqref{eq:QV-stochastic-integral} yields
	\begin{align*}
	\mme\sqrt{[L^H, L^H]_T} & \lee \|\cE(U)_T\|_{L_2(\p)} \sqrt{\E [L^H, L^H]_T}  <\infty,\\
	\mme \sqrt{\int_0^T |\vartheta^H_t|^2 \od [S, S]_t} & \lee \|\cE(U)_T\|_{L_2(\p)} \sqrt{\E \int_0^T |\vartheta^H_t|^2 \od [S, S]_t}   <\infty.
	\end{align*} 
	Hence, the Burkholder--Davis--Gundy  inequality verifies that both $L^H$ and $\int_0^{\boldsymbol{\cdot}} \vartheta^H_t \od S_t$ are $\mmm$-martingales. Combining \eqref{eq:mart-representation-H} with \eqref{eq:FS-decomp-g(S_T)}, we derive $H_0 = \mme g(S_T)$ and
	\begin{align}\label{eq:identity-for-g(S)}
	\int_0^{\boldsymbol{\cdot}} \vartheta^H_t \od S_t + L^H = \int_0^{\boldsymbol{\cdot}} \sigma S_{t-}\pd_y G^*(t, S_{t-}) \od W_t^* + \int_0^{\boldsymbol{\cdot}} \!\! \int_{\R_0} \cD_J G^*(t, x) \wt N^*(\od t, \od x).
	\end{align}
	Recall that the martingale part of $S$ is $S^{\mar} = \int_0^{\boldsymbol{\cdot}} \sigma S_{t-} \od W_t + \int_0^{\boldsymbol{\cdot}}\int_{\R_0} S_{t-}(\e^x-1)\wt N(\od t, \od x)$. Since $\<L^H, S^{\mar}\>^\p = 0$ by the definition of the FS decomposition, we take the predictable quadratic covariation on both sides of \eqref{eq:identity-for-g(S)} with $S^{\mar}$ under $\p$ and notice that the integrability condition \eqref{eq:integrable-LRM-strategy-P} holds to obtain
	\begin{align*}
	\sn \int_0^{\boldsymbol{\cdot}} \vartheta^H_t S_{t-}^2\od t= \int_0^{\boldsymbol{\cdot}} \sigma^2 S_{t-}^2 \pd_y G^*(t, S_{t-}) \od t + \int_0^{\boldsymbol{\cdot}} \!\! \int_{\R} \cD_J G^*(t, x) S_{t-}(\e^x -1) \nu(\od x) \od t,
	\end{align*}
	which yields \eqref{eq:LRM-strategy} as desired.
	
	\medskip
	\eqref{item:LRM-strategy-cadlag} It follows from Cauchy--Schwarz's inequality and \eqref{eq:integrable-under-P} that
	\begin{align}\label{eq:integrable-strategy-under-MMM}
	& \mme \int_0^T |\sigma^2 S_{t-} \pd_y G^*(t, S_{t-})| \od t + \mme \int_0^T \!\! \int_{\R} |\cD_J G^*(t, x) (\e^x-1)| \nu(\od x) \od t \notag \\
	& \lee \sqrt{T} \|\cE(U)_T\|_{L_2(\p)} \sqrt{\E \int_0^T |\sigma^2 S_{t-} \pd_y G^*(t, S_{t-})|^2 \od t} \notag \\
	& \qquad + \|\cE(U)_T\|_{L_2(\p)} \sqrt{T \int_{\R}(\e^x -1)^2 \nu(\od x)}  \sqrt{\E \int_0^T\!\! \int_{\R}|\cD_J G^*(t, x)|^2 \nu(\od x)\od t} \notag \\
	& < \infty.
	\end{align}
	By assumption, it is clear that $(G^*(t, \e^x S_t) - G^*(t, S_t))_{t\in [0, T]}$ is a $\mmm$-martingale  for each $x\in \R$. In the case $\sigma >0$, due to $g(S_T) \in L_2(\mmm)$ and \cref{lemm:sigma>0}, $(S_t \pd_y G^*(t, S_t))_{t\in [0, T)}$ is also a $\mmm$-martingale.  Hence, the function
	\begin{align*}
	[0, T) \ni t \mapsto \mme |\sigma^2 S_t \pd_yG^*(t, S_t)| + \mme \int_{\R}|G^*(t, \e^x S_t) - G^*(t, S_t)| |\e^x -1|\nu(\od x)
	\end{align*}
	is non-decreasing by the martingale property. In addition, noting that $S_{t-} = S_t$ a.s. for each $t\in [0, T]$, we infer from \eqref{eq:integrable-strategy-under-MMM} and Fubini's theorem that
	\begin{align*}
	\mme |\sigma^2 S_{t} \pd_y G^*(t, S_t)| + \mme \int_{\R} |G^*(t, \e^x S_t) - G^*(t, S_t)||\e^x -1| \nu(\od x) <\infty
	\end{align*} 
	for all $t\in [0, T)$. Therefore,
	\begin{align*}
	 \(\frac{1}{\sn}\(\sigma^2 S_t \pd_y G^*(t, S_t) + \int_{\R}(G^*(t, \e^x S_t) - G^*(t, S_t))(\e^x -1)\nu(\od x)\)\)_{t\in [0, T)}
	\end{align*}
	is a $\mmm$-martingale for which one can find a c\`adl\`ag modification, denoted by $\varphi^g$. Then, the process $\tilde \vartheta^g$ defined by
	\begin{align}\label{eq:defi-cadlag-version-LRM}
	\tilde \vartheta^g : = \varphi^g/S
	\end{align}
	satisfies the desired requirements. 
\qed

\begin{rema}\label{rema:unique-cadlag-verion-LRM}
	Let $\tilde \vartheta \in \CL([0, T))$ such that $\tilde \vartheta = \tilde \vartheta^g$ for $\p \otimes \Leb$-a.e. $(\omega, t) \in \Omega \times [0, T)$, where $\tilde \vartheta^g$ given in \eqref{eq:defi-cadlag-version-LRM}. Then,  $\p(\tilde \vartheta_{t} = \tilde \vartheta^g_{t}, \; \forall t\in [0, T)) =1$ due to the c\`adl\`ag property. Hence, $\tilde \vartheta_{-}$ is also a LRM strategy of $H = g(S_T)$. Moreover, for any $t \in [0, T)$,
	\begin{align*}
	\tilde \vartheta_t = \frac{1}{\sn}\(\sigma^2 S_t \pd_y G^*(t, S_t) + \int_{\R}(G^*(t, \e^x S_t) - G^*(t, S_t))(\e^x -1)\nu(\od x)\)\quad \mbox{a.s.}
	\end{align*}
\end{rema}

\section{Discrete-time hedging via weighted BMO-approach}\label{sec:Discrete-time hedging}

This section is a continuation of the work in \cite{Th20} in the exponential L\'evy models. First, we use the discrete-time approximation for stochastic integrals introduced in \cite{Th20} and investigate the resulting error in weighted $\BMO$ spaces. Consequently, the $L_p$-estimates, $p \in (2, \infty)$, for the error are provided. Secondly, to illustrate the obtained results, the approximated stochastic integrals we consider is the integral term, which can be interpreted as the hedgeable part, in the FS decomposition of a European type option. 


\subsection{Weighted bounded mean oscillation (BMO)} Let $\cS([0, T])$ denote the family of all stopping times $\rho\colon \Omega \to [0, T]$. Set $\inf \emptyset : = \infty$.

\begin{defi}[\cite{Ge05, GN20}]
Let $p\in (0, \infty)$ and $\Phi \in \CL^+([0, T])$.
\begin{enumerate}[\rm(1)]
	\item  Let $\BMO^\Phi_p(\p)$ consist of all $Y \in \CL_0([0, T])$ with $\|Y\|_{\BMO_p^\Phi(\p)} <\infty$, where
	\begin{align*}
	\|Y\|_{\BMO_p^\Phi(\p)}: = \inf\left\{c \gee 0 : \ce{\F_\rho}{|Y_T-Y_{\rho-}|^p} \lee c^p \Phi^p_{\rho} \quad\mbox{a.s., } \forall \rho \in \cS([0, T])\right\}.
	\end{align*}

	\item (\textit{Weight regularity}) We let $\Phi \in \cSM_p(\p)$ if $	\|\Phi\|_{\cSM_p(\p)} <\infty$, where
	\begin{align*}
	\|\Phi\|_{\cSM_p(\p)} : = \inf \left\{\ts c\gee 0: \bce{\F_a}{\sup_{t \in [a, T]} \Phi_t^p} \lee c^p \Phi_{a}^p \quad \mbox{a.s., } \forall a \in [0, T]\right\}.
	\end{align*} 
\end{enumerate}
\end{defi}
The theory of non-weighted $\BMO$-martingales, i.e., when $\Phi \equiv 1$ and $Y$ is a martingale, can be found in \cite[Ch.IV]{Pr05}. Remark that the weighted $\BMO$ spaces above were introduced and discussed in \cite{Ge05} for general c\`adl\`ag processes which are not necessarily martingales. 

\begin{defi}[\cite{Ge05}]\label{defi:RH-space} For $s\in (1, \infty)$, we denote by $\RH_s(\p)$ the family of all probability measures $\bQ$ equivalent to $\p$ such that    $\od \bQ/\od \p =: U \in L_s(\p)$ and there exists a constant $c_{\eqref{eq:RH-inequality}}>0$ such that $U$ satisfies the following \textit{reverse H\"older inequality }
	\begin{align}\label{eq:RH-inequality}
	\E^{\p}_{\F_\rho}[U^s] \lee c_{\eqref{eq:RH-inequality}}^s (\E^{\p}_{\F_\rho}[U])^s \quad \mbox{a.s., } \forall \rho \in \cS([0, T]).
	\end{align}
\end{defi}

We refer the reader to \cite{Ge05, GN20} for further properties of those quantities.  \cref{prop:feature-BMO} below recalls some features of weighted $\BMO$ which are crucial for our applications, and  their proofs can be found in \cite[Proposition A.6]{GN20} and \cite[Proposition 2.5]{Th20}.

\begin{prop}[\cite{GN20, Th20}]\label{prop:feature-BMO} Let $p\in (0, \infty)$ and $\Phi \in \CL^+([0, T])$.
	\begin{enumerate}[\rm (1)]

		\item \label{item:Lp-BMO-estimate} For any $r \in (0, \infty)$, there is a  $c_1= c_1(r, p)>0$ such that $\|\cdot\|_{S_p(\p)} \lee c_1 \|\Phi\|_{S_p(\p)} \|\cdot \|_{\BMO^\Phi_r(\p)}$.
		
	\item \label{item:equivalent-BMO} If $\Phi \in \cSM_p(\p)$, then for any $r\in (0, p]$ there exists a $c_2 = c_2(r, p, \|\Phi\|_{\cSM_p(\p)}) >0$ such that $\|\cdot\|_{\BMO^\Phi_p(\p)} \sim_{c_2} \|\cdot\|_{\BMO^\Phi_r(\p)}$. 
	
	\item \label{item:change-of-measure} If $\bQ \in \RH_s(\p)$ for some $s\in (1, \infty)$ and $\Phi \in \cSM_p(\bQ)$, then there exists a $c_3 = c_3(s, p)>0$ such that $\|\cdot \|_{\BMO^\Phi_p(\bQ)} \lee c_3 \|\cdot\|_{\BMO^\Phi_p(\p)}$.
	\end{enumerate}
\end{prop}

\subsection{Setting}\label{subsec-setting-5} We assume \Cref{sett:section4} from now until the end of \cref{sec:Discrete-time hedging}, i.e., $S = \e^X$ is a square integrable exponential L\'evy process under $\p$ with 
\begin{align*}
(X|\p)\sim (\gamma, \sigma, \nu) \quad \mbox{and} \quad 
\sigma^2 + \nu(\R) \in (0, \infty].	
\end{align*}
Then, as aforementioned in \Cref{subsec:exponential-levy}, there is a L\'evy process $Z$ such that $S$ is the stochastic exponential of $Z$, i.e.,
\begin{align}\label{eq:stoch-logarithm}
	\od S_t = S_{t-} \od Z_t, \quad t \in [0, T].
\end{align}
Let $(Z|\p) \sim (\gamma_Z, \sigma, \nu_Z)$. Then, the square integrability of $S$ implies  $\int_{\R} z^2 \nu_Z(\od z) <\infty$.

The stochastic integral we are going to discretize is of the form
\begin{align}\label{eq:approx-integral}
	\int_0^T \tilde\vartheta_{t-} \od S_t  = \int_0^T \tilde\vartheta_{u-} S_{t-} \od Z_t
\end{align}
for suitable $\tilde\vartheta \in \CL([0, T))$\footnote{We use the tilde sign to indicate the c\`adl\`ag property of the applied process.} introduced in \cref{assum:change-measure} below.

\subsection{Discrete-time approximation with jump correction scheme} \label{subsec:approximate-BMO-general} Although the martingale framework is convenient for handling the discrete-time hedging problem as one can simplify the selection of quadratic hedging strategies and reduce some technicalities, it is more natural to consider the problem in the semimartingale framework.  In principle, one can switch the semimartingale setting to the martingale setting by an appropriate change of measure. However, in general, the convergence rates of the discretization error might be slower if one use H\"older's inequality  to switch distributional estimates back to the original semimartingale setting.

The approach we use here is beneficial from the change of measure feature  of weighted BMO where the convergence rates remain the same after a (suitable) change of measure. Here, we do not assume the (local) martingale property under $\p$ for the underlying price process $S$. The idea is as follows: For $\tilde \vartheta$ as in \eqref{eq:approx-integral}, one looks for a suitable equivalent probability measure $\whP \sim \p$ such that the process $M:=\tilde \vartheta S$ is a square integrable martingale (not necessarily uniformly integrable) under $\whP$. We then exploit the $\whP$-orthogonality of increments of $M$, which naturally appear in the discretization of the stochastic integral, to achieve favorable estimates for the approximation error which possibly yields considerably good convergence rates. Here, $S$ is not necessarily a (local) $\whP$-martingale but a $\whP$-semimartingale. One also notices that the obtained convergence rates are still measured under $\whP$. Thanks to \cref{prop:feature-BMO}\eqref{item:change-of-measure}, under satisfiable conditions, one can switch from $\whP$ back to the original measure $\p$ so that the convergence rates under $\p$ remain unchanged. Such a change of measure is in general not unique, and hence, we have a flexibility to choose an appropriate $\whP$ so that the problem is convenient to handle. In our situation considered later, one can choose $\whP$ such that $S$ is also a (local) $\whP$-martingale.

 Let us introduce the main assumption.

\begin{assum}
	\label{assum:change-measure} For $\tilde\vartheta \in \CL([0, T))$, $\theta \in (0, 1]$, an a.s. non-decreasing  $\Theta \in \CL^+([0, T])$ and for a probability measure $\whP$ equivalent to $\p$, we let 
	$$\tilde \vartheta \in \bfA(\theta, \Theta | \whP)$$ if the following conditions hold:
	\begin{enumerate}[\rm (a)]
		\item \label{item:condition-stoch-integral} $\E \int_0^T \tilde \vartheta_t^2 S_t^2 \od t <\infty$.
		
		\item\label{item:equivalent-martingale-measure} $M: = \tilde \vartheta S \in \CL([0, T))$ is an $L_2(\whP)$-martingale.
			
		\item \label{item:growth-condition} \textit{$($Growth condition$)$}  There are constants $c_{\eqref{eq:theta-1}}, c_{\eqref{eq:theta-(0,1)}} >0$, which might depend on $\theta$, such that
		\begin{enumerate}[\rm (i)]
			\item Case $\theta =1$: One has 
			\begin{align}\label{eq:theta-1}
				|\tilde \vartheta_a| \lee c_{\eqref{eq:theta-1}} \Theta_a \quad \mbox{a.s., } \forall a \in [0, T).
			\end{align}
			
			\item Case $\theta \in (0, 1)$: One has
			\begin{align}\label{eq:theta-(0,1)}
				\hbbce{\F_a}{\int_{(a, T)} (T-t)^{- \theta} \tilde \vartheta_t^2 S_t^2 \od t} \lee c_{\eqref{eq:theta-(0,1)}}^2 \Theta_a^2 S_a^2 \quad \mbox{a.s., } \forall a \in [0, T).
			\end{align}
		\end{enumerate}

		\item \label{item:semimartingale-Z-hat-P} $Z$ is an $L_2(\whP)$-semimartingale with the canonical representation
		\begin{align*}
			Z_t = \bar Z_0 + \int_0^t \bar b^Z_s \od s + \sigma W^{\whP}_t + \int_0^t\!\! \int_{\R_0} z (N_Z - \pi_Z^{\whP})(\od z, \od s), \quad t\in [0, T],
		\end{align*}
		where $\bar Z_0 \in \R$,  $W^{\whP}$ is a Brownian motion under $\whP$, $\pi_Z^{\whP}$ is the predictable $\whP$-compensator of the jump measure $N_Z$ of $Z$ such that 
		\begin{align*}
			\pi_Z^{\whP}(\omega, \od z, \od t) = \nu_Z^{\whP} (t, \od z) \od t, \quad (\omega, t, z) \in \Omega \times [0, T] \times \R
		\end{align*}
		where $\nu_Z^{\whP} (t, \cdot)$, $t \in [0, T]$,  are L\'evy measures, and that 
		\begin{align}\label{eq:2nd-moment-hat-mu-Z}
			\ts \nu_{\eqref{eq:2nd-moment-hat-mu-Z}}: = \big\|t \mapsto  \int_{\R_0} z^2  \nu_Z^{\whP}(t, \od z)\big\|_{L_\infty([0, T], \Leb)} < \infty,
		\end{align}
		and  $\bar b^Z$ is a measurable and deterministic function with
		\begin{align}\label{eq:drift-Z-hat-P}
		 b_{\eqref{eq:drift-Z-hat-P}} : = \|\bar b^Z\|_{L_\infty ([0, T], \Leb)} <\infty.
		\end{align}
	\end{enumerate}
\end{assum}
	
Let us briefly comment on the conditions \eqref{item:condition-stoch-integral}--\eqref{item:growth-condition} above:
\begin{itemize}
	\item Since  c\`adl\`ag functions have at most countable many discontinuities, condition \eqref{item:condition-stoch-integral} ensures that $\tilde \vartheta_{-}$ is an admissible integrand, i.e., $\tilde \vartheta_{-} \in \mathcal A(S|\p)$ for $\mathcal A(S|\p)$ given in \eqref{eq:adm-strategies}.
	
	\item Apparently, condition \eqref{item:equivalent-martingale-measure} looks unnatural, however, this is typically satisfied in applications where $\tilde \vartheta_{-}$ is the hedging strategies of  European options derived from the quadratic hedging approach. We clarify this fact in \cref{exam:martinagle-structure} below. 
	
	\item The parameter $\theta \in (0, 1]$ in \eqref{item:growth-condition} is related to the growth property of $\tilde \vartheta$ (relatively to the corresponding weights and measures)  when the time variable approaches $T$. Remark that the bigger $\theta$ is, the less singular $\tilde \vartheta$ has at $T$, which leads the approximated stochastic integral to be more regular. In particular, in the Black--Scholes model where $\tilde \vartheta$ is the delta-hedging strategy of a European payoff, $\theta$ can be interpreted as the \textit{fractional smoothness} of the payoff in the sense of \cite{GT15} in which $\theta = 1$ corresponds to the smoothness of order 1. Therefore, in this article we call the situation where \eqref{item:growth-condition} holds for $\theta = 1$ as \textit{regular regime}. Further discussion on \eqref{item:growth-condition} is given in \cref{rema:parameter-theta}.
\end{itemize}    

\begin{exam}[Exponential L\'evy model]\label{exam:martinagle-structure} Consider a European payoff $g(S_T)$. \begin{enumerate}[\rm (1)]
		\item \label{exam:martinagle} (\textit{MVH and delta hedging in the martingale framework}) Assume that $S$ is a $\p$-martingale. Analogously to the Black--Scholes model, the \textit{delta hedging} strategy in the exponential L\'evy model is defined by $\tilde \vartheta_t = \frac{\pd G}{\pd y} (t, S_t)$ provided that the derivative exists, where $G(t,y) : = \E g(y S_{T-t})$ for $(t, y) \in [0, T] \times (0, \infty)$. Although being not optimal in general, the delta hedging is often used by financial market practitioners. Then, it is shown in \cite[Theorem 4.2]{Th20} that under an integrability condition, both delta hedging and MVH strategies satisfy \cref{assum:change-measure}\eqref{item:equivalent-martingale-measure} with $\whP = \p$.
		
		\item \label{exam:semimartinagle} (\textit{LRM in the semimartingale framework}) Under the assumptions of \cref{theo:LRM-strategy}, if  $\tilde \vartheta^g$ is a c\`adl\`ag modification for the LRM strategy of $g(S_T)$, then \cref{theo:LRM-strategy}\eqref{item:LRM-strategy-cadlag} verifies that $M: = \tilde \vartheta^g S$ is a c\`adl\`ag $\mmm$-martingale. In particular, in the application when $g$ is bounded or H\"older/Lipschitz continuous, then $M$ is square integrable under $\mmm$, see  \cref{thm:approximation-convergence-rate}\eqref{item:2:growth-strategy} below. Hence, \cref{assum:change-measure}\eqref{item:equivalent-martingale-measure}  holds for $\tilde\vartheta^g$ and $\whP = \mmm$.
	\end{enumerate}
\end{exam}

\begin{prop}\label{rema:parameter-theta}
	\begin{enumerate}[\rm (1)]
		\item \label{item:2:rema:parameter-theta} For $0 < \theta' < \theta <1$, one has $\bfA(\theta, \Theta|\whP) \subseteq \bfA(\theta', \Theta|\whP)$.
		
		\item \label{item:1:rema:parameter-theta} If $M = \tilde \vartheta S$ is an $L_2(\whP)$-martingale and $\theta \in (0,1)$, then \eqref{eq:theta-(0,1)} is equivalent to the following system of two inequalities:
		\begin{align}
			   & |\tilde\vartheta_a|  \lee c_{\eqref{item:growth-varphi}} (T-a)^{\frac{\theta -1}{2}} \Theta_a  \quad \mbox{a.s. for all } a \in [0, T);  \label{item:growth-varphi}\\
			   & \hbbce{\F_a}{\int_{(a, T)} (T-t)^{1- \theta} \od \<M\>^{\whP}_t}  \lee c_{\eqref{item:curvature-condition} }^2 \Theta_a^2 S_a^2 \quad \mbox{a.s. for all } a \in [0, T). \label{item:curvature-condition}
		\end{align}
		Here, $c_{\eqref{item:growth-varphi}}$, $c_{\eqref{item:curvature-condition} } >0$ are constants independent of $a$.

		\item\label{item:3:rema:parameter-theta} Assume Items \eqref{item:condition-stoch-integral}, \eqref{item:equivalent-martingale-measure}, \eqref{item:semimartingale-Z-hat-P} in \cref{assum:change-measure}.  If $\Theta S \in \cSM_2(\whP)$ and if there exist a constant $c_{\eqref{eq:upper-weight-function}}>0$ and a measurable function $w\colon [0, T) \to [0, \infty)$ such that
		\begin{align}\label{eq:upper-weight-function}
			|\tilde \vartheta_t| \lee c_{\eqref{eq:upper-weight-function}} w(t) \Theta_t \quad \mbox{a.s., } \forall t\in [0, T)
		\end{align}
	and that $I_w: = \{\delta \in (0, 1] : \int_0^T (T-t)^{- \delta} w(t)^2 \od t <\infty\} \neq \emptyset$. Then $\tilde \vartheta \in \cap_{0 < \theta'<\theta}\bfA(\theta', \Theta|\whP)$ for $\theta : = \sup I_w$.
	\end{enumerate}
\end{prop}

\begin{proof}
\eqref{item:2:rema:parameter-theta} is obvious. \eqref{item:1:rema:parameter-theta} Let $\theta \in (0,1)$ and assume that $M = \tilde \vartheta S$ is an $L_2(\whP)$-martingale. For $a\in [0, T)$, by the orthogonality and Fubini's theorem, we get, a.s., 
\begin{align*}
	\hbbce{\F_a}{\int_{(a, T)} (T-t)^{- \theta} \tilde \vartheta_t^2 S_t^2 \od t} & = \hbbce{\F_a}{\int_{(a, T)} (T-t)^{- \theta} (M_t - M_a)^2 \od t} + M_a^2 \int_a^T (T-t)^{-\theta} \od t \\
	& =  \hbbce{\F_a}{\int_{(a, T)} (T-t)^{- \theta} \int_{(a, t]} \od \<M\>^{\whP}_u \od t} + M_a^2 \frac{(T-a)^{1- \theta}}{1- \theta}\\
	& = \frac{1}{1- \theta} \bigg[ \hbbce{\F_a}{\int_{(a, T)} (T-u)^{1 - \theta} \od \<M\>^{\whP}_u} + (T-a)^{1- \theta} M_a^2 \bigg],
\end{align*}
which yields the equivalence as desired.

\eqref{item:3:rema:parameter-theta} For any $0 < \theta' < \theta = \sup I_w$, there exits $\delta \in (\theta', \theta]$ such that
$\int_0^T (T-t)^{-\delta} w(t)^2 \od t<\infty$. Since $\Theta S \in \cSM_2(\whP)$, it holds that, for $a\in [0, T)$, a.s., 
\begin{align*}
	\hbbce{\F_a}{\int_{(a, T)} (T-t)^{- \theta'} \tilde \vartheta_t^2 S_t^2 \od t} & \lee c_{\eqref{eq:upper-weight-function}}^2 \|\Theta S\|_{\cSM_2(\whP)}^2 \Theta_a^2 S_a^2 \int_a^T (T-t)^{-\theta'} w(t)^2 \od t\\
	& \lee  \bigg[c_{\eqref{eq:upper-weight-function}}^2 \|\Theta S\|_{\cSM_2(\whP)}^2 T^{\delta - \theta'} \int_0^T (T-t)^{-\delta} w(t)^2 \od t \bigg]\Theta_a^2 S_a^2,
\end{align*} 
which then implies $\ttt \in \bfA(\theta', \Theta | \whP)$.
\end{proof}

\subsubsection{Deterministic discretization times} Let $\cT_{\det}$ denote the family of all \textit{deterministic} time-nets $\tau = (t_i)_{i=0}^n$, $0=t_0<t_1<\cdots<t_n =T$, $n\gee 1$. The mesh size of $\tau = (t_i)_{i=0}^n \in \cT_{\det}$ associated with a parameter $\theta \in (0, 1]$ is defined by
\begin{align*}
	\|\tau\|_{\theta} : = \max_{i =1, \ldots, n} \frac{t_i - t_{i-1}}{(T-t_{i-1})^{1- \theta}} = \max_{i =1, \ldots, n} \; \sup_{a \in [t_{i-1}, t_i)}\frac{t_i - a}{(T- a)^{1- \theta}}.
\end{align*}
Let $\tau_n \in \cT_{\det}$ with $\# \tau_n = n+1$. By a short calculation we find that $\|\tau_n\|_\theta \gee T^\theta/n$. Minimizing $\|\tau_n\|_\theta$ over $\tau_n \in \cT_{\det}$ with $\#\tau = n+1$ leads us to the following time-nets, which were exploited in \cite{GGL13, Ge02, Ge05, GN20, Th20}: For $\theta \in (0, 1]$ and $n\gee 1$, the \textit{adapted time-net} $\tau^\theta_n = (t^\theta_{i, n})_{i=0}^n$ is given by
\begin{align}\label{eq:adapted-time-net}
	t^\theta_{i, n} : = T  - T \big(1-\tfrac{i}{n}\big)^{1/\theta}, \quad i =1, \ldots, n, \quad\mbox{and satisfies } \tfrac{T^\theta}{n} \lee \|\tau_n^\theta\|_\theta \lee \tfrac{T^\theta}{\theta n}, \quad n \gee 1.
\end{align}

\subsubsection{Approximation with jump correction} We recall from \cite{Th20} the discrete-time approximation with the jump correction method. Let $\tilde\vartheta \in \CL([0, T))$ such that $\E\int_0^T \tilde\vartheta_t^2 S_t^2\od t <\infty$.
The Riemann approximation $A^{\riem}(\tilde \vartheta, \tau)$ of $\int_0^T \tilde\vartheta_{t-} \od S_t$ along with the time-net $\tau = (t_i)_{i=0}^n\in \cT_{\det}$ is given by
\begin{align*}
	A^{\riem}_t(\tilde\vartheta, \tau) : = \sum_{i=1}^n \tilde \vartheta_{t_{i-1}- } (S_{t_i \wedge t} - S_{t_{i-1} \wedge t}), \quad t\in [0, T].
\end{align*}

The following random times are employed to capture the relatively large jumps of $S$: for $\ep>0$ and $\kappa \gee 0$, we define  $\rho(\ep, \kappa) = (\rho_i(\ep, \kappa))_{i\gee 0}$ by setting $\rho_0(\ep, \kappa): = 0$ and 
\begin{align*}
	\rho_i(\ep, \kappa) : = \inf\{T \gee t > \rho_{i-1}(\ep, \kappa) : |\Delta S_t| >  \ep (T-t)^{\kappa} S_{t-}\} \wedge T,\quad  i \gee 1.
\end{align*}
Since $Z$ is c\`adl\`ag and the considered filtration satisfies the usual conditions, it holds that $\rho_i(\ep, \kappa)$ are stopping times. The quantity $\ep(T-t)^\kappa$ is the threshold at the time $t$ from which we decide which jumps of $S$ are (relatively) large. If we normalize $T=1$ and let $t = 0$, then $\ep>0$ can be regarded as the \textit{initial jump size threshold}. Moreover, for $\kappa>0$, this threshold continuously shrinks as $t \uparrow T$, and thus the parameter $\kappa$ indices the \textit{jump size decay rate}.  The reason for using the decay function $\ep(T-t)^\kappa$ is to compensate the growth of integrands.
\begin{defi}[\cite{Th20}]\label{def:approx-correction}
	Let $\ep >0$, $\kappa \in [0, \frac{1}{2})$ and $\tau = (t_i)_{i =0}^n \in \cT_{\det}$. 
	\begin{enumerate}[\rm (1)]
		\item Let $\comb{\tau}{\rho(\ep, \kappa)}$ denote the \textit{combined time-net}  derived from combining $\tau$ with $\rho(\ep, \kappa)$ and re-ordering their time-knots.
		
		\item For $t \in [0, T]$, we define
		\begin{align*}
			\tilde\vartheta^\tau_t  &: = \sum_{i=1}^n \tilde\vartheta_{t_{i-1}-} \1_{(t_{i-1}, t_i]}(t),\\
			A^{\corr}_t(\tilde\vartheta, \tau | \ep, \kappa) &: = A^{\riem}_t(\tilde\vartheta, \tau) + \sum_{\rho_i(\ep, \kappa) \in [0, t] \cap [0, T)} \(\tilde\vartheta_{\rho_i(\ep, \kappa)-} - \tilde\vartheta^\tau_{\rho_i(\ep, \kappa)}\) \Delta S_{\rho_i(\ep, \kappa)},\\
			E^{\corr}_t(\tilde\vartheta, \tau | \ep, \kappa) &: = \int_0^t \tilde\vartheta_{u-} \od S_u - A^{\corr}_t(\tilde\vartheta, \tau | \ep, \kappa).
		\end{align*}
	\end{enumerate}
\end{defi}

The number of random jump times is denoted by 
\begin{align*}
	\cN(\ep, \kappa): = \inf\{ i \gee 1: \rho_i(\ep, \kappa) = T\}.
\end{align*}
Under the condition \eqref{eq:2nd-moment-hat-mu-Z}, we apply \cref{lem:cardinality-quantify} with $\alpha = 2$ and $\ep_n = \ep$ to conclude that $\cN(\ep, \kappa) < \infty$ a.s. for any $\ep >0$ and $\kappa \in [0, \frac{1}{2})$. Hence, the sum in the definition of $A^{\corr}(\tilde\vartheta, \tau | \ep, \kappa)$ is a finite sum a.s. Then, we adjust this sum on a set of probability zero so that one may assume that $A^{\corr}(\tilde\vartheta, \tau | \ep, \kappa)$ and $E^{\corr}(\tilde\vartheta, \tau | \ep, \kappa)$ belong to $\CL_0([0, T])$.

\subsubsection{A general approximation result}

There are three factors which affect the approximation rates: the presence of the diffusion component of $Z$, the small jump activity of the underlying L\'evy process, and the growth property of the integrand which is described by $\theta$.

\begin{theo}\label{thm:approximation-BMO}
	Let $\theta \in (0, 1]$, $\kappa:= \frac{1- \theta}{2}$ and $\tilde \vartheta \in \mathbf A(\theta, \Theta | \whP)$. Let $\Phi \in \CL^+([0, T])$ such that
		\begin{align*}
		\max\{\Theta_t S_t, \Theta_{t-}S_{t-}\} \lee \Phi_t, \quad \forall t \in [0, T].
			\end{align*}
	Assume that $\Phi \in  \cSM_2(\whP)$ and that
	\begin{align}\label{eq:small-jump-conditino}
	\ts c_{\eqref{eq:small-jump-conditino}}(\alpha): = 	\sup_{r \in (0, 1)} r^\alpha                                                                                                \big\|t \mapsto  \int_{r < |z| \lee 1} \nu_Z^{\whP}(t, \od z) \big\|_{L_\infty([0, T], \Leb)} < \infty \quad \mbox{for some } \alpha \in (0, 2].
	\end{align}
Then the following assertions hold:
	\begin{enumerate}[\rm(1)]
		\item\label{item:thm:rate-whP} For $\mathcal X = \BMO_2^{\Phi}(\whP)$, one has
\begin{align}\label{eq:thm:rate-whP}
	\sup_{\ep>0, \,\tau \in \cT_{\det}} \frac{ \big\|E^{\corr} (\tilde \vartheta, \tau \,\big|\, \ep, \kappa) \big\|_{\mathcal X}}{\max\big\{\sigma\sqrt{\|\tau\|_\theta}, H(\theta, \alpha, \ep, \|\tau\|_\theta)\big\}} <\infty,
\end{align}
	where 
	\begin{align*}
		H(\theta, \alpha, \ep, \|\tau\|_\theta) : = \begin{cases}
			\max\Big\{\ep,\, \ep^{1- \frac{\alpha}{2}} \sqrt{\|\tau\|_\theta},\, \|\tau\|_\theta + \ep^{1- \alpha} \|\tau\|_\theta^{1- \kappa(\alpha-1)}\Big\} & \mbox{ if } \alpha \in (1, 2]\\[3pt]
			\max\Big\{\ep,\, \sqrt{\ep} \sqrt{\|\tau\|_\theta},\, \big[1 + \log^+(\frac{1}{\ep}) + \log^+\big(\frac{1}{\|\tau\|_\theta}\big)\big] \|\tau\|_\theta\Big\} & \mbox{ if } \alpha =1\\[3pt]
			\max\Big\{\ep, \, \sqrt{\ep}\sqrt{\|\tau\|_\theta}, \, \|\tau\|_\theta\Big\} & \mbox{ if } \alpha \in (0, 1).
	\end{cases}
	\end{align*}

	\item \label{eq:thm:Sp-whP-estimate} ($S_p(\whP)$-estimates) If $\Phi \in S_p(\whP)$ for some $p \in (2, \infty)$, then \eqref{eq:thm:rate-whP}  holds for $\mathcal X = S_p(\whP)$.
	
	\item \label{eq:thm:P-estimate} (Switch back to $\p$) If $\p \in \RH_s(\whP)$  and $\Phi \in \cSM_r(\p)$ for some $s\in (1, \infty)$, $r \in (0, 2]$, then  \eqref{eq:thm:rate-whP} holds for $\mathcal X = \BMO^\Phi_r(\p)$.
	
	\item \label{eq:thm:Sp-P-estimate} ($S_p(\p)$-estimates) If $\p \in \RH_s(\whP)$  and $\Phi \in \cSM_r(\p) \cap S_p(\p)$ for some $s\in (1, \infty)$, $r \in (0, 2]$, $p \in [r, \infty)$, then  \eqref{eq:thm:rate-whP} holds for $\mathcal X = S_p(\p)$.
	\end{enumerate}
\end{theo}

\begin{proof} The idea for the proof of Item \eqref{item:thm:rate-whP} follows closely \cite[Proofs of Theorems 3.11 and 3.15]{Th20}. Hence, we provide in \cref{appen:proof:thm:approximation-BMO} the essential steps of the proof. Items \eqref{eq:thm:Sp-whP-estimate}, \eqref{eq:thm:P-estimate} and \eqref{eq:thm:Sp-P-estimate} follow from applying \cref{prop:feature-BMO}(Items \ref{item:Lp-BMO-estimate} and \ref{item:change-of-measure}).
\end{proof}

A minimization procedure for the function $H(\theta, \alpha, \ep, \|\tau\|_\theta)$ over $\ep>0$ leads to the selection
\begin{align*}
	\ep = \ep(\theta, \alpha, \|\tau\|_\theta): = \begin{cases}
		\|\tau\|_\theta^{\frac{1}{\alpha}(1- \kappa(\alpha-1))} & \mbox{ if } \alpha \in [1, 2]\\
		\|\tau\|_\theta & \mbox{ if } \alpha \in (0, 1).
	\end{cases}
\end{align*}
Continuing to minimize $H(\theta, \alpha, \ep(\theta, \alpha, \|\tau\|_\theta), \|\tau\|_\theta)$ over $\tau \in \cT_{\det}$ with $\# \tau = n+1$ yields to the adapted time-net $\tau = \tau^\theta_n$. Therefore, we arrive at the following corollary.

\begin{coro}[Convergence rates in terms of discretization cardinality]\label{coro:convergence-rate} Assume the assumptions of \cref{thm:approximation-BMO}.
	\begin{enumerate}[\rm (1)]
		\item \label{item:coro:convergence-rate} For  $\mathcal X = \BMO_2^\Phi(\whP)$ one has
		\begin{align}\label{eq:coro:convergence-rates}
			\sup_{n \gee 1}  R(n)\big\|E^{\corr}\big(\tilde \vartheta, \tau^\theta_n\,|\, \ep_n, \tfrac{1- \theta}{2}\big)\big\|_{\mathcal X} < \infty,
		\end{align}
		where the rate of convergence $R(n)$ and the initial jump size threshold $\ep_n$ are provided depending on $\sigma$ as follows:
		\begin{align}
		&\mbox{when } \sigma = 0 \mbox{ one has }	\begin{cases}
				R(n) = 1/\ep_n = n^{\frac{1}{\alpha}(1- \frac{1}{2}(1- \theta)(\alpha-1))} & \mbox{ if } \alpha \in (1, 2]\\
				R(n) = n/(1+ \log n), \ep_n = 1/n & \mbox{ if } \alpha = 1\\
				R(n) = 1/\ep_n = n & \mbox{ if } \alpha \in (0, 1),		
			\end{cases} \label{item:coro:convergence-rate-sigma=0}\\
			&\mbox{when } \sigma > 0 \mbox{ one has }	\begin{cases}
			R(n) = 1/\ep_n = n^{\frac{1}{\alpha}(1- \frac{1}{2}(1- \theta) (\alpha-1))} & \mbox{ if } \alpha \in (\frac{3 - \theta}{2- \theta}, 2]\\[3pt]
			R(n) = 1/\ep_n = n^{\frac{1}{2}} & \mbox{ if } \alpha  \in (0, \frac{3 - \theta}{2- \theta}]. 		
		\end{cases} \label{item:coro:convergence-rate-sigma>0}	
	\end{align}
		
		\item \label{item:coro:Sp-whP-estimate} ($S_p(\whP)$-estimates) If $\Phi \in S_p(\whP)$ for some $p \in (2, \infty)$, then \eqref{eq:coro:convergence-rates}  holds for $\mathcal X = S_p(\whP)$.
		
		\item \label{item:coro:P-estimate} (Switch back to $\p$) If $\p \in \RH_s(\whP)$  and $\Phi \in \cSM_r(\p)$ for some $s\in (1, \infty)$, $r \in (0, 2]$, then  \eqref{eq:coro:convergence-rates} holds for $\mathcal X = \BMO^\Phi_r(\p)$.
		
		\item \label{item:coro:Sp-P-estimate} ($S_p(\p)$-estimates) If $\p \in \RH_s(\whP)$  and $\Phi \in \cSM_r(\p) \cap S_p(\p)$ for some $s\in (1, \infty)$, $r \in (0, 2]$, $p \in [r, \infty)$, then  \eqref{eq:coro:convergence-rates} holds for $\mathcal X = S_p(\p)$.
	\end{enumerate}
\end{coro}

\begin{proof}
\eqref{item:coro:convergence-rate} The conclusions for $R(n)$ and $\ep_n$ when $\sigma=0$ are straightforward from the minimization argument for $H(\theta, \alpha, \ep, \|\tau\|_\theta)$ above, together with \eqref{eq:adapted-time-net}. For $\sigma >0$, the minimization procedure for the denominator in \eqref{eq:thm:rate-whP} is given as follows in which the constants $c_1, c_2>0$ do not depend on $n$.

For $\alpha  \in (0, \frac{3 - \theta}{2- \theta}] \supset (0, 1]$, we choose $\ep_n = \frac{1}{\sqrt{n}}$ to get that $\sup_{n \gee 1} \sqrt{n} H\big(\theta, \alpha, \frac{1}{\sqrt{n}}, \|\tau_n^\theta\|_\theta\big) <\infty$. Then, $\max\big\{\sigma\sqrt{\|\tau_n^\theta\|_\theta}, H\big(\theta, \alpha, \frac{1}{\sqrt{n}}, \|\tau_n^\theta\|_\theta\big)\big\} \sim_{c_1} \frac{1}{\sqrt{n}}$, and this yields $R(n) = \sqrt{n}$. 

For $\alpha \in (\frac{3 - \theta}{2- \theta}, 2] \subset (1, 2]$, it is easy to check that $\frac{1}{\alpha}(1- \frac{1}{2}(1- \theta)(\alpha -1)) \lee \frac{1}{2}$. Then, we choose $\ep_n = n^{-\frac{1}{\alpha}(1- \frac{1}{2}(1- \theta)(\alpha - 1))}$ to obtain that  $\max\big\{\sigma\sqrt{\|\tau_n^\theta\|_\theta}, H(\theta, \alpha, \ep_n , \|\tau_n^\theta\|_\theta)\big\} \lee  c_2 n^{-\frac{1}{\alpha}(1- \frac{1}{2}(1- \theta)(\alpha - 1))}$. Hence, $R(n) = n^{\frac{1}{\alpha}(1- \frac{1}{2}(1- \theta)(\alpha - 1))}$.

Items \eqref{item:coro:Sp-whP-estimate}, \eqref{item:coro:P-estimate} and \eqref{item:coro:Sp-P-estimate} follow directly from the respective ones in \cref{thm:approximation-BMO}.
\end{proof}

In \eqref{eq:coro:convergence-rates}, to disclose the role of the parameter $n$ in $R(n)$ we need the following:

	\begin{prop}\label{lem:cardinality-quantify}
		If \eqref{eq:small-jump-conditino} holds for some $\alpha \in (0, 2]$, then for any $\theta \in (0, 1]$ and $(\ep_n)_{n \gee 1} \subset (0, \infty)$ with $\inf_{n \gee 1} n \ep_n^\alpha >0$ there exists a constant $c_{\eqref{eq:relation-time-net}}>0$ such that for any $n \gee 1$, $\tau_n \in \cT_{\det}$ with $\#\tau_n = n+1$ one has
			\begin{align}\label{eq:relation-time-net}
			\| \# \comb{\tau_n^\theta}{\rho(\ep_n, \kappa)}\|_{L_2(\whP)} \sim_{c_{\eqref{eq:relation-time-net}}} n.
		\end{align}
	\end{prop} 
	\begin{proof}
		The idea of the proof is similar to that in \cite[Proposition 3.16]{Th20}, and hence, we omit the presentation here.
	\end{proof}

\begin{rema}[Role of the parameter $n$ in \eqref{eq:coro:convergence-rates}] The discretization time-nets employed in the approximation in \eqref{eq:coro:convergence-rates} are $\comb{\tau_n^\theta}{\rho(\ep_n, \kappa)}$, $n\gee 1$, with $\kappa = \frac{1- \theta}{2} \in [0, \frac{1}{2})$. Because of the randomness, a natural way to quantify their cardinality is to evaluate its expectation. This approach has been considered, for example, in \cite{Fu11} or \cite[Eq. (10) with $\beta =0$]{RT14}.
	\begin{enumerate}[\rm (1)]
		\item In \cref{lem:cardinality-quantify}, it is shown that the $L_2(\whP)$-norm of $\comb{\tau_n^\theta}{\rho(\ep_n, \kappa)}$ is comparable to the cardinality of the deterministic time-net $\tau_n^\theta$ for any $\ep_n$ taken case-wise from (\ref{item:coro:convergence-rate-sigma=0}, \ref{item:coro:convergence-rate-sigma>0}). It  means that the argument $n$ in the convergence rate $R(n)$ can be regarded under $\whP$ as the $2$-norm, and consequently, as any $q$-norm, $q \in [1, 2)$, of the employed time-nets. Moreover, the relation \eqref{eq:relation-time-net} indicates that the approximation with jump correction scheme using the time-nets $\comb{\tau_n^\theta}{\rho(\ep_n, \kappa)}$ does, in average, not increase the time complexity  compared to the classical Riemann approximation with the respective deterministic discretization times $\tau^\theta_n$. 
		
		\item To quantify $\comb{\tau_n^\theta}{\rho(\ep_n, \kappa)}$ under the original measure $\p$, we have the following: If the density $\od \p/\od \whP \in L_r(\whP)$ for some $r \in (1, \infty]$, then 
		 	\begin{align*}
		 	\| \# \comb{\tau_n^\theta}{\rho(\ep_n, \kappa)}\|_{L_{2(1-\frac{1}{r})}(\p)} \sim_{c} n, \quad n \gee 1,
		 \end{align*}
	 for some $c\gee 1$ not depending on $n$. This relation can be derived from using directly H\"older's inequality together with \eqref{eq:relation-time-net}. In particular, when $r=2$, then the expected number of $\comb{\tau_n^\theta}{\rho(\ep_n, \kappa)}$ under $\p$ is, up to a multiplicative positive constant, comparable to $n$.
	\end{enumerate}
\end{rema}

\subsection{Discretization for LRM strategies with jump correction method} \label{subsec:approximation}

\subsubsection{H\"older spaces and $\alpha$-stable-like measures} Before tending to applications of \cref{thm:approximation-BMO}, let us introduce in this part H\"older spaces where the payoff functions belong to and $\alpha$-stable-like measures which describes the small jump intensity of the underlying process.

\begin{defi} \label{defi:holder-stable}
	\begin{enumerate}[\rm (1)]
		\item \label{defi:Holder-spaces}  (\textit{H\"older spaces}) For $\eta \in [0, 1]$, let $C^{0, \eta}$ be the family of all Borel functions $g\colon (0, \infty) \to \R$ with $|g|_{C^{0, \eta}}<\infty$, where
		\begin{align*}
			|g|_{C^{0, \eta}} := \inf\{c\gee 0 : |g(x) - g(y)| \lee c |x-y|^\eta, \;\forall x, y >0, x\neq y\}.
		\end{align*}
		
		\item For $\alpha \in (0, 2]$, we let $\ell \in \sUS(\alpha)$ if $\ell$ is a L\'evy measure with
		\begin{align*}
			\sup_{r \in (0, 1)} r^\alpha \int_{r <|x| \lee 1} \ell(\od x) <\infty.
		\end{align*}
		\item \label{defi:alpha-stable-like} (\textit{$\alpha$-stable-like measures}) For $\alpha \in (0, 2)$ and  a L\'evy measure  $\ell$, we let $\ell \in \sS(\alpha)$ if $\ell = \ell_1 + \ell_2$ for some L\'evy measures $\ell_1, \ell_2$ which satisfy that
		\begin{align}\label{eq:item:alpha-stable}
			&\ell_1(\od x) = \frac{k(x)}{|x|^{\alpha +1}}\1_{\{x \neq 0\}} \od x \quad \mbox{and} \quad  \ell_2 \in \sUS(\alpha), 
		\end{align}
		where 
		\begin{itemize}
			\item $k \gee 0$ is right-continuous,
			
			\item $\limsup_{x \to 0} k(x) <\infty$ and  $\liminf_{x \to 0-} k(x) +  \liminf_{x \to 0+} k(x) >0$,
			
			\item $x\mapsto |x|^{-\alpha} k(x)$ is non-decreasing on $(-\infty, 0)$ and non-increasing on $(0, \infty)$.
		\end{itemize}
	\end{enumerate}
\end{defi}

\begin{rema}
	\begin{enumerate}[\rm (1)]
		\item It is obvious that $C^{0,1}$ consists of Lipschitz continuous functions on $(0, \infty)$ and $C^{0, 0}$ contains bounded Borel functions which are \textit{not} necessarily continuous.
		
		\item Evidently, any L\'evy measure belongs to $\sUS(2)$.
		
		\item  Observe that the function $k$ in \eqref{eq:item:alpha-stable} is bounded on compact intervals. Therefore, for any $\alpha \in (0, 2)$, one has $\ell_1 \in \sUS(\alpha)$, and then $\sS(\alpha) \subset \sUS(\alpha)$. Some further properties of $\sUS(\alpha)$ and $\sS(\alpha)$ are given in \cref{lemm:property-stable}.
	\end{enumerate}
\end{rema}

 We now provide some examples in the context of financial modelling using H\"older functions and $\alpha$-stable-like measures above. 

\begin{exam}
	\begin{enumerate}[\rm (1)]
		\itemsep0.3em
		\item The European call and put payoff functions are Lipschitz, hence they belong to $C^{0, 1}$.	For $\eta \in (0, 1)$, the \textit{power call} $g_\eta(y) := ((y-K)\vee 0)^\eta$ with strike $K>0$ belongs to $C^{0, \eta}$. The binary payoff function $g_0(y) : = \1_{(K, \infty)}(y)$ belongs to $C^{0, 0}$ obviously.
		
		\item The following examples are taken from \cite{KL05}:
		\begin{itemize}
			\item A \textit{CGMY process}  with parameters $C, G, M >0$ and $Y \in (0, 2)$  has the L\'evy measure $$\nu_{\mathrm{CGMY}}(\od x) = C \frac{\e^{Gx}\1_{\{x<0\}} + \e^{-Mx}\1_{\{x>0\}} }{|x|^{1+ Y}} \1_{\{x \neq 0\}} \od x.$$ 
			According to \cref{lemm:property-stable}\eqref{item:sufficient-S1}, $\nu_{\mathrm{CGMY}} \in \sS(Y)$.
			
			\item A \textit{Normal Inverse Gaussian} (NIG) process  has a L\'evy density $p_{\mathrm{NIG}}(x) : = \nu_{\mathrm{NIG}}(\od x)/\od x$ that satisfies $$\ts 0< \liminf_{|x| \to 0} x^2 p_{\mathrm{NIG}}(x) \lee \limsup_{|x|\to 0} x^2 p_{\mathrm{NIG}}(x) <\infty.$$
			Hence, \cref{lemm:property-stable}\eqref{item:sufficient-S1} verifies that $\nu_{\mathrm{NIG}} \in \sS(1)$.
			
			\item The L\'evy measure $\nu_{-}$ of a \textit{spectrally negative L\'evy process} is
			\begin{align*}
				\nu_{-}(\od x) = c\1_{\{x <0\}}  |x|^{-1-\alpha} \od x
			\end{align*}
		for some $c>0$ and $\alpha \in (0, 2)$. Then  \cref{lemm:property-stable}\eqref{item:sufficient-S1} shows that $\nu_{-} \in \sS(\alpha)$.
		\end{itemize}	
	\end{enumerate}	
\end{exam}

\subsubsection{Applications to LRM strategies}
This part provides a realization of \cref{thm:approximation-BMO} where the approximated stochastic integral is the integral term in the FS decomposition of $g(S_T)$ (see \cref{defi:FS-decomposition}). Furthermore, we choose the c\`adl\`ag version $\tilde\vartheta^g$ of the LRM strategy, which is feasible due to \cref{theo:LRM-strategy}\eqref{item:LRM-strategy-cadlag}, so that the integral we are going to approximate is of the form
\begin{align*}
\int_0^T \tilde\vartheta^g_{t-} \od S_t.
\end{align*}
Then, under the assumptions of \cref{theo:LRM-strategy}, it follows from \cref{rema:unique-cadlag-verion-LRM} that, for $t\in [0, T)$,
\begin{align}\label{eq:integrand-formula}
\tilde\vartheta^g_t = \frac{1}{\sn}\(\sigma^2 \pd_y G^*(t, S_t) + \int_{\R} \frac{G^*(t, \e^x S_t) - G^*(t, S_t)}{S_t}(\e^x-1)\nu(\od x)\) \quad \mbox{a.s.}
\end{align} 
The weight processes in this context are given by
\begin{align*}
&\Theta(\eta) : = \ts\sup_{u \in [0, \cdot]} (S^{\eta -1}_u), \quad \Psi(\eta) : = \Theta(\eta) S, \quad \Phi(\eta) : = \Psi(\eta) + \ts \sup_{u \in [0, \cdot]}|\Delta \Psi(\eta)_u|.
\end{align*}

Regarding the coefficient $\gamma_{\eqref{gammaS}}$, here we focus on the case $\gamma_{\eqref{gammaS}} \neq 0$ since the case $\gamma_{\eqref{gammaS}} = 0$, which corresponds to the martingale setting, has been investigated in \cite[Section 4]{Th20}. Recall from \cref{defi:minimal-martingale-measure} the minimal martingale measure $\mmm$ for $S$.

\begin{theo}\label{thm:approximation-convergence-rate}
Assume \cref{assum:condition-MMM}, $\gamma_{\eqref{gammaS}} \neq 0$ and $\int_{x>1} \e^{3x}\nu(\od x) <\infty$. Let $\eta \in [0, 1]$ and $g \in C^{0, \eta}$. Then the following assertions hold:
\begin{enumerate}[\rm(1)]
	\itemsep0.3em
	\item \label{item:1:weight-regularity} Both $\Psi(\eta)$ and $\Phi(\eta)$ belong to $\cSM_3(\p) \cap \cSM_{2}(\mmm)$.
	
	\item \label{item:mmm-reverse-holder} $\mmm \in \RH_{3}(\p)$ and $\|\cdot\|_{\BMO^{\Phi(\eta)}_{2}(\mmm)} \lee c \|\cdot\|_{\BMO^{\Phi(\eta)}_{2}(\p)}$ for some constant $c>0$.
	
	\item \label{item:2:growth-strategy}  In the notations of \cref{assum:change-measure}, one has $\tilde\vartheta^g \in \bfA(\theta, \Theta(\eta)|\mmm)$ for $\ttt^g$ given in \cref{theo:LRM-strategy} and for the parameter $\theta$ case-wise provided in \cref{tab:rates}.
		
	\item \label{item:3:approximation-mmm} For the adapted time-nets $\tau^\theta_n$ given in \eqref{eq:adapted-time-net} and for $\mathcal X = \BMO_2^{\Phi(\eta)}(\mmm)$, one has 
	\begin{align}\label{eq:comvergence-rate-mmm}
	\sup_{n\gee 1} R(n) \big\|E^{\corr}\big(\tilde\vartheta^g, \tau_n^\theta \,\big|\, \ep_n, \tfrac{1-\theta}{2}\big)\big\|_{\mathcal X}<\infty,
	\end{align}
	where $\theta$, $R(n)$, $\ep_n$ are case-wise provided in \cref{tab:rates}.
	
	\item\label{item:4:approximation-original} Let $s \in (1, \infty)$. Assume in addition that
	\begin{itemize}
		\item $\int_{x<-1} \e^{(1-s)x} \nu(\od x) <\infty$ and  $\ln\big( 1 + \frac{\sn}{\gamma_{\eqref{gammaS}}}\big)=: x_0 \notin \supp\nu$ when $\frac{\sn}{\gamma_{\eqref{gammaS}}} \in (-1, \infty)$;
		
		\item $\int_{x<-1} \e^{(1-s)x} \nu(\od x) <\infty$ when $\frac{\sn}{\gamma_{\eqref{gammaS}}} =-1$.
	\end{itemize}
 Then, one has $\p \in \RH_s(\mmm)$, and there is a constant  $c'\gee 1$ such that
	\begin{align*}
	\|\cdot\|_{\BMO_2^{\Phi(\eta)}(\mmm)} \sim_{c'} \|\cdot\|_{\BMO_2^{\Phi(\eta)}(\p)}.
	\end{align*}
	Consequently, \eqref{eq:comvergence-rate-mmm} holds for $\mathcal X = \BMO_2^{\Phi(\eta)}(\p)$ and for $\mathcal X = L_3(\p)$.
	
	\item \label{item:Lq-estimate-error} Under the assumptions of Item \eqref{item:4:approximation-original}, if  $\int_{x>1} \e^{px} \nu(\od x) <\infty$ for some $p \in (3, \infty)$, then \eqref{eq:comvergence-rate-mmm} holds for any $\mathcal X \in \{\BMO^{\Phi(\eta)}_{p-1}(\mmm), S_{p-1}(\mmm), \BMO^{\Phi(\eta)}_{p}(\p), S_p(\p)\}$.
\end{enumerate}

\vspace{-.5 em}
\begin{longtable}{|l|l|l|l|}
	\caption{Parameter $\theta$, convergence rate $R(n)$ and initial jump size threshold $\ep_n$}
	\label{tab:rates}
	\vspace{-.5em}\\
	\hline
	& Interplay between $g$ and $X$ & Values of  $\theta$ & $R(n)$ and $\ep_n$  \\ \hline
	\endfirsthead
	\endhead
	$\mathrm{C1}$ & \begin{tabular}[c]{@{}l@{}} 
		$\sigma =0$ and\\
		$\nu \in \sUS(\alpha)$ for some  \\
		$(\eta,\alpha)  \in ([0, 1) \times (0, 1+ \eta))$\\
		\qquad \qquad $\cup (\{1\} \times (0, 2])$
	\end{tabular}    & $\theta =1$    &      \begin{tabular}[c]{@{}l@{}}   $R(n) = 1/\ep_n = \sqrt[\alpha]{n}$  if $\alpha \in (1, 2]$,  \\
		$R(n) = n/(1+ \log n)$, $\ep_n = 1/n$\\
		\qquad  if $\alpha =1$,\\
		$R(n) = 1/\ep_n = n$ if $\alpha \in (0, 1)$   \end{tabular}   \\ \hline
	$\mathrm{C2}$ & \begin{tabular}[c]{@{}l@{}}
		$\sigma =0$ and\\
		$\nu \in \sS(\alpha)$ for some\\  $(\eta, \alpha) \in [0, 1) \times [1+ \eta, 2)$ \end{tabular}    &    \begin{tabular}[c]{@{}l@{}}
		$\forall\theta \in\Big(0, \frac{2(1+\eta)}{\alpha} -1\Big)$     \end{tabular}     &    \begin{tabular}[c]{@{}l@{}}  	$R(n) = 1/\ep_n = n^{\frac{1}{\alpha}(1- \frac{1}{2}(1- \theta)(\alpha -1))}$  \\
		\qquad if $(\eta, \alpha) \neq (0, 1)$,  \\
		$R(n) = n/(1+ \log n)$, $\ep_n = 1/n$ \\
		\qquad if $(\eta, \alpha) = (0, 1)$	
	\end{tabular} \\ \hline
	$\mathrm{C3}$ & \begin{tabular}[c]{@{}l@{}}  	$\sigma >0$ and $\eta =1$
	\end{tabular} & $\theta = 1$ & $R(n) = 1/\ep_n =  \sqrt{n}$
	\\ \hline 
	$\mathrm{C4}$ &	\begin{tabular}[c]{@{}l@{}}  	$\sigma >0$, $\eta \in (0, 1)$ and\\
		$\nu \in \sUS(\alpha)$ \\ \quad for some 
		$\alpha \in (0, 2]$ 
	\end{tabular} & $\forall \theta \in (0, \eta)$ &  \begin{tabular}[c]{@{}l@{}}  	$R(n) = 1/\ep_n = \sqrt{n}$ if $\alpha \in (0, \frac{3 - \theta}{2- \theta}]$,\\[2pt]
		$R(n) = 1/\ep_n = n^{\frac{1}{\alpha}(1- \frac{1}{2}(1- \theta)(\alpha -1))}$\\
		\qquad  if $\alpha \in (\frac{3- \theta}{2- \theta}, 2]$
	\end{tabular} 	\\ \hline 
\end{longtable}
\end{theo}
\vspace{-.9 em}

\begin{rema}
	The parameter $\theta$, which describes the growth property of $\ttt^g$, is the outcome of the interplay between the small jump activity of $X$ and the H\"older regularity of the payoff function. 
\end{rema}

For the proof of \cref{thm:approximation-convergence-rate}, we need the following lemma.

\begin{lemm} \label{rema:jump-intensity-relation} Under \cref{assum:condition-MMM}, the following assertions hold:
	\begin{enumerate}[\rm (1)]
		\item \label{rema:small-jump-intensity-relation}  For $\alpha \in (0, 2]$, one has $\nu \in \sUS(\alpha)$ if and only if $\nu^* \in \sUS(\alpha)$.
		
		\item \label{lemm:nu*-stable-like-process} For $\alpha \in (0, 2)$, one has $\nu \in \sS(\alpha)$ if and only if $\nu^* \in \sS(\alpha)$.
		
		\item \label{rema:big-jump-intensity-relation} Assume $\gamma_{\eqref{gammaS}} \neq 0$. Then, for $r \in [1, \infty)$ one has
		\begin{align*}
		&\E \e^{rX_t} <\infty, \forall t>0  \Leftrightarrow \int_{|x|>1} \e^{rx} \nu(\od x) <\infty \Leftrightarrow \int_{x>1} \e^{rx} \nu(\od x) <\infty \\
		& \Leftrightarrow \int_{x>1} \e^{(r-1)x} \nu^*(\od x)<\infty \Leftrightarrow \int_{|x|>1} \e^{(r-1)x} \nu^*(\od x)<\infty \Leftrightarrow \mme \e^{(r-1)X_t} <\infty, \forall t>0.
		\end{align*}
	\end{enumerate}
\end{lemm}

\begin{proof} We recall $\nu^*(\od x) = \big(1 - \frac{\gamma_{\eqref{gammaS}} }{\sn}(\e^x -1)\big)\nu(\od x) =: (\frac{\od \nu^*}{\od \nu})(x) \nu(\od x)$ from \eqref{eq:Levy-measure-MMM} and the set $A$ from \cref{assum:condition-MMM}.
	
	\medskip
	
	\eqref{rema:small-jump-intensity-relation} follows directly from the fact that $\lim_{x \to 0} (\frac{\od \nu^*}{\od \nu})(x) = 1$ and $\sup_{x \in (0, 1]} (\frac{\od \nu^*}{\od \nu})(x) <\infty$.
	
	\medskip
	
	\eqref{lemm:nu*-stable-like-process} Assume that $\nu = \nu_1 + \nu_2 \in \sS(\alpha)$, where $\nu_1$, $\nu_2$ are L\'evy measures which satisfy \eqref{eq:item:alpha-stable} for $\ell_i = \nu_i$, $i =1, 2$. We define
	\begin{align*}
		\nu_1^*(\od x) := \begin{cases}
			\big(\big(1- \frac{\gamma_{\eqref{gammaS}}}{\sn}(\e^x -1)\big)\1_{\{x<0\}} + \1_{\{x \gee 0\}}\big) \1_{\{x \in A\}} \nu_1(\od x) & \mbox{if } \frac{\gamma_{\eqref{gammaS}}}{\sn} \lee 0\\[5pt]
			\big(\1_{\{x<0\}} + \big(1- \frac{\gamma_{\eqref{gammaS}}}{\sn}(\e^x -1)\big)\1_{\{x \gee 0\}}\big)\1_{\{x \in A\}} \nu_1(\od x) & \mbox{if } \frac{\gamma_{\eqref{gammaS}}}{\sn} > 0,
		\end{cases}
	\end{align*}
	and set
	\begin{align*}
		\nu^*_2(\od x) & : = \nu^*(\od x) - \nu_1^*(\od x) = \1_{\{x \in A\}} \nu^*(\od x) - \nu_1^*(\od x) \\
		& = \begin{cases}
			- \frac{\gamma_{\eqref{gammaS}}}{\sn}(\e^x -1)\1_{\{x \gee 0,\,x\in A\}}\nu_1(\od x) + \big(1- \frac{\gamma_{\eqref{gammaS}}}{\sn}(\e^x -1)\big) \1_{\{x \in A\}} \nu_2(\od x) & \mbox{if } \frac{\gamma_{\eqref{gammaS}}}{\sn} \lee 0\\[5pt]
			\frac{\gamma_{\eqref{gammaS}}}{\sn}(1-\e^x)\1_{\{x<0,\,x\in A\}}\nu_1(\od x) + \big(1- \frac{\gamma_{\eqref{gammaS}}}{\sn}(\e^x -1)\big)\1_{\{x \in A\}} \nu_2(\od x) & \mbox{if } \frac{\gamma_{\eqref{gammaS}}}{\sn} > 0.
		\end{cases}
	\end{align*}
	One remarks that $\int_{|x| >1} \e^{x} \nu_i (\od x) <\infty$ by the integrability assumption in \Cref{subsec-setting-5}. Then, it is straightforward to check that both $\nu^*_1$ and $\nu^*_2$ are L\'evy measures. Moreover, since the origin is an interior point of $A$, a short calculation shows that $\nu^*_1$ and $\nu^*_2$ satisfy \eqref{eq:item:alpha-stable} for $\ell_i = \nu^*_i$. Hence, $\nu^* \in \sS(\alpha)$.
	
	For the converse implication, we assume $\nu^* = \nu^*_1 + \nu^*_2 \in \sS(\alpha)$. Then, for 
		\begin{align*}
		\nu_1(\od x) := \begin{cases}
			\big( \1_{\{x < 0\}} + \big(1- \frac{\gamma_{\eqref{gammaS}}}{\sn}(\e^x -1)\big)^{-1}\1_{\{x\gee 0\}}  \big) \1_{\{x \in A\}} \nu^*_1(\od x) & \mbox{if } \frac{\gamma_{\eqref{gammaS}}}{\sn} \lee 0\\[5pt]
			\big(\big(1- \frac{\gamma_{\eqref{gammaS}}}{\sn}(\e^x -1)\big)^{-1} \1_{\{x < 0\}} + \1_{\{x\gee 0\}}\big)\1_{\{x \in A\}} \nu^*_1(\od x) & \mbox{if } \frac{\gamma_{\eqref{gammaS}}}{\sn} > 0,
		\end{cases}
	\end{align*}
and $\nu_2(\od x) := \nu(\od x) - \nu_1(\od x)$, an analogous argument as in the first implication yields $\nu = \nu_1 + \nu_2 \in \sS(\alpha)$.
	
	\medskip
	
\eqref{rema:big-jump-intensity-relation} follows from combining \cite[Theorem 25.3]{Sa13} with the relation between $\nu$ and $\nu^*$.
\end{proof}

\begin{proof}[Proof of \cref{thm:approximation-convergence-rate}]
Recall $(X|\mmm) \sim (\gamma^*, \sigma, \nu^*)$ from \eqref{eq:Levy-measure-MMM} and  $A$ from \cref{assum:condition-MMM}.

\medskip 

\eqref{item:1:weight-regularity}  Combining \cref{rema:jump-intensity-relation}\eqref{rema:big-jump-intensity-relation} with \cref{prop:regularity-weight-process}, we obtain that $\Psi(\eta) \in \cSM_3(\p) \cap \cSM_{2}(\mmm)$. Thanks to \cref{lemm:olderline-Phi}, one has $\Phi(\eta) \in \cSM_3(\p) \cap \cSM_{2}(\mmm)$.  

\medskip 

\eqref{item:mmm-reverse-holder} We recall $\cE(U)$ from \cref{defi:minimal-martingale-measure} and notice that $\cE(U)>0$ by \cref{assum:condition-MMM}. According to  \Cref{subsec:exponential-levy}, there is a L\'evy process $V$ with $(V|\p)\sim (\gamma_V, \sigma, \nu_V)$ such that $\cE(U) = \e^V$. Due to \eqref{eq:characteristics-U}, by letting $h(x) : = \ln(1+x)$ for $x>-1$ one has 
\begin{align}\label{eq:form-nuV}
\nu_V = \nu_U \circ h^{-1} = (\nu \circ \alpha^{-1}_U)\circ h^{-1} = \nu \circ(h \circ \alpha_U)^{-1}.
\end{align}
Since $h(\alpha_U(x)) = \ln\big(1-\frac{\gamma_{\eqref{gammaS}}(\e^x -1)}{\sn}\big)$ for $x\in A$, there is an $\ep_{\eqref{eq:exists-epsilon}}>0$ such that 
\begin{align}\label{eq:exists-epsilon}
\left\{x \in A : |h(\alpha_U(x))|>1\right\} \subseteq A \backslash(-\ep_{\eqref{eq:exists-epsilon}}, \ep_{\eqref{eq:exists-epsilon}}).
\end{align}
Then, the assumption $\int_{|x|>1} \e^{3x} \nu(\od x) <\infty$ implies that
 \begin{align*}
\int_{|x|>1} \e^{3x} \nu_V(\od x) &= \int_{|h(\alpha_U(x))|>1} \e^{3(h(\alpha_U(x)))} \nu(\od x)\\
&  \lee  \int_{A \backslash (-\ep_{\eqref{eq:exists-epsilon}}, \ep_{\eqref{eq:exists-epsilon}})} \(1- \frac{\gamma_{\eqref{gammaS}}(\e^x -1)}{\sn}\)^3 \nu(\od x) <\infty.
\end{align*}
Let $(V|\p) \sim \psi_V$. Since $(\e^{3V_t + t \psi_V(-3\im)})_{t \in [0, T]}$ is a c\`adl\`ag martingale, it follows from the optional stopping theorem that for any stopping time $\rho \colon \Omega \to [0, T]$, a.s., 
\begin{align*}
\bce{\F_\rho}{\e^{3V_T}} & = \e^{-T \psi_V(-3\im)}\bce{\F_\rho}{\e^{3V_T + T \psi_V(-3\im)}} =  \e^{-T\psi_V(-3\im )} \e^{3V_\rho + \rho \psi_V(-3\im)}\\
& \lee \e^{T|\psi_V(-3\im)|} \e^{3 V_\rho} = \e^{T|\psi_V(-3\im)|} \left|\bce{\F_\rho}{\e^{V_T}}\right|^3,
\end{align*}
where we use the martingale property of $\e^V$ for the last equality. According to \cref{defi:RH-space},  $\od \mmm = \e^{V_T} \od \p \in \RH_3(\p)$. Hence,  \cref{prop:feature-BMO}\eqref{item:change-of-measure} yields $\|\cdot\|_{\BMO^{\Phi(\eta)}_{2}(\mmm)} \lee c \|\cdot\|_{\BMO^{\Phi(\eta)}_{2}(\p)}$.

\medskip

\eqref{item:2:growth-strategy} Since the function $g \in C^{0, \eta}$ has at most linear growth at infinity and $\int_{|x|>1} \e^{3x} \nu(\od x) <\infty$, which is equivalent to $\int_{|x|>1} \e^{2x} \nu^*(\od x) <\infty$ by \cref{rema:jump-intensity-relation}\eqref{rema:big-jump-intensity-relation}, the assumptions of \cref{theo:LRM-strategy} are satisfied so that \eqref{eq:integrand-formula} is applicable. We now verify \cref{assum:change-measure} with $\whP = \mmm$.

\smallskip

$\bullet$ \underline{For Item} \eqref{item:semimartingale-Z-hat-P}, we notice that $S$ is an $L_2(\mmm)$-martingale. Hence, the stochastic logarithm $Z$ appeared in \eqref{eq:stoch-logarithm} is a L\'evy process under $\mmm$ and is also an $L_2(\mmm)$-martingale. Denote 
$$(Z|\mmm) \sim (\gamma_Z^*, \sigma, \nu^*_Z).$$ Then, Item \eqref{item:semimartingale-Z-hat-P} is  satisfied for $\bar Z_0 = 0$, $\bar b^Z = 0$, $W^{\mmm} = W^*$ and $\pi^{\mmm} = \Leb\otimes \nu^*_Z$.

\smallskip

$\bullet$ \underline{For Item} \eqref{item:growth-condition}, we aim to apply \cref{theo:envelop-MVH-strategy} with the selections 
$$\bQ = \mmm \quad \mbox{and} \quad \ell = \nu$$
so that, for any $t \in [0, T)$, \eqref{eq:integrand-formula} gives
\begin{align*}
	|\ttt^g_t| = (\sn)^{-1} |\varGamma_\nu^{\mmm}(T-t, S_t)| \quad \mbox{a.s.}
\end{align*}
 Let us examine each case in \cref{tab:rates}. One recalls that, thanks to \cref{rema:jump-intensity-relation}, $\nu \in \sUS(\alpha) \Leftrightarrow \nu^* \in \sUS(\alpha)$ and $\nu \in \sS(\alpha) \Leftrightarrow \nu^* \in \sS(\alpha)$.\\ 
  \underline{For $\mathrm{C1}$}, the given range of $(\eta, \alpha)$ yields $\int_{|x| \lee 1} |x|^{1+ \eta} \nu(\od x) <\infty$. Thus, \cref{theo:envelop-MVH-strategy}\eqref{item:2:sigma=0-envelope} gives
 \begin{align*}
 	|\ttt^g_t| \lee (\sn)^{-1} c_{\eqref{eq:thm:Holder-case-estimate:psi-process}} S_t^{\eta -1} \lee (\sn)^{-1} c_{\eqref{eq:thm:Holder-case-estimate:psi-process}} \Theta(\eta)_t \quad \mbox{a.s.,}
 \end{align*}
which shows that \eqref{item:growth-condition} holds with $\theta = 1$ and $\Theta = \Theta(\eta)$.\\
\underline{For $\mathrm{C2}$}, when $(\eta, \alpha) \in [0, 1) \times (1 + \eta, 2)$ (resp. $(\eta, \alpha) \in [0, 1) \times \{1 + \eta\}$) we apply \cref{theo:envelop-MVH-strategy}\eqref{item:2.2:sigma=0-envelope} to find that \eqref{eq:upper-weight-function} is satisfied for $\Theta = \Theta(\eta)$ and $w(t) = (T-t)^{\frac{1+ \eta}{\alpha} - 1}$ (resp. $w(t) = \max\{1, \log\frac{1}{T-t}\}$). Hence, applying \cref{rema:parameter-theta}\eqref{item:3:rema:parameter-theta}  yields $\ttt^g \in \cap_{0<\theta<\frac{2(1+\eta)}{\alpha}-1} \bfA(\theta, \Theta(\eta)|\mmm)$ for any $(\eta, \alpha) \in [0, 1) \times [1 + \eta, 2)$.\\
\underline{For $\mathrm{C3}$}, we use \cref{theo:envelop-MVH-strategy}\eqref{item:1:sigma>0-envelope} to get $\theta = 1$ and $\Theta = \Theta(\eta)$.\\
\underline{For $\mathrm{C4}$}, we apply \cref{theo:envelop-MVH-strategy}\eqref{item:1:sigma>0-envelope} again to find that \eqref{eq:upper-weight-function} is satisfied for $\Theta = \Theta(\eta)$ and $w(t) = (T-t)^{\frac{\eta-1}{2}}$. Then,  \cref{rema:parameter-theta}\eqref{item:3:rema:parameter-theta} shows $\ttt^g \in \cap_{0<\theta< \eta} \bfA(\theta, \Theta(\eta)|\mmm)$.

\smallskip

$\bullet$ \underline{For Item} \eqref{item:equivalent-martingale-measure}, $\ttt S$ is a $\mmm$-martingale because of \cref{theo:LRM-strategy}\eqref{item:LRM-strategy-cadlag}. The square $\mmm$-integrability of $\ttt S$ follows from the growth property of $\ttt$ in Item \eqref{item:growth-condition} and $\Phi(\eta) \in \cSM_2(\mmm)$.

\smallskip

$\bullet$ \underline{Item} \eqref{item:condition-stoch-integral} follows from combining the growth property of $\ttt$ with $\Phi(\eta) \in \cSM_3(\p)\subset \cSM_2(\p)$.

\medskip

\eqref{item:3:approximation-mmm}  Our purpose is to apply \cref{coro:convergence-rate} with $\whP = \mmm$, $\Theta = \Theta(\eta)$ and $\Phi = \Phi(\eta)$. We check the assumptions of \cref{thm:approximation-BMO}. By definition, it is clear that 
\begin{align*}
	\max\{\Theta(\eta)_{-} S_{-}, \Theta(\eta) S\} \lee \Phi(\eta)
\end{align*}
and that $\Phi(\eta) \in \cSM_2(\mmm)$ according to Item \eqref{item:1:weight-regularity}. Condition \eqref{eq:small-jump-conditino} in this context simply means that $\nu^*_Z \in \sUS(\alpha)$. Since $\nu_Z^* = \nu^* \circ \tilde h^{-1}$ for $\tilde h(x): = \e^x -1$, it is straightforward that 
$$\nu^*_Z \in \sUS(\alpha) \Leftrightarrow \nu^* \in \sUS(\alpha) \Leftrightarrow \nu \in \sUS(\alpha),$$
 where the second equivalence is  due to   \cref{rema:jump-intensity-relation}.  In \cref{tab:rates}, the conclusions for $R(n)$ and $\ep_n$ in the \underline{cases $\mathrm{C1}$ and $\mathrm{C2}$} (resp. \underline{cases $\mathrm{C3}$ and $\mathrm{C4}$}) follow directly from \eqref{eq:coro:convergence-rates} and \eqref{item:coro:convergence-rate-sigma=0} (resp. \eqref{item:coro:convergence-rate-sigma>0}).

\medskip

\eqref{item:4:approximation-original} \textit{Step 1}. For $\nu_V$ given in \eqref{eq:form-nuV}, we first show that  $\int_{|x|>1} \e^{(1-s)x} \nu_V(\od x) <\infty$. Indeed,
\begin{align}\label{eq:itegrability-nu-minusV}
 \int_{|x|>1} \e^{(1-s)x} \nu_V(\od x) & = \int_{|h(\alpha_U(x))|>1} \(1- \frac{\gamma_{\eqref{gammaS}}(\e^x -1)}{\sn}\)^{1-s} \nu(\od x) \notag \\
 &\lee \int_{A \backslash (-\ep_{\eqref{eq:exists-epsilon}}, \ep_{\eqref{eq:exists-epsilon}})} \(1- \frac{\gamma_{\eqref{gammaS}}(\e^x -1)}{\sn}\)^{1-s} \nu(\od x) =: \mathrm{I}_{\eqref{eq:itegrability-nu-minusV}}.
\end{align}
We consider three cases for $\frac{\sn}{\gamma_{\eqref{gammaS}}}$ as follows:\\
\underline{Case 1}: $\frac{\sn}{\gamma_{\eqref{gammaS}}}>-1$. Since $x_0 \notin \supp\nu$ by assumption, it implies that $\nu((x_0 - \ep_0, x_0 + \ep_0)) =0$ for some $\ep_0>0$. Moreover, one has $1- \frac{\gamma_{\eqref{gammaS}}(\e^x -1)}{\sn} \gee |x-x_0| \frac{|\gamma_{\eqref{gammaS}}|}{\sn} \e^{x\wedge x_0}$ for all $x\in A$ by the mean value theorem. Hence,
\begin{align*}
\mathrm{I}_{\eqref{eq:itegrability-nu-minusV}} & = \int_{|x-x_0| \gee \ep_0,\,x \in A \backslash (-\ep_{\eqref{eq:exists-epsilon}}, \ep_{\eqref{eq:exists-epsilon}})} \(1- \frac{\gamma_{\eqref{gammaS}}(\e^x -1)}{\sn}\)^{1-s} \nu(\od x),\\
& \lee \ep_0^{1-s} \frac{|\gamma_{\eqref{gammaS}}|^{1-s}}{\sn^{1-s}} \int_{|x-x_0|\gee \ep_0,\,x \in A \backslash (-\ep_{\eqref{eq:exists-epsilon}}, \ep_{\eqref{eq:exists-epsilon}})} \e^{(1-s)(x\wedge x_0)} \nu(\od x)  \\
& \lee \ep_0^{1-s} \frac{|\gamma_{\eqref{gammaS}}|^{1-s}}{\sn^{1-s}} \int_{A \backslash (-\ep_{\eqref{eq:exists-epsilon}}, \ep_{\eqref{eq:exists-epsilon}})} \e^{(1-s)(x\wedge x_0)} \nu(\od x)  <\infty,
\end{align*}
where the finiteness is due to the assumption $\int_{x < -1} \e^{(1-s)x} \nu(\od x) <\infty$.\\
\underline{Case 2}: $\frac{\sn}{\gamma_{\eqref{gammaS}}}=-1$. We have 
$\mathrm{I}_{\eqref{eq:itegrability-nu-minusV}} =  \int_{A \backslash (-\ep_{\eqref{eq:exists-epsilon}}, \ep_{\eqref{eq:exists-epsilon}})} \e^{(1-s)x} \nu(\od x)<\infty.$\\
\underline{Case 3}: $\frac{\sn}{\gamma_{\eqref{gammaS}}}<-1$. In this case one has $\gamma_{\eqref{gammaS}} <0$, which implies that $\inf_{x \in \R}\big(1 - \frac{\gamma_{\eqref{gammaS}}(\e^x-1)}{\sn}\big) = 1 + \frac{\gamma_{\eqref{gammaS}}}{\sn} >0$. Hence,
\begin{align*}
\mathrm{I}_{\eqref{eq:itegrability-nu-minusV}} \lee \(1 + \frac{\gamma_{\eqref{gammaS}}}{\sn}\)^{1-s} \int_{A \backslash (-\ep_{\eqref{eq:exists-epsilon}}, \ep_{\eqref{eq:exists-epsilon}})} \nu(\od x) <\infty.
\end{align*}
We conclude from three cases above that $ \int_{|x|>1} \e^{(1-s)x} \nu_V(\od x) <\infty$, or equivalently,
$$\e^{-t \psi_V((s-1)\im)} = \E \e^{(1-s)V_t} <\infty, \quad t>0.$$

  \textit{Step 2}.  We show $\p \in \RH_s(\mmm)$. Indeed, since $\od \p  = \e^{-V_T} \od \mmm$ and $\e^V = \cE(U)$ is a $\p$-martingale, it holds that $\e^{-V}$ is a $\mmm$-martingale. We have for any $t\in [0, T]$ that, a.s.,
\begin{align*}
\bcem{\F_t}{\e^{s(-V_T)}} = \e^{-V_t} \bce{\F_t}{\e^{-sV_T} \e^{V_T}} = \e^{-V_t} \bce{\F_t}{\e^{(1-s)V_T}} \lee  \e^{T|\psi_V((s-1)\im)|} \e^{-s V_t}.
\end{align*}
By a standard approximation argument, we infer that 
\begin{align*}
\bcem{\F_\rho}{\e^{s(-V_T)}} \lee \e^{T|\psi_V((s-1)\im)|} \e^{-s V_\rho}  = \e^{T|\psi_V((s-1)\im)|} \big|\bcem{\F_\rho}{\e^{-V_T}}\big|^s \quad \mbox{a.s.}
\end{align*}
for any stopping times $\rho \colon \Omega \to [0, T]$. Therefore, $\p \in \RH_s(\mmm)$. 

\textit{Step 3}. Thanks to \textit{Step 2} and Items  \eqref{item:1:weight-regularity},  \eqref{item:mmm-reverse-holder}, we apply \cref{prop:feature-BMO}\eqref{item:change-of-measure} with $\bQ = \mmm$ and $p =2$  to obtain  $$\|\cdot\|_{\BMO_2^{\Phi(\eta)}(\mmm)} \sim_c \|\cdot\|_{\BMO_2^{\Phi(\eta)}(\p)}.$$
The ``Consequently'' part is straightforward because of $\Phi(\eta) \in \cSM_3(\p)$ and \cref{prop:feature-BMO}\eqref{item:Lp-BMO-estimate}.

\medskip

\eqref{item:Lq-estimate-error} An analogous argument as in the proof of Item \eqref{item:1:weight-regularity} shows that both $\Psi(\eta)$ and $\Phi(\eta)$ belong to $\cSM_p(\p) \cap \cSM_{p-1}(\mmm)$. We now apply \cref{prop:feature-BMO}\eqref{item:equivalent-BMO} with $r =2$ and then combine this with Item \eqref{item:4:approximation-original} to derive the assertion.
\end{proof}

\appendix

\section{Proof of \cref{thm:approximation-BMO}, Item \eqref{item:thm:rate-whP}}\label{appen:proof:thm:approximation-BMO}

In this proof, we often use the fact that a \textit{c\`agl\`ad} (left-continuous with right limits) function has countably many discontinuities which then implies that when integrating such a function with respect to the Lebesgue measure we may use its \textit{c\`adl\`ag} version without changing the value of the integral.  The proof is divided in the following steps.

\smallskip

\textit{\textbf{\textlabel{Step 1}{thm:step-1}:}} \textit{Growth in time and no fixed-time discontinuities of $\tilde\vartheta$}. For $0\lee a  < T$, a.s., 
\begin{align}\label{eq:estimate-variance-vartheta}
	|\ttt_a| & \lee \begin{cases}
		c_{\eqref{eq:theta-1}} \Theta_a & \mbox{ if } \theta =1\\
		c_{\eqref{item:growth-varphi}}  (T-a)^{\frac{\theta-1}{2}}\Theta_a & \mbox{ if } \theta \in (0, 1)
	\end{cases} 
	=: c_{\eqref{eq:estimate-variance-vartheta}}(T-a)^{\frac{\theta-1}{2}} \Theta_a \quad \mbox{for } \theta \in (0, 1],
\end{align}
where the case $\theta \in (0,1)$ holds because of \cref{rema:parameter-theta}\eqref{item:1:rema:parameter-theta} (here, by assumption, $M = \tilde\vartheta S$ is an $L_2(\whP)$-martingale), and where $c_{\eqref{eq:estimate-variance-vartheta}} : = c_{\eqref{eq:theta-1}}\vee c_{\eqref{item:growth-varphi}}$. Moreover, it follows from the monotonicity of $\Theta$ and the c\`adl\`ag property of $\ttt, \Theta$ that
\begin{align}\label{eq:path-variation}
	|\ttt_t - \ttt_a| \lee 2 c_{\eqref{eq:estimate-variance-vartheta}}(T-t)^{\frac{\theta-1}{2}} \Theta_t \quad \mbox{for all } 0\lee a <t <T \mbox{ a.s.}
\end{align} 
On the other hand, since $\ttt S$ is a martingale adapted to the natural completed filtration of a L\'evy process, which is a quasi-left continuous filtration (see, e.g., \cite[p.150, Exercise 9]{Pr05}), it follows from \cite[p.190]{Pr05} that $\ttt S$ has no fixed-time of discontinuities, i.e., $\ttt_t S_t = \ttt_{t-} S_{t-}$ a.s. for $t \in [0, T)$. Since $S_{-} >0$ and $S_t = S_{t-}$ a.s., we infer that $\ttt_{t} = \ttt_{t-}$ a.s. for $t \in [0, T)$. Hence, $\ttt$ has no fixed-time of discontinuities. As a consequence, for any $\tau = (t_i)_{i=0}^n \in \cT_{\det}$ one has, a.s., $\ttt^\tau_u = \sum_{i=1}^n \ttt_{t_{i-1}}\1_{(t_{i-1}, t_i]}(u)$, $\forall u \in [0, T]$.

\smallskip

\textbf{\textit{Step 2:}} \textit{One-step approximation}. Let $0 \lee a < t <T$ arbitrarily. Since $S$ is a $\whP$-semimartingale which satisfies the SDE $\od S_t = S_{t-} \od Z_t$, using a Gronwall argument as in \cite[Lemma 5.1]{Th20} yields constants $c_{\eqref{eq:estimate-S}}, d_{\eqref{eq:estimate-S}}>0$ independent of $a, t$ such that, a.s., 
\begin{align}\label{eq:estimate-S}
	\hbbce{\F_a}{\int_{(a, t]} S_u^2 \od u} \lee c_{\eqref{eq:estimate-S}}^2 (t-a) S_a^2 \quad \mbox{and} \quad  \hbbce{\F_a}{\int_{(a, t]} |S_u- S_a|^2 \od u} \lee d_{\eqref{eq:estimate-S}}^2 (t-a)^2 S_a^2.
\end{align}
Recall that $M = \tilde\vartheta S$ is an $L_2(\whP)$-martingale. Then, applying the triangle inequality and Fubini's theorem we obtain, a.s., 
\begin{align}\label{eq:one-step-approximation}
	& \hbbce{\F_a}{\int_{(a, t]} |\tilde\vartheta_u - \tilde\vartheta_a|^2 S_u^2 \od u} \lee 2 \hbbce{\F_a}{\int_{(a, t]} |M_u - M_a|^2 \od u} + 2 \ttt_a^2\, \hbbce{\F_a}{\int_{(a, t]} |S_u - S_a|^2\od u} \notag\\
	& = 2\hbbce{\F_a}{\int_{(a, t]} (t-u) \od \<M\>^{\whP}_u} + 2 d_{\eqref{eq:estimate-S}}^2 (t-a)^2  \ttt_a^2  S_a^2 \notag\\
	& \lee 2 \frac{t-a}{(T-a)^{1- \theta}}\bigg[ \hbbce{\F_a}{\int_{(a, t]} (T-u)^{1- \theta} \od \<M\>^{\whP}_u} + c_{\eqref{eq:one-step-approximation}}^2 (t-a) \Theta_a^2 S_a^2\bigg],
\end{align}
where $c_{\eqref{eq:one-step-approximation}}:= c_{\eqref{eq:estimate-variance-vartheta}} d_{\eqref{eq:estimate-S}}$.

\smallskip

\textbf{\textit{Step 3:}} \textit{Multi-step approximation}. For $\tau = (t_i)_{i=0}^n \in \cT_{\det}$ and $t\in [0, T]$, we set
$$\<\tilde \vartheta, \tau\>_t: =  \int_0^t   \bigg|\tilde\vartheta_u - \sum_{i=1}^n\tilde\vartheta_{t_{i-1}} \1_{(t_{i-1}, t_i]}(u)\bigg|^2 S_u^2 \od u \stackrel{\textrm{a.s.}}{=}  \int_0^t   |\tilde\vartheta_u - \ttt^\tau_u|^2 S_u^2 \od u.$$ Then there exists a constant $c_{\eqref{eq:estimate:bmo}}>0$ such that for any $\tau \in \cT_{\det}$ and any $a \in [0, T]$,
\begin{align}\label{eq:estimate:bmo}
	\hbce{\F_a}{\<\tilde\vartheta, \tau\>_T - \<\tilde\vartheta, \tau\>_a} \lee c_{\eqref{eq:estimate:bmo}}^2 \|\tau\|_\theta \Phi_a^2 \quad \mbox{a.s.}
\end{align} 
Indeed, for $\tau = (t_i)_{i=0}^n \in \cT_{\det}$ and $a \in [t_{k-1}, t_k)$, $k = 1, \ldots, n$, we let $s_i: = a \vee t_i$, $i = k-1, \ldots, n$. Then, applying \eqref{eq:one-step-approximation} with the aid of \cref{rema:parameter-theta}\eqref{item:1:rema:parameter-theta} yields, a.s., 
\begin{align*}
	& \hbce{\F_a}{\<\tilde\vartheta, \tau\>_T - \<\tilde\vartheta, \tau\>_a} 
	\lee 2\hbbce{\F_a}{|\tilde\vartheta_{a} - \tilde\vartheta_{t_{k-1}}|^2\int_{(a, t_k]} S_u^2 \od u} + 2\sum_{i=k}^n \hbbce{\F_a}{\int_{(s_{i-1}, s_i]}|\tilde\vartheta_{u} - \tilde\vartheta_{a}|^2 S_u^2 \od u}\\
	& \lee 8 c_{\eqref{eq:estimate-variance-vartheta}}^2  c_{\eqref{eq:estimate-S}}^2 \frac{t_k -a}{(T-a)^{\theta-1}} \Phi_a^2  + 4 \|\tau\|_\theta \, \hbbce{\F_a}{\int_{(a, T)}(T-u)^{1- \theta} \od \<M\>^{\whP}_u + c_{\eqref{eq:one-step-approximation}}^2\sum_{i=k}^n (s_i - s_{i-1})\Phi_{s_{i-1}}^2} \\
	& \lee c_{\eqref{eq:estimate:bmo}}^2 \|\tau\|_\theta \Phi_a^2,
\end{align*}
where $c_{\eqref{eq:estimate:bmo}}>0$ is determined by $c_{\eqref{eq:estimate:bmo}}^2 = 8 c_{\eqref{eq:estimate-variance-vartheta}}^2 c_{\eqref{eq:estimate-S}}^2 + 4c_{\eqref{item:curvature-condition} }^2 + 4Tc_{\eqref{eq:one-step-approximation}}^2 \|\Phi\|_{\cSM_2(\whP)}^2$.

\smallskip

\textit{\textbf{Step 4:}} \textit{Examine $E^{\corr}(\tilde \vartheta, \tau|\ep, \kappa)$}. Since
\begin{align*}
	\int_0^T\!\!\int_{|z| > \ep(T-u)^\kappa} |z| \pi^{\whP}_Z(\od z, \od u) \lee \frac{1}{\ep}  \int_0^T (T-u)^{-\kappa}\int_{\R_0} z^2 \nu^{\whP}_Z(u, \od z) \od u \lee \frac{\nu_{\eqref{eq:2nd-moment-hat-mu-Z}} T^{1- \kappa}}{\ep (1- \kappa)}  < \infty,  
\end{align*}
the jump part of $Z$ can be decomposed under $\whP$ as\footnote{For a semimartingale $Y$, a random measure $\Pi$, and suitable integrands $\vartheta$ and $f$, we denote by $\vartheta \cdot Y = ((\vartheta \cdot Y)_t)_{t \in [0, T]}$ and by $f\cdot \Pi = ((f\cdot \Pi)_t)_{t \in [0, T]}$ the integral processes which are c\`adl\`ag and start at zero with 
	$$(\vartheta \cdot Y)_t  : = \int_{(0, t]} \vartheta_u \od Y_u, \quad (f\cdot \Pi)_t:  = \int_{(0, t] \times \R_0} f(z,s) \Pi(\od z, \od s), \quad t \in (0, T].$$
	In particular, if the integrator is induced by the Lebesgue measure, then $\vartheta \cdot \Leb = ((\vartheta \cdot \Leb)_t)_{t\in [0, T]}$ denotes a continuous process given by $(\vartheta \cdot \Leb)_t : = \int_0^t \vartheta_u \od u$.}
\begin{align*}
	z \cdot (N_Z - \pi^{\whP}_Z) & =  (\1_{\{0 <|z| \lee \ep (T-u)^\kappa\}} z)\cdot (N_Z - \pi^{\whP}_Z)  + (\1_{\{|z| > \ep (T-u)^\kappa\}} z)\cdot  N_Z  - (\1_{\{|z| > \ep (T-u)^\kappa\}} z)\cdot  \pi^{\whP}_Z\\
	&=: \ol Z^{1,\ep, \kappa} +  Z^{2, \ep, \kappa} - \bar \gamma^{\ep, \kappa}
\end{align*}
so that
\begin{align*}
	Z = \bar Z_0 + \bar b^Z \cdot \Leb + \sigma  W^{\whP} + \ol Z^{1,\ep, \kappa} + Z^{2, \ep, \kappa} - \bar \gamma^{\ep, \kappa}.
\end{align*}
Then we have the following decomposition 
\begin{align*}
	(\tilde \vartheta_{-} - \tilde \vartheta^\tau)\cdot S  & = ((\tilde \vartheta_{-} - \tilde \vartheta^\tau)S_{-}) \cdot Z  \\
	& = ((\tilde \vartheta_{-} - \tilde \vartheta^\tau)S_{-} \bar b^Z) \cdot \Leb + (\sigma (\tilde \vartheta_{-} - \tilde \vartheta^\tau)S_{-}) \cdot W^{\whP}  \\
	& \quad + ((\tilde \vartheta_{-} - \tilde \vartheta^\tau)S_{-})\cdot \ol Z^{1, \ep, \kappa} + ((\tilde \vartheta_{-} - \tilde \vartheta^\tau)S_{-}) \cdot Z^{2, \ep, \kappa} - ((\tilde \vartheta_{-} - \tilde \vartheta^\tau)S_{-})\cdot \ol \gamma^{\ep, \kappa}.
\end{align*}
We remark that the integral processes on the right-hand side of the equation above exist in $L_2(\whP)$ as a by-product of estimates in \textit{Step 5}-\textit{Step 8} below. The correction term of $A^{\corr}(\ttt, \tau|\ep, \kappa)$ given in \cref{def:approx-correction} satisfies that
\begin{align*}
\sum_{\rho_i(\ep, \kappa) \in [0, \cdot] \cap [0, T)} \(\tilde\vartheta_{\rho_i(\ep, \kappa)-} - \tilde\vartheta^\tau_{\rho_i(\ep, \kappa)}\) \Delta S_{\rho_i(\ep, \kappa)} = ((\tilde \vartheta_{-} - \tilde \vartheta^\tau)S_{-}) \cdot Z^{2, \ep, \kappa}.
\end{align*}
Hence, we arrive at the following decomposition 
\begin{align*}
	E^{\corr}(\tilde \vartheta, \tau|\ep, \kappa) 
	& = \underbrace{((\tilde \vartheta_{-} - \tilde \vartheta^\tau)S_{-} \bar b^Z) \cdot \Leb}_{=:E^{\bar b}(\tilde \vartheta, \tau)} + \underbrace{(\sigma (\tilde \vartheta_{-} - \tilde \vartheta^\tau)S_{-}) \cdot W^{\whP}}_{=:E^{\rmC}(\tilde \vartheta, \tau)} \\
	& \quad + \underbrace{((\tilde \vartheta_{-} - \tilde \vartheta^\tau)S_{-})\cdot \ol Z^{1, \ep, \kappa}}_{=:E^{\rmS}(\tilde \vartheta, \tau|\ep, \kappa)}  - \underbrace{((\tilde \vartheta_{-} - \tilde \vartheta^\tau)S_{-})\cdot \ol \gamma^{\ep, \kappa}}_{=:E^{\rmD}(\tilde \vartheta, \tau|\ep, \kappa)},
\end{align*}
where $E^{\bar b}$ denotes the error involving with the process $\bar b^Z$, $E^{\rmC}$ (resp. $E^{\rmS}$, $E^{\rmD}$) is the error related to the continuous martingale (resp. small jump, drift) part of $Z$. By the triangle inequality,

\begin{align}\label{eq:estimate-error-BMO}
	\big\|E^{\corr}(\tilde \vartheta, \tau|\ep, \kappa) \big\|_{\BMO^\Phi_2(\whP)} \lee \sum_{i \in \{\bar b, \rmC\}} \big\|E^{i}(\tilde \vartheta, \tau) \big\|_{\BMO^\Phi_2(\whP)} + \sum_{i \in \{\rmS, \rmD\}} \big\|E^{i}(\tilde \vartheta, \tau|\ep, \kappa) \big\|_{\BMO^\Phi_2(\whP)}.
\end{align}
Before investigating the right-hand side of \eqref{eq:estimate-error-BMO}, let us introduce a variant for the $\BMO_2^\Phi(\whP)$-norm: for $Y \in \CL_0([0, T])$, we let $Y \in \bmo_2^\Phi(\whP)$ if 
\begin{align*}
	\|Y\|_{\bmo_2^\Phi(\whP)} : = \inf \big\{c \gee  0 : \hce{\F_a}{|Y_T - Y_a|^2} \lee c^2 \Phi_a^2 \quad \mbox{a.s., } \forall a\in [0, T]\big\} <\infty.
\end{align*} 

\smallskip

\textit{\textbf{Step 5:}} \textit{Estimate $\big\|E^{\bar b}(\ttt, \tau)\big\|_{\BMO^\Phi_2(\whP)}$}.  One first observes that $E^{\bar b}(\ttt, \tau)$ is a continuous process  which then implies that $E^{\bar b}_{\rho-}(\ttt, \tau) = E^{\bar b}_{\rho}(\ttt, \tau)$ for any stopping times $\rho$. Hence,
\begin{align*}
	\big\|E^{\bar b}(\ttt, \tau)\big\|_{\BMO^\Phi_2(\whP)} & = \inf\big\{c \gee 0 : \hbce{\F_\rho}{|E^{\bar b}_T(\ttt, \tau) - E^{\bar b}_{\rho}(\ttt, \tau)|^2} \lee c^2  \Phi_\rho^2 \quad\mbox{a.s., } \forall \rho \in \cS([0, T]) \big\}\\
	& = \inf\big\{c \gee 0 : \hbce{\F_a}{|E^{\bar b}_T(\ttt, \tau) - E^{\bar b}_{a}(\ttt, \tau)|^2} \lee c^2 \Phi_a^2 \quad\mbox{a.s., } \forall a \in [0, T] \big\}\\
	& =	\big\|E^{\bar b}(\ttt, \tau)\big\|_{\bmo^\Phi_2(\whP)},
\end{align*}
where in order to get the second equality we may replace the stopping times $\rho$ by the deterministic times $a$ by using a standard approximation argument as for instance in \cite[Proposition A.4]{GN20}. Since $S = 1 + S_{-} \cdot Z$, we write $S = 1 + S^{\mar, \whP} + S^{\fv, \whP}$ as the canonical semimartingale decomposition of $S$ under $\whP$ where $S^{\mar, \whP} := \sigma S_{-} \cdot W^{\whP} + (S_{-}z) \cdot (N_Z - \pi^{\whP}_Z)$ and $S^{\fv, \whP}: = (S_{-} \bar b^Z) \cdot \Leb$. Now, for $a\in [t_{k-1}, t_k)$, $k = 1, \ldots, n$, by setting $s_i : = a \vee t_i$, $i = k-1, \ldots, n$ we obtain
\begin{align}\label{eq:estimate:drift-b}
	& \frac{1}{4}\hbce{\F_a}{|E^{\bar b}_T(\ttt, \tau) - E^{\bar b}_{a}(\ttt, \tau)|^2}  = \frac{1}{4} \hbbce{\F_a}{\bigg|\int_a^T (\ttt_{u-} - \ttt^\tau_u)S_{u-} \bar b^Z_u \od u\bigg|^2} \notag\\
	& =  \frac{1}{4} \hbbce{\F_a}{\bigg|(\ttt_a - \ttt_{t_{k-1}}) \int_a^{t_k} S_{u} \bar b^Z_u \od u + \sum_{i=k}^n \int_{s_{i-1}}^{s_i} \big[(M_u - M_{s_{i-1}})  -  \ttt_{s_{i-1}}(S_u - S_{s_{i-1}}) \big] \bar b^Z_u \od u\bigg|^2} \notag \\
	& \lee  \hbbce{\F_a}{\bigg|(\ttt_a - \ttt_{t_{k-1}}) \int_a^{t_k} S_{u} \bar b^Z_u \od u\bigg|^2}  + \hbbce{\F_a}{\bigg|\sum_{i=k}^n \int_{s_{i-1}}^{s_i}(M_u - M_{s_{i-1}}) \bar b^Z_u \od u\bigg|^2} \notag \\
	& \quad + \hbbce{\F_a}{\bigg|\sum_{i=k}^n \ttt_{s_{i-1}} \int_{s_{i-1}}^{s_i}(S^{\mar, \whP}_u - S^{\mar, \whP}_{s_{i-1}}) \bar b^Z_u \od u\bigg|^2} + \hbbce{\F_a}{\bigg|\sum_{i=k}^n \ttt_{s_{i-1}} \int_{s_{i-1}}^{s_i}(S^{\fv, \whP}_u - S^{\fv, \whP}_{s_{i-1}}) \bar b^Z_u \od u\bigg|^2} \notag \\
	&	=: \rmI^{(1)}_{\eqref{eq:estimate:drift-b}} + \rmI^{(2)}_{\eqref{eq:estimate:drift-b}} + \rmI^{(3)}_{\eqref{eq:estimate:drift-b}} + \rmI^{(4)}_{\eqref{eq:estimate:drift-b}}.
\end{align}
Recall $\kappa = \frac{1- \theta}{2}$ and denote
\begin{align}\label{eq:constant-drift-b}
	B_{\eqref{eq:constant-drift-b}} : = \sup_{i = k,\ldots, n}\;\sup_{u \in (s_{i-1}, s_i]\cap (0, T)} \bigg[\frac{1}{(T-u)^{\kappa}}  \int_{u}^{s_i} |\bar b^Z_r| \od r \bigg].
\end{align}
$\bullet$ For $\rmI^{(1)}_{\eqref{eq:estimate:drift-b}}$, using the non-decreasing property of $\Theta$ we get, a.s., 
\begin{align*}
	\rmI^{(1)}_{\eqref{eq:estimate:drift-b}} & \lee   4 c_{\eqref{eq:estimate-variance-vartheta}}^2 (T-a)^{\theta-1} \Theta_a^2 \, \hbbce{\F_a}{\bigg|\int_a^{t_k} S_u |\bar b^Z_u| \od u\bigg|^2} \notag\\
	& \ts \lee 4 c_{\eqref{eq:estimate-variance-vartheta}}^2 (T-a)^{1- \theta}  B_{\eqref{eq:constant-drift-b}}^2 (T-a)^{1- \theta}\; \hbce{\F_a}{\sup_{u \in (a, t_k]} \Phi_u^2}\\
	& \lee \Big[4 c_{\eqref{eq:estimate-variance-vartheta}}^2 \|\Phi\|_{\cSM_2(\whP)}^2 \Big] B_{\eqref{eq:constant-drift-b}}^2 \Phi_a^2.
\end{align*}
$\bullet$ For $\rmI^{(2)}_{\eqref{eq:estimate:drift-b}}$, we use the orthogonality of martingale increments and notice that $\bar b^Z$ is deterministic to find that the mixed terms in the square expansion vanish under the respective conditional expectation. Then, using the stochastic Fubini theorem and applying the conditional It\^o isometry yield, a.s., 
\begin{align*}
	\rmI^{(2)}_{\eqref{eq:estimate:drift-b}} & = \sum_{i=k}^n \hbbce{\F_a}{\bigg| \int_{s_{i-1}}^{s_i}(M_u - M_{s_{i-1}}) \bar b^Z_u \od u\bigg|^2}   = \sum_{i=k}^n \hbbce{\F_a}{\bigg| \int_{(s_{i-1}, s_i]\cap (0, T)}\bigg( \int_{[u, s_i]} \bar b^Z_r \od r\bigg) \od M_u\bigg|^2} \notag \\
	& = \sum_{i=k}^n \hbbce{\F_a}{ \int_{(s_{i-1}, s_i]\cap (0, T)}\bigg| \int_{[u, s_i]} \bar b^Z_r \od r\bigg|^2 \od \<M\>^{\whP}_u} \lee B_{\eqref{eq:constant-drift-b}}^2 \hbbce{\F_a}{\int_{(a, T)} (T-u)^{1- \theta} \od \<M\>^{\whP}_u} \notag \\
	& \lee  c_{\eqref{item:curvature-condition}}^2  B_{\eqref{eq:constant-drift-b}}^2  \Phi_a^2 .
\end{align*}
$\bullet$ For $\rmI^{(3)}_{\eqref{eq:estimate:drift-b}}$, we also proceed in the same way as for $\rmI^{(2)}_{\eqref{eq:estimate:drift-b}}$ to get, a.s.,
\begin{align*}
	\rmI^{(3)}_{\eqref{eq:estimate:drift-b}} & = \sum_{i=k}^n  \hbbce{\F_a}{\bigg|\ttt_{s_{i-1}} \int_{s_{i-1}}^{s_i}(S^{\mar, \whP}_u - S^{\mar, \whP}_{s_{i-1}}) \bar b^Z_u \od u\bigg|^2} = \sum_{i=k}^n \hbbce{\F_a}{ \ttt_{s_{i-1}}^2 \int_{s_{i-1}}^{s_i}\bigg| \int_u^{s_{i}} \bar b^Z_r \od r\bigg|^2 \od \<S^{\mar, \whP}\>^{\whP}_u}\\
	& \lee B_{\eqref{eq:constant-drift-b}}^2 \sum_{i=k}^n \hbbce{\F_a}{c_{\eqref{eq:estimate-variance-vartheta}}^2 (T-s_{i-1})^{\theta - 1} \Theta_{s_{i-1}}^2 \int_{s_{i-1}}^{s_i} (T-u)^{1- \theta} S_u^2 \bigg(\sigma^2 + \int_{\R_0} z^2 \nu^{\whP}_Z(u, \od z)\bigg)\od u}\\
	& \lee (\sigma^2 + \nu_{\eqref{eq:2nd-moment-hat-mu-Z}})  B_{\eqref{eq:constant-drift-b}}^2 c_{\eqref{eq:estimate-variance-vartheta}}^2  \hbbce{\F_a}{\int_a^T \Theta_u^2 S_u^2 \od u}\\
	&  \lee \Big[(\sigma^2 + \nu_{\eqref{eq:2nd-moment-hat-mu-Z}})   c_{\eqref{eq:estimate-variance-vartheta}}^2 T \|\Phi\|_{\cSM_2(\whP)}^2 \Big] B_{\eqref{eq:constant-drift-b}}^2 \Phi_a^2.
\end{align*}
$\bullet$ For $\rmI^{(4)}_{\eqref{eq:estimate:drift-b}}$, a standard argument using Fubini's theorem yields, a.s.,
\begin{align*}
	\rmI^{(4)}_{\eqref{eq:estimate:drift-b}} & \lee \hbbce{\F_a}{\bigg|\sum_{i=k}^n  c_{\eqref{eq:estimate-variance-vartheta}}  (T-s_{i-1})^{\frac{\theta-1}{2}} \Theta_{s_{i-1}} \int_{s_{i-1}}^{s_i} S_r |\bar b^Z_r| \int_{r}^{s_i} |\bar b^Z_u| \od u \od r\bigg|^2}\\
	& \lee c_{\eqref{eq:estimate-variance-vartheta}}^2 \bigg|\sum_{i=k}^n    (T-s_{i-1})^{\frac{\theta-1}{2}}  \int_{s_{i-1}}^{s_i} |\bar b^Z_r| B_{\eqref{eq:constant-drift-b}} (T-r)^{\frac{1- \theta}{2}} \od r\bigg|^2 \hbce{\F_a}{\ts \sup_{r \in (a, T]} \Phi_r^2}\\
	& \lee \bigg[c_{\eqref{eq:estimate-variance-vartheta}}^2   \bigg|\int_{a}^{T}|\bar b^Z_r| \od r \bigg|^2 \|\Phi\|_{\cSM_2(\whP)}^2 \bigg] B_{\eqref{eq:constant-drift-b}}^2 \Phi_a^2\\
	& \lee \Big[c_{\eqref{eq:estimate-variance-vartheta}}^2 b_{\eqref{eq:drift-Z-hat-P}}^2 T^2 \|\Phi\|_{\cSM_2(\whP)}^2 \Big] B_{\eqref{eq:constant-drift-b}}^2 \Phi_a^2.
\end{align*} 
Plugging those estimates for $\rmI^{(1)}_{\eqref{eq:estimate:drift-b}}$--$\rmI^{(4)}_{\eqref{eq:estimate:drift-b}}$ into \eqref{eq:estimate:drift-b}, we get a constant $c_{\eqref{eq:estimate-BMO-drift-b}}>0$ not depending on $\ep$, $\tau$ such that
\begin{align}\label{eq:estimate-BMO-drift-b}
	\hbce{\F_a}{|E^{\bar b}_T(\ttt, \tau) - E^{\bar b}_{a}(\ttt, \tau)|^2} \lee c_{\eqref{eq:estimate-BMO-drift-b}}^2 B_{\eqref{eq:constant-drift-b}}^2 \Phi_a^2 \quad \mbox{a.s. for all } a\in [0, T].
\end{align}
A simple computation shows $B_{\eqref{eq:constant-drift-b}} \lee T^{\frac{1- \theta}{2}} b_{\eqref{eq:drift-Z-hat-P}} \|\tau\|_\theta$, and hence,
\begin{align}\label{eq:estimate-error-drift-b}
	\big\|E^{\bar b}(\ttt, \tau)\big\|_{\BMO^\Phi_2(\whP)} \lee \big[c_{\eqref{eq:estimate-BMO-drift-b}} T^{\frac{1- \theta}{2}} b_{\eqref{eq:drift-Z-hat-P}} \big] \|\tau\|_\theta.
\end{align}

\smallskip

\textit{\textbf{Step 6:}} \textit{Estimate $\big\|E^{\rmC}(\ttt, \tau)\big\|_{\BMO^\Phi_2(\whP)}$}. By the same reason as in \textit{Step 5}, the pathwise continuity of $E^{\rmC}(\ttt, \tau)$ gives
\begin{align*}
	\big\|E^{\rmC}(\ttt, \tau)\big\|_{\BMO^\Phi_2(\whP)} = \inf\big\{c \gee 0 : \hbce{\F_a}{|E^{\rmC}_T(\ttt, \tau) - E^{\rmC}_{a}(\ttt, \tau)|^2} \lee c^2 \Phi_a^2 \quad\mbox{a.s., } \forall a \in [0, T] \big\}.
\end{align*}
For $a\in [0, T]$, it follows from the conditional It\^o isometry that, a.s.,
\begin{align*}
	& \hbce{\F_a}{|E^{\rmC}_T(\ttt, \tau) - E^{\rmC}_{a}(\ttt, \tau)|^2}  = \sigma^2 \hbbce{\F_a}{\int_a^T |\ttt_u - \ttt^\tau_u|^2S_{u}^2 \od u}  \\
	& = \sigma^2 \hbce{\F_a}{\<\tilde\vartheta, \tau\>_T - \<\tilde\vartheta, \tau\>_a} \lee c_{\eqref{eq:estimate:bmo}}^2 \sigma^2 \|\tau\|_\theta \Phi_a^2.
\end{align*}
Therefore,
\begin{align}\label{eq:estimate-BMO-continuous}
	\big\|E^{\rmC}(\ttt, \tau)\big\|_{\BMO^\Phi_2(\whP)} \lee  c_{\eqref{eq:estimate:bmo}} \sigma \sqrt{\|\tau\|_\theta}.
\end{align}

\smallskip

\textit{\textbf{Step 7:}} \textit{Estimate $\big\|E^{\rmS}(\ttt, \tau|\ep, \kappa)\big\|_{\BMO^\Phi_2(\whP)}$}. Due to \cite[Propositions A.5 and A.4]{GN20}, one has
\begin{align}\label{eq:estimate:error-S}
	\big\|E^{\rmS}(\ttt, \tau|\ep, \kappa)\big\|_{\BMO^\Phi_2(\whP)} \lee \big\|E^{\rmS}(\ttt, \tau|\ep, \kappa)\big\|_{\bmo^\Phi_2(\whP)} + |\Delta E^{\rmS}(\ttt, \tau|\ep, \kappa)|_{\Phi}
\end{align}
where $|\Delta E^{\rmS}(\ttt, \tau|\ep, \kappa)|_{\Phi} : = \inf\{c \gee 0:  |\Delta E^{\rmS}_t(\ttt, \tau|\ep, \kappa)|_{\Phi} \lee c \Phi_t \mbox { for all } t \in [0, T] \mbox{ a.s.}\}$. For the first term on the right-hand side of \eqref{eq:estimate:error-S} and for $a \in [0, T]$, using the conditional It\^o isometry gives, a.s., 
\begin{align*}
	& \hbce{\F_a}{|E^{\rmS}_T(\ttt, \tau|\ep, \kappa) - E^{\rmS}_a(\ttt, \tau|\ep, \kappa)|^2} = \hbbce{\F_a}{\int_a^T |\ttt_u - \ttt^\tau_u|^2S_u^2 \int_{0 < |z| \lee \ep(T-u)^\kappa} z^2 \nu^{\whP}_Z(u, \od z) \od u}\\
	& \lee \bigg\|u \mapsto \int_{0 < |z| \lee \ep T^\kappa} z^2 \nu_Z^{\whP} (u, \od z) \bigg\|_{L_\infty([0, T], \Leb)}   \hbce{\F_a}{\<\tilde\vartheta, \tau\>_T - \<\tilde\vartheta, \tau\>_a}.
\end{align*} 
By an argument using Fubini's theorem as in \eqref{eq:small-jump-estimate-around-0}, we obtain
\begin{align}
	\bigg\|u \mapsto \int_{0 < |z| \lee \ep T^\kappa} z^2 \nu_Z^{\whP} (u, \od z) \bigg\|_{L_\infty([0, T], \Leb)}  & \lee \nu_{\eqref{eq:2nd-moment-hat-mu-Z}} \1_{\{\ep T^\kappa >1\}} + \frac{2 c_{\eqref{eq:small-jump-conditino}}(\alpha)}{2- \alpha}  \ep^{2-\alpha} T^{\kappa(2- \alpha)} \1_{\{0<\ep T^\kappa \lee 1\}} \notag \\
	& \lee c_{\eqref{eq:rate-step-7}} (\ep \wedge 1)^{2- \alpha} \label{eq:rate-step-7}
\end{align}
for some constant $c_{\eqref{eq:rate-step-7}} >0$ depending at most on $\nu_{\eqref{eq:2nd-moment-hat-mu-Z}},  c_{\eqref{eq:small-jump-conditino}}(\alpha), \alpha, T, \kappa$. Thus,
\begin{align*}
	\big\|E^{\rmS}(\ttt, \tau|\ep, \kappa)\big\|_{\bmo^\Phi_2(\whP)} \lee c_{\eqref{eq:estimate:bmo}} \sqrt{c_{\eqref{eq:rate-step-7}}} (\ep\wedge 1)^{1- \frac{\alpha}{2}} \sqrt{\|\tau\|_\theta}.
\end{align*}
Moreover, due to \eqref{eq:path-variation} we can find an $\Omega_0$ with probability one such that on $\Omega_0$ one has
\begin{align*}
	\forall t \in [0, T] :	|\Delta E^{\rmS}_t(\ttt, \tau|\ep, \kappa)| & = |(\ttt_{t-} - \ttt^\tau_t)S_{t-}|  |\Delta \ol Z^{1, \ep, \kappa}_t| \lee  2 c_{\eqref{eq:estimate-variance-vartheta}}(T-t)^{\frac{\theta-1}{2}} \Theta_{t-} S_{t-} |\Delta \ol Z^{1, \ep, \kappa}_t|\\
	& \lee 2 c_{\eqref{eq:estimate-variance-vartheta}}(T-t)^{\frac{\theta-1}{2}} \Theta_{t-} S_{t-} \ep (T-t)^{\frac{1- \theta}{2}} \lee  2 c_{\eqref{eq:estimate-variance-vartheta}} \ep \Phi_t.
\end{align*}
Hence, $|\Delta E^{\rmS}(\ttt, \tau|\ep, \kappa)|_{\Phi} \lee 2 c_{\eqref{eq:estimate-variance-vartheta}} \ep$. Plugging those estimates into \eqref{eq:estimate:error-S} yields
\begin{align}\label{eq:estimate-small-jump-BMO}
	\big\|E^{\rmS}(\ttt, \tau|\ep, \kappa)\big\|_{\BMO^\Phi_2(\whP)} \lee  c_{\eqref{eq:estimate-small-jump-BMO}} (\ep + (\ep \wedge 1)^{1- \frac{\alpha}{2}} \sqrt{\|\tau\|_\theta})
\end{align}  
for $c_{\eqref{eq:estimate-small-jump-BMO}} : = \max\{c_{\eqref{eq:estimate:bmo}} \sqrt{c_{\eqref{eq:rate-step-7}}}, 2 c_{\eqref{eq:estimate-variance-vartheta}}\}$.

\smallskip

\textit{\textbf{Step 8:}} \textit{Estimate $\big\|E^{\rmD}(\ttt, \tau|\ep, \kappa)\big\|_{\BMO^\Phi_2(\whP)}$}. We first define
\begin{align*}
	D_{\ep, \kappa}(u) := \int_{|z| > \ep (T-u)^\kappa} z \nu^{\whP}_Z(u, \od z),\quad u \in [0, T).
\end{align*}
Then, by the same estimation as given in \cite[Proof of Theorem 3.15, Step 1]{Th20} we get a constant $c_{\eqref{eq:estimate-constant-drift}}>0$ independent of $\tau = (t_i)_{i=0}^n \in \cT_{\det}$ and $\ep>0$ such that
\begin{align}\label{eq:estimate-constant-drift}
	D_{\eqref{eq:estimate-constant-drift}} : & = \sup_{i = 1, \ldots, n} \; \sup_{u \in (t_{i-1}, t_i] \cap [0, T)} \bigg[ \frac{1}{(T-u)^\kappa} \int_u^{t_i} |D_{\ep, \kappa}(r)| \od r\bigg] \notag\\
	& \lee c_{\eqref{eq:estimate-constant-drift}} \begin{cases}
		\|\tau\|_\theta + \ep^{1-\alpha} \|\tau\|_\theta^{1- \kappa(\alpha -1)} & \mbox{if } \eqref{eq:small-jump-conditino} \mbox{ holds with } \alpha \in (1, 2]\\
		\big[1 + \log^+(\frac{1}{\ep}) + \log^+(\frac{1}{\|\tau\|_\theta})\big] \, \|\tau\|_\theta & \mbox{if } \eqref{eq:small-jump-conditino} \mbox{ holds with } \alpha =1\\
		\|\tau\|_\theta & \mbox{if } \eqref{eq:small-jump-conditino} \mbox{ holds with } \alpha \in (0, 1).
	\end{cases}
\end{align}
We now estimate $\big\|E^{\rmD}(\ttt, \tau|\ep, \kappa)\big\|_{\BMO^\Phi_2(\whP)}$ analogously as for $\big\|E^{\bar b}(\ttt, \tau)\big\|_{\BMO^\Phi_2(\whP)}$ in Step 5, where $D_{\ep, \kappa}(u)$ and $D_{\eqref{eq:estimate-constant-drift}}$ play the role of $\bar b^Z_u$ and $B_{\eqref{eq:constant-drift-b}}$ respectively. This then leads us to
\begin{align*}
	\hbce{\F_a}{|E^{\rmD}_T(\ttt, \tau|\ep, \kappa) - E^{\rmD}_a(\ttt, \tau|\ep, \kappa)|^2} \lee c_{\eqref{eq:estimate-BMO-drift-b}}^2 D_{\eqref{eq:estimate-constant-drift}}^2 \Phi_a^2 \quad \mbox{a.s. for all } a\in [0, T]
\end{align*}
which means that
\begin{align}\label{eq:estimate-BMO-drift-D}
	\big\|E^{\rmD}(\ttt, \tau|\ep, \kappa)\big\|_{\BMO^\Phi_2(\whP)} = \big\|E^{\rmD}(\ttt, \tau|\ep, \kappa)\big\|_{\bmo^\Phi_2(\whP)} \lee c_{\eqref{eq:estimate-BMO-drift-b}} D_{\eqref{eq:estimate-constant-drift}}.
\end{align}
Combining \eqref{eq:estimate-error-drift-b}, \eqref{eq:estimate-BMO-continuous}, \eqref{eq:estimate-small-jump-BMO} and \eqref{eq:estimate-BMO-drift-D} with \eqref{eq:estimate-error-BMO} yields \eqref{eq:thm:rate-whP}.\qed

\section{Some technical results}\label{appen:technical-results}

\subsection{Some properties of stable-like-measures} \label{subsec:property-stable} 
We recall $\sUS(\alpha), \sS(\alpha)$ from \cref{defi:holder-stable}. 

\begin{lemm}[See also \cite{Th20}, Remark 4.5]\label{lemm:property-stable}  Let $\ell$ be a L\'evy measure.
	\begin{enumerate}[\rm (1)]
\item \label{item:equivalent-upper-stable} For $\alpha \in (0, 2]$ one has
\begin{align*}
	 \underbrace{\sup_{r \in (0,1)} r^\alpha \int_{\frac{r}{2} < |x| \lee r} \ell(\od x) <\infty}_{=: A(\alpha)} \Leftrightarrow \ell \in \sUS(\alpha) \Leftrightarrow 
	 \underbrace{\limsup_{|u| \to \infty} \frac{1}{|u|^{\alpha}} \int_{\R} (1-\cos (ux)) \ell
	 (\od x) <\infty}_{=: B(\alpha)}. 
\end{align*}
		
		\item \label{GB-index} If $\ell \in \sS(\alpha)$ for some $\alpha \in (0, 2)$, then $\alpha$ is equal to the Blumenthal--Getoor index of $\ell$, i.e.,  $\alpha = \inf \{ q \in [0, 2] : \int_{|x| \lee 1} |x|^q \ell(\od x) <\infty\}$.
		
		\item \label{item:sufficient-S1} If $\ell$ has a density $p(x) : = \ell(\od x)/\od x$ which satisfies
		\begin{align*}
		\liminf_{x \to 0-} (-x)^{1+\alpha} p(x) +  \liminf_{x \to 0+} x^{1+\alpha} p(x) >0 \quad \mbox{and} \quad  \limsup_{|x|\to 0} |x|^{1+\alpha} p(x) <\infty
		\end{align*}
	for some $\alpha \in (0, 2)$,		then $\ell \in \sS(\alpha)$.
	\end{enumerate}
\end{lemm}

\begin{proof}
	 \eqref{item:equivalent-upper-stable} We show $A(\alpha) \Rightarrow \ell \in \sUS(\alpha) \Rightarrow B(\alpha) \Rightarrow A(\alpha)$. Denote
	 \begin{align*}
	 	c_1(\alpha): = \sup_{r \in (0, 1)} r^\alpha \int_{r < |x| \lee 1} \ell(\od x) \quad \mbox{and} \quad  c_2(\alpha): = \sup_{r \in (0,1)} r^\alpha \int_{\frac{r}{2} < |x| \lee r} \ell(\od x).
	 \end{align*}
 $\bullet$ $A(\alpha) \Rightarrow \ell \in \sUS(\alpha)$: Let $r \in (0,1)$ and $m_r \in \bN \cup\{0\}$ with $2^{-m_r - 1} < r \lee 2^{-m_r}$. Then
 \begin{align*}
 	r^\alpha \int_{r < |x| \lee 1} \ell(\od x) & \lee 2^{- \alpha m_r} \int_{2^{-m_r - 1} < |x| \lee 1} \ell(\od x) = 2^{- \alpha m_r} \sum_{k=0}^{m_r} \int_{2^{-k - 1} < |x| \lee 2^{-k}} \ell(\od x)\\
 	& \lee c_2(\alpha) 2^{- \alpha m_r} \sum_{k=0}^{m_r} 2^{\alpha k} \lee c_2(\alpha) 2^{- \alpha m_r} \frac{2^{\alpha (m_r+1)}}{2^\alpha -1} = c_2(\alpha) \frac{2^\alpha}{2^\alpha -1}.
 \end{align*}
  $\bullet$ $\ell \in \sUS(\alpha) \Rightarrow B(\alpha)$: For $|u| > 1$, one has
  \begin{align*}
  	\frac{1}{|u|^\alpha} \int_{\R} (1- \cos(ux) \ell(\od x) & = \frac{1}{|u|^\alpha} \bigg(\int_{1 < |ux| \lee |u|} + \int_{|ux| \lee 1} + \int_{|ux| >|u|}\bigg) (1- \cos(ux)) \ell(\od x)\\
  	& \lee \frac{2}{|u|^\alpha} \int_{\frac{1}{|u|} < |x| \lee 1} \ell(\od x) + |u|^{2- \alpha} \int_{|x| \lee \frac{1}{|u|}} x^2 \ell(\od x) + \frac{2}{|u|^\alpha} \int_{|x|>1} \ell(\od x)\\
  	& \lee 2 c_1(\alpha) + |u|^{2- \alpha} \int_{|x| \lee \frac{1}{|u|}} x^2 \ell(\od x) + \frac{2}{|u|^\alpha} \int_{|x|>1} \ell(\od x).
  \end{align*}
If $\alpha =2$, then it is obvious from the estimate above that $\ell \in \sUS(2) \Rightarrow B(2)$. For $\alpha \in (0, 2)$, we use Fubini's theorem to get that, for $r \in (0, 1]$,
\begin{align}\label{eq:small-jump-estimate-around-0}
	\int_{|x| \lee r} x^2 \ell(\od x) & = \int_{|x| \lee r}\int_0^{x^2} \od y \,                                                                                                              \ell(\od x) = \int_0^{r^2} \int_{\sqrt{y} <|x| \lee r} \ell(\od x) \od y \notag \\
& \lee c_1(\alpha) \int_0^{r^2} y^{-\frac{\alpha}{2}} \od y = \frac{2c_1(\alpha)}{2 - \alpha} r^{2- \alpha}.
\end{align}
Hence, 
\begin{align*}
	\limsup_{|u| \to \infty} \frac{1}{|u|^\alpha} \int_{\R} (1- \cos(ux)) \ell(\od x) \lee 2 c_1(\alpha) + \frac{2 c_1(\alpha)}{2- \alpha}  <\infty.
\end{align*}
$\bullet$ $B(\alpha) \Rightarrow A(\alpha)$: Using the inequality
$1- \cos y \gee \frac{y^2}{3}$, $y \in [0, 2]$, we infer that, for $|u| \gee 2 $,
\begin{align*}
	& \frac{1}{|u|^\alpha}\int_{\R} (1 - \cos(ux))\ell(\od x) \gee \frac{1}{|u|^\alpha} \int_{2\gee |ux| >1} (1- \cos (ux)) \ell(\od x) \\
	& \gee \frac{|u|^{2-\alpha}}{3} \int_{\frac{2}{|u|} \gee |x| > \frac{1}{|u|}} x^2 \ell(\od x)
	\gee \frac{1}{3|u|^\alpha} \int_{\frac{2}{|u|} \gee |x| > \frac{1}{|u|}} \ell(\od x).
\end{align*}
Since $\limsup_{|u| \to \infty} \frac{1}{|u|^\alpha} \int_{\R} (1 - \cos(ux)) \ell(\od x) <\infty$, it implies that 
\begin{align*}
	\sup_{|u| \gee 2} \frac{1}{|u|^\alpha} \int_{\frac{2}{|u|} \gee |x| > \frac{1}{|u|}} \ell(\od x) \lee \sup_{|u| \gee 2} \frac{3}{|u|^\alpha} \int_{\R} (1 - \cos(ux)) \ell(\od x) <\infty.
\end{align*}
	 
	 
	\eqref{GB-index} follows from \cite[Theorem 3.2]{BG61}.
	 
	 \medskip
	 \eqref{item:sufficient-S1}  By assumption, there exist constants $0< c_1, c_2 \lee C$ and $\ep >0$ such that
	 \begin{align*}
	 c_1(-x)^{-1 - \alpha}\1_{\{-\ep \lee x <0\}} + c_2  x^{-1- \alpha} \1_{\{0 < x \lee \ep\}} \lee p(x)  \lee C |x|^{-1- \alpha},\quad \forall 0< |x| \lee \ep.
	 \end{align*}
	 We set 
	 \begin{align*}
	 \ell_1(\od x) : = (c_1(-x)^{-1 - \alpha}\1_{\{-\ep \lee x <0\}} + c_2  x^{-1- \alpha} \1_{\{0 < x < \ep\}}) \od x \quad \mbox{and} \quad \ell_2(\od x) : = \ell(\od x) - \ell_1(\od x).
	 \end{align*}
	 Then, $\ell_1$ satisfies \eqref{eq:item:alpha-stable} with $k(x) = c_1 \1_{\{-\ep \lee x <0\}} + c_2 \1_{\{0 \lee x < \ep\}}$. For $\ell_2$, we have
	 \begin{align*}
	 \int_{\R} (1- \cos(ux)) \ell_2(\od x) & \lee (C-(c_1 \wedge c_2))\int_{0<|x| < \ep} \frac{1- \cos(ux)}{|x|^{1+\alpha}} \od x + 2 \int_{|x|\gee \ep}  \ell(\od x)\\
	 & = 2(C-(c_1 \wedge c_2))|u|^\alpha \int_0^{|u|\ep} \frac{1- \cos x}{x^{1+\alpha}} \od x + 2 \int_{|x|\gee \ep}  \ell(\od x)
	 \end{align*}
	from which we derive that $\limsup_{|u| \to \infty} \frac{1}{|u|^\alpha} \int_{\R} (1- \cos(ux)) \ell_2(\od x) <\infty$. By Item \eqref{item:equivalent-upper-stable}, it holds that $\ell_2 \in \sUS(\alpha)$, and hence, $\ell \in \sS(\alpha)$.
\end{proof}

\subsection{Regularity of weight processes}\label{secsec:apendix-weight-regularity} Let $(\Omega, \F, \bQ, (\F_t)_{t \in [0, T]})$ be a filtered probability space, $T \in (0, \infty)$. For $\Psi \in \CL^+([0, T])$, we define $\Phi \in \CL^+([0, T])$ by setting 
\begin{align*}
\ts \Phi_t : = \Psi_t + \sup_{u \in [0, t]} |\Delta \Psi_u|, \quad t \in [0, T].
\end{align*}
It is clear that $\Psi \vee \Psi_{-} \lee \Phi$, and moreover, $\Psi \equiv \Phi$ if and only if $\Psi$ is continuous.

 \begin{lemm}[\cite{Th20}, Proposition A.1] \label{lemm:olderline-Phi} If $\Psi \in \cSM_q(\bQ)$ for some $q\in (0, \infty)$, then $\Phi \in \cSM_q(\bQ)$.
 \end{lemm}
 
Assume that $X = (X_t)_{t \in [0, T]}$ is a L\'evy process with $(X|\bQ) \sim (\gamma^\bQ, \sigma^\bQ, \nu^\bQ)$.  For  $\eta \in [0, 1]$, we define $\Psi(\eta) \in \CL^+([0, T])$ by
 \begin{align*}
 \ts \Psi(\eta)_t = \e^{X_t} \sup_{u \in [0, t]} \e^{(\eta -1) X_u},\quad t\in [0, T].
 \end{align*}

 \begin{lemm}[\cite{Th20}, Proposition A.2] \label{prop:regularity-weight-process} If $\int_{|x|>1} \e^{qx} \nu^\bQ(\od x) <\infty$ for some $q \in (1, \infty)$, then $\Psi(\eta) \in \cSM_q(\bQ)$  for all $\eta \in [0, 1]$.
 \end{lemm}

\subsection{Gradient type estimates for a L\'evy semigroup on H\"older spaces} In this part we recall $C^{0, \eta}$ and $\sS(\alpha)$ from \cref{defi:holder-stable}.

Assume that $X = (X_t)_{t \gee 0}$ is a L\'evy process with respect to a probability measure $\bQ$ with $(X|\bQ) \sim (\gamma^\bQ, \sigma^\bQ, \nu^\bQ)$. 
Let $\eta \in [0, 1]$, $g \in C^{0, \eta}$, and assume that $\int_{|x|>1} \e^{\eta x} \nu^{\bQ}(\od x) <\infty$. Then, for all $t\gee 0$ and $y, z >0$ one has $\E^{\bQ}|g(y\e^{X_t})| <\infty$ and 
$$|\E^{\bQ} g(y\e^{X_t}) - \E^{\bQ}g(z\e^{X_t})| \lee |g|_{C^{0, \eta}} \E^{\bQ} \e^{\eta X_t} |y-z|^\eta.$$ This leads us to define the map $Q_t \colon C^{0, \eta} \to C^{0, \eta}$ by setting
	\begin{align*}
	Q_t g(y) : = \E^\bQ g(y \e^{X_t}), \quad y >0, t \gee 0.
	\end{align*}
	It is clear that $Q_{t+s} = Q_t \circ Q_s$ for all $s,t \gee 0$ which means that $(Q_t)_{t\gee 0}$ is a semigroup on $C^{0, \eta}$.

For a L\'evy measure $\ell$ on $\cB(\R)$ and a Borel function $g$, we write symbolically 
\begin{align*}
\varGamma^\bQ_\ell(t, y) := |\sigma^\bQ|^2 \pd_y Q_{t}g(y) + \int_{ \R} \frac{Q_{t} g(\e^x y) - Q_{t} g(y)}{y}(\e^x -1)\ell(\od x), \quad y >0, t\gee 0,
\end{align*}
where we set $\pd_y Q_t g : =0$ if $\sigma^\bQ=0$.

\begin{prop}[\cite{Th20}, Proposition B.5]\label{theo:envelop-MVH-strategy} Let $g\in C^{0, \eta}$ with $\eta \in [0, 1]$. Assume that $\ell$ is a L\'evy measure with $\int_{|x|>1} \e^{(\eta +1)x} \ell(\od x) <\infty$. Then, for any $T\in (0, \infty)$ there is a  $c_{\eqref{eq:thm:Holder-case-estimate:psi-process}}>0$ such that
	\begin{align}\label{eq:thm:Holder-case-estimate:psi-process}
		|\varGamma_\ell^{\bQ}(t, y)| \lee c_{\eqref{eq:thm:Holder-case-estimate:psi-process}}  V(t) y^{\eta -1}, \quad \forall (t, y) \in (0, T]\times (0, \infty),
	\end{align}
	where the cases for $V(t)$ are provided as follows:
	\begin{enumerate}[\rm (1)]

		\item \label{item:1:sigma>0-envelope} If $\sigma^{\bQ} >0$ and $\int_{|x|>1} \e^{2x} \nu^{\bQ} (\od x) <\infty$, then $V(t) = t^{\frac{\eta -1}{2}}$.
		
		\item \label{item:2:sigma=0-envelope} If $\sigma^{\bQ} =0$, $\int_{|x|>1}\e^{\eta x} \nu^{\bQ}(\od x) <\infty$ and $\int_{|x| \lee 1} |x|^{\eta +1} \ell(\od x) <\infty$, then $V(t)=1$.
		
		\item \label{item:2.2:sigma=0-envelope} If $\sigma^{\bQ} =0$, $\eta \in [0, 1)$ and if the following two conditions hold:
		\begin{enumerate}[\rm (a)]
			\item \label{item:3a} $\nu^{\bQ} \in \sS(\alpha)$ for some $\alpha \in (0, 2)$ and $\int_{|x|>1} \e^x \nu^{\bQ} (\od x) <\infty$,
			\item \label{item:3b} there is a $\beta \in [0, 2]$ such that 
			\begin{align}\label{eq:small-ball-condition-ell}
				\ts 0 < c_{\eqref{eq:small-ball-condition-ell}} : = \sup_{r \in (0, 1)} r^\beta \int_{r < |x| \lee 1} \ell(\od x) < \infty,
			\end{align}
		\end{enumerate}
		then 
		\begin{align*}
			V(t) = \begin{cases}
				t^{\frac{\eta +1 -\beta}{\alpha}} & \mbox{ if } \beta \in (1 + \eta, 2]\\
				\max\{1, \log(1/t)\} & \mbox{ if } \beta = 1+ \eta\\
				1 & \mbox{ if } \beta \in [0, 1+ \eta).
			\end{cases}
		\end{align*}
	\end{enumerate}
	Here, the constant $c_{\eqref{eq:thm:Holder-case-estimate:psi-process}}$ may depend on $\beta$ in Item \eqref{item:2.2:sigma=0-envelope}.
\end{prop}
One remarks that the definition of $\sS(\alpha)$ in \cref{theo:envelop-MVH-strategy} is slightly more general than that in \cite[Proposition B.5]{Th20}, however, the proof remains the same.

\subsection*{Acknowledgment} A major part of this work was done when the author was affiliated to the Department of Mathematics and Statistics, University of Jyv\"askyl\"a, Finland. The author is very grateful to Christel Geiss and Stefan Geiss for helpful discussions.

\bibliographystyle{amsplain}

\begin{thebibliography}{100}
	
	
	\bibitem{Ap09} D. Applebaum, \textit{L\'evy processes and stochastic calculus (2nd ed.)}, Cambridge University Press, Cambridge, 2009.
	
	
	
	\bibitem{AS15} T. Arai, R. Suzuki, Local risk-minimization for L\'evy markets, \textit{Int. J. Financ. Eng.} 2(2), 2015, 1550015.
	
	
	
	\bibitem{BNLOP03} F. Benth, G. Di Nunno, A. L\o{}kka, B. \O{}ksendal, F. Proske, Explicit representation of the minimal variance portfolio in markets driven by L\'evy processes, \textit{Math. Finance} 13(1), 2003, 55--72.

	\bibitem{BG61} R. Blumenthal, R. Getoor, Sample functions of stochastic processes with stationary independent increments, \textit{J. Math. Mech.} 10, 1961, 493--516.
	
	\bibitem{BT11} M. Brod\'en, P. Tankov, Tracking errors from discrete hedging in exponential L\'evy model, \textit{Int. J. Theor. Appl. Finance} 14(6), 2011, 803--837.
	
	
	\bibitem{CVV10} T. Choulli, N. Vandaele, M. Vanmaele, The F\"ollmer--Schweizer decomposition. Comparison and description, \textit{Stochastic Process. Appl.} 120, 2010, 853--872.
	
	\bibitem{CKS98} T. Choulli, L. Krawczyk, C. Stricker, $\mathcal E$-martingales and their applications in mathematical finance, \textit{Ann. Probab.} 26(2), 1998, 853--876.
	
	
	
	\bibitem{CTV05} R. Cont, P. Tankov, E. Voltchkova, \textit{Hedging with options in models with jumps}, In: Stochastic Analysis  and Applications -- The Abel Symposium 2005, Springer, Berlin, 2007.
	
	
	\bibitem{ES05} F. Esche, M. Schweizer, Minimal entropy preserves the L\'evy property:
	how and why, \textit{Stochastic Process. Appl.} 115, 2005, 299--327.
	
	
	\bibitem{Fu11} M. Fukasawa, Asymptotically efficient discrete hedging. In: \textit{Stochastic Analysis with Financial Applications. Progress in Probability} 65,  331--346, Springer, Basel, 2011.
	
	\bibitem{GL11} C. Geiss, E. Laukkarinen, Denseness of certain smooth L\'evy functionals in $\mathbb D_{1,2}$, \textit{Probab. Math. Statist.} 31, 2011, 1--15.
	
	
	\bibitem{GGL13} C. Geiss, S. Geiss, E. Laukkarinen, A note on Malliavin fractional smoothness for L\'evy processes and approximation, \textit{Potential Anal.} 39, 2013, 203--230.
	
	\bibitem{GS16} C. Geiss, A. Steinicke, Malliavin derivative of random functions and applications to L\'evy driven BSDEs, \textit{Electron. J. Probab.} 21(10), 2016, 28 pp.
	
	\bibitem{Ge02} S. Geiss, Quantitative approximation of certain stochastic integrals, \textit{Stochastics Stochastics Rep}. 73(3--4), 2002, 241--270.
	
	
	\bibitem{Ge05} S. Geiss, Weighted BMO and discrete time hedging within the Black-Scholes model, \textit{Probab. Theory Related Fields} 132, 2005, 13--38.
	
	
	\bibitem{GN20} S. Geiss, N. T. Thuan, On  Riemann--Liouville operators, BMO, gradient estimates in the L\'evy--It\^o space, and approximation, 2021. arXiv:2009.00899v2.
	
	\bibitem{GT15} S. Geiss, A. Toivola, On fractional smoothness and $L_p$-approximation on the Gaussian space, \textit{Ann. Probab.} 43(2), 2015, 605--638.

	
	\bibitem{GOR14} S. Goutte, N. Oudjane, F. Russo, Variance optimal hedging for
	continuous time additive processes and
	applications, \textit{Stochastics} 86(1), 2014, 147--185.
	
	
	\bibitem{HKK06} F. Hubalek, J. Kallsen, L. Krawczyk, Variance-optimal hedging for processes with stationary independent increments, \textit{Ann. Appl. Probab.} 16(2), 2006, 853--885.
	
	
	\bibitem{KL05} A. Kyprianou, R. Loeffen, L\'evy processes in finance distinguished by their coarse and fine path properties, In: \textit{Exotic option pricing and advanced Lévy models}, Wiley, 2005.
	
	\bibitem{It56} K. It\^o, Spectral type of the shift transformation of differential processes with stationary increments, \textit{Trans. Amer. Math. Soc.} 81, 1956, 253--263.
	
	\bibitem{JMP00} J. Jacob, S. M\'el\'eard, P. Protter, Explicit form and robustness of martingale representations, \textit{Ann. Probab.} 28(4), 2000, 1747--1780.
	
	\bibitem{KP10} J. Kallsen, A. Pauwels, Variance-optimal hedging in general affine stochastic volatility models, \textit{Adv. Appl. Prob.} 42, 2010, 83--105.
	

	\bibitem{La13} E. Laukkarinen, \textit{On Malliavin calculus and approximation of stochastic integrals for L\'evy process}, PhD thesis, University of Jyv\"askyl\"a, 2013.
	
	
	\bibitem{La20} E. Laukkarinen, Malliavin smoothness on the L\'evy space with H\"older continuous or $BV$ functionals, \textit{Stochastic Process. Appl.} 130(8), 2020, 4766--4792.
	
	\bibitem{Lo04} A. L\o{}kka, Martingale representation of functionals of L\'evy processes, \textit{Stoch. Anal. Appl.} 22(4), 2004, 867--892.
	
	\bibitem{MPZ01} J. Ma, P. Protter, J. Zhang, Explicit form and path regularity of martingale representations, In: \textit{L\'evy processes: Theory and applications}, Springer, New York, 2001.
	
	
	\bibitem{MAKL95} P. Malliavin, H. Airault, L. Kay, G. Letac, \textit{Integration and probability}, Springer, New York, 1995.  
	
	\bibitem{MS95} P. Monat, C. Stricker, F\"ollmer--Schweizer decomposition and mean-variance hedging for general claims, \textit{Ann. Probab.} 23(2), 1995, 605--628.
	
	
	
	\bibitem{NN18} D. Nualart, E. Nualart, \textit{Introduction to Malliavin calculus}, Cambridge University Press, Cambridge, 2018.
	
	

	
	\bibitem{NOP09} G. Di Nunno, B. \O{}ksendal, F. Proske, \textit{Malliavin calculus for L\'evy processes with applications to finance}, Springer-Verlag Berlin Heidelberg, 2009.
	
	\bibitem{Pe08} E. Petrou, Malliavin calculus in L\'evy spaces and applications to finance, \textit{Electron. J. Probab.} 13, 2008, 852--879.
	

	
	\bibitem{Pr05} P. Protter, \textit{Stochastic integration and differential equations (2nd ed., ver. 2.1)}, Springer-Verlag, Berlin, 2005.
	

	
	\bibitem{RT14} M. Rosenbaum, P. Tankov, Asymptotically optimal discretization of hedging strategies with jumps, \textit{Ann. Appl. Probab.} 24(3), 2014, 1002--1048.
	

	
	\bibitem{Sa13} K. Sato, \textit{L\'evy processes and infinitely divisible distributions (2nd ed.)}, Cambridge University Press, Cambridge, 2013.
	
	
	\bibitem{Sc95} M. Schweizer, On the minimal martingale measure and the F\"ollmer--Schweizer
	decomposition, \textit{Stoch. Anal. Appl.}, 13(5), 1995, 573--599.
	
	\bibitem{Sc01} M. Schweizer, A guided tour through quadratic hedging approaches, In: \textit{Option pricing, interest rates and risk management},  Cambridge Univ. Press, Cambridge, 2001.
	

	\bibitem{SUV07b} J. Sol\'e, F. Utzet, J. Vives, \textit{Chaos expansion and Malliavin calculus for L\'evy processes},
	In: Stochastic Analysis and Applications -- The Abel Symposium 2005, Springer, Berlin, 2007.
	

	
	\bibitem{Ta10} P. Tankov, 	\textit{Pricing and hedging in exponential L\'evy models: review of recent results}, Paris-Princeton Lecture Notes in Mathematical Finance, Springer, 2010.
	

	
	\bibitem{Th20} N. T. Thuan, Approximation of stochastic integrals with jumps via weighted BMO approach, 2021. arXiv:2009.02116v3.
		
\end{thebibliography}

\end{document}